\documentclass[a4paper, 10pt]{article}\date{}\textheight=28.1cm \voffset=-4cm \textwidth=18.6cm \hoffset=-3.25cm
\title{Asymptotic Analysis of Boundary Layer Solutions to Poisson--Boltzmann Type Equations in General Bounded Smooth Domains}
\usepackage[bookmarks,bookmarksnumbered]{hyperref}
\hypersetup{colorlinks = true,linkcolor = blue,anchorcolor =red,citecolor = blue,filecolor = red,urlcolor = red, pdfauthor=author}
\usepackage{amsthm,amsmath,amssymb,mathrsfs,color,mathtools,subcaption,cleveref,graphicx}

\allowdisplaybreaks[1]

\def\_email#1@#2\q_nil{\href{mailto:#1@#2}{{\emailfont #1\emailampersat #2}}}
\newcommand\emailampersat{{\color{cyan}\small@}}

\author{Jhih-Hong Lyu \thanks{Department of Mathematics, National Taiwan University, Taipei 10617, Taiwan (\tt d06221001@ntu.edu.tw).}, Tai-Chia Lin \thanks{Department of Mathematics, National Taiwan University, Taipei 10617, Taiwan; National Center for Theoretical Sciences, Mathematics Division, Taipei 10617, Taiwan ({\tt tclin@math.ntu.edu.tw}).}}

\newtheorem{theorem}{Theorem}
\newtheorem{remark}{Remark}
\newtheorem{proposition}{Proposition}[section]
\newtheorem{corollary}{Corollary}
\newtheorem{lemma}[proposition]{Lemma}
\newtheorem{claim}{Claim}
\numberwithin{equation}{section}

\Crefname{claim}{Claim}{Claims}
\Crefname{corollary}{Corollary}{Corollaries}
\newcommand{\dist}{\operatorname{dist}}

\newcommand{\sgn}{\operatorname{sgn}}
\newcommand{\dd}{\displaystyle}
\newcommand{\R}{\mathbb{R}}
\newcommand{\C}{\mathcal{C}}
\newcommand{\N}{\mathbb{N}}
\newcommand{\W}{\mathcal{W}}

\begin{document}
\maketitle
\begin{abstract}
We study the boundary layer solution to singular perturbation problems involving Poisson--Boltzmann (PB) type equations with a small parameter $\epsilon$ in general bounded smooth domains (including multiply connected domains) under the Robin boundary condition.
The PB type equations include the classical PB, modified PB and charge-conserving PB (CCPB) equations, which are mathematical models for the electric potential and ion distributions.
The CCPB equations present particular analytical challenges due to their nonlocal nonlinearity introduced through integral terms enforcing charge conservation.
Using the principal coordinate system, exponential-type estimates and the moving plane agruments, we rigorously prove asymptotic expansions of boundary layer solutions throughout the whole domain.
The solution domain is partitioned into three characteristic regions based on the distance from the boundary:
\begin{itemize}
\item Region I, where the distance from the boundary is at most $T\sqrt\epsilon$,
\item Region II, where the distance ranges between $T\sqrt\epsilon$ and $\epsilon^\beta$,
\item Region III, where the distance is at least $\epsilon^\beta$,
\end{itemize}
for given parameters $T>0$ and $0<\beta<1/2$.
In Region I, we derive second-order asymptotic formulas explicitly incorporating the effects of boundary mean curvature, while exponential decay estimates are established for Regions II and III.
Furthermore, we obtain asymptotic expansions for key physical quantities, including the electric potential, electric field, total ionic charge density and total ionic charge, revealing how domain geometry regulates electrostatic interactions.
\vspace{5mm}
\vspace{1mm}
\end{abstract}

\section{Introduction}
When a charged surface (e.g., electrode, membrane, colloid) is in contact with an electrolyte, a structured layer of charge known as the electric double layer (EDL) forms.
The EDL plays a crucial role in various physical, engineering, and biological systems (cf. \cite{2022dourado,1999evans,2021petsev}).
It typically consists of two distinct regions:
\begin{itemize}
\item Stern Layer --- The region closest to the surface, where ions are strongly adsorbed, governed by the surface charge density;
\item Diffuse Layer --- A region extending into the electrolyte, where ion distributions are influenced by electrostatic interactions and thermal fluctuations.
\end{itemize}
The behavior of the diffuse layer can be described by different types of Poisson--Boltzmann (PB) equations, including:
\begin{itemize}
\item The classical PB equation, which models charge screening in electrolytes (cf. \cite{2023blossey,2002fogolari,2003lamm});
\item The modified PB equation, which incorporates ion size effects and steric interactions (cf. \cite{2011bazant,1997borukhov,2007grochowski,2009li,2013li,2011lu,2023lyu});
\item The charge-conserving PB equation, which enforces global charge neutrality (cf. \cite{2012fontelos,2011lee,2016lee,2006ryham,2014wan}). 
\end{itemize}
To analyze the structure of the diffuse layer, we study boundary layer solutions to PB-type equations in general bounded smooth domains (including multiply connected domains). This analysis provides key insights into electrostatic interactions in complex geometries, with applications in electrochemistry, biophysics, and materials science (cf. \cite{2009altman,2004baker,2015clarke,2008dong,2021eakins,2023elisea-espinoza,2023khlyipin,1995yang,1997yang}).

The classical PB and modified PB equations can be represented by
\begin{align}
\label{eq.1.01}
-\epsilon\Delta\phi_{\epsilon}=f(\phi_{\epsilon})\quad\text{in}~\Omega,
\end{align}
where $\epsilon>0$ is the dielectric constant and $\phi_{\epsilon}$ is the electric potential.
Here $f=f(\phi)$ is a smooth function for the total ionic charge density which satisfies the following conditions.
\begin{enumerate}
\item[(A1)] The function $f=f(\phi)$ is smooth, strictly decreasing in $\phi$, and satisfies\[m_f=m_f(\mathcal{K}):=\sqrt{-\max_{\phi\in\mathcal{K}}f'(\phi)}>0\quad\text{for any compact interval $\mathcal{K}\subseteq\R$.}\]
\item[(A2)] The function $f$ has a unique zero point $\phi^{*}$, i.e., $f(\phi^{*})=0$, where $\phi^{*}$ is called as reference potential (cf. \cite{1995andelman,2018gray}).
\end{enumerate}
For example, function $f(\phi)=\sum\limits_{i=1}^{I}z_{i}c_{i}^{\text{b}}\exp(-z_{i}\phi)$ describes the total ionic charge density of the classical PB equation, and satisfies conditions (A1) and (A2), where $I$ is the number of ion species, $z_{i}\neq0$ is the valence and $c_i^{\text{b}}$ is the concentration of the $i$th species in the bulk (cf. \cite{2002fogolari,2003lamm}).

The charge-conserving PB (CCPB) equation was derived from the static limit of the Poisson--Nernst--Planck equations in order to guarantee charge neutrality within electrolyte domains (cf. \cite{2014wan}).
The CCPB equation can be denoted as
\begin{align}
\label{eq.1.02}
-\epsilon\Delta\phi_{\epsilon}=\sum_{i=1}^{I}\frac{m_{i}z_{i}\exp(-z_{i}\phi_{\epsilon})}{\int_\Omega\!\exp(-z_{i}\phi_{\epsilon}(y))\,\mathrm{d}y}\quad\text{in}~\Omega,
\end{align}
where the constant $m_{i}>0$ is the total concentration of species $i$ with valence $z_{i}\neq0$ for $i=1,\dots,I$.
For charge neutrality, we assume
\begin{align}
\label{eq.1.03}
\sum_{i=1}^{I}m_{i}z_{i}=0.
\end{align}
Due to the integral term $\int_{\Omega}\!\exp(-z_{i}\phi_{\epsilon}(y))\,\mathrm{d}y$, \eqref{eq.1.02} has nonlocal nonlinearity which makes \eqref{eq.1.02} more difficult than \eqref{eq.1.01}.
Note that equation \eqref{eq.1.02} can be denoted as $-\epsilon\Delta\phi_{\epsilon}=f_{\epsilon}(\phi_{\epsilon}):=\sum\limits_{i=1}^{I}z_{i}c_{i,\epsilon}^{\mathrm b}\exp(-z_{i}\phi_{\epsilon})$ in $\Omega$, which has the same form as the classical PB equation.
Here $f_{\epsilon}(\phi)=\sum\limits_{i=1}^{I}z_{i}c_{i,\epsilon}^{\mathrm b}\exp(-z_{i}\phi)$ presents the total ionic charge density (also see \eqref{eq.3.003}) and $c_{i,\epsilon}^{\mathrm b}=m_{i}\big\slash\left(\int_\Omega\!\exp(-z_{i}\phi_{\epsilon}(y))\,\mathrm{d}y\right)$ is the concentration of the $i$th ion species in the bulk.

For the general bounded smooth domain $\Omega\subseteq\R^d$ ($d\geq2$), we assume
\begin{itemize}
\item[(A3)] $\Omega=\Omega_{0}-\bigcup_{k=1}^K\Omega_{k}$, $K\in\N\cup\{0\}$ (the number of holes) and $\Omega_k$ are bounded, smooth, simply connected domains with $\Omega_k\subset\subset\Omega_{0}$ for $k\in\{1,\dots,K\}$ and $\dist(\Omega_{i},\Omega_{j})>0$ for $i,j\in\{1,\dots,K\}$ and $i\neq j$ (see \Cref{figure:1}). Here $\dist$ denotes the distance.
\end{itemize}
\begin{figure}[!htb]\centering\includegraphics[scale=0.3]{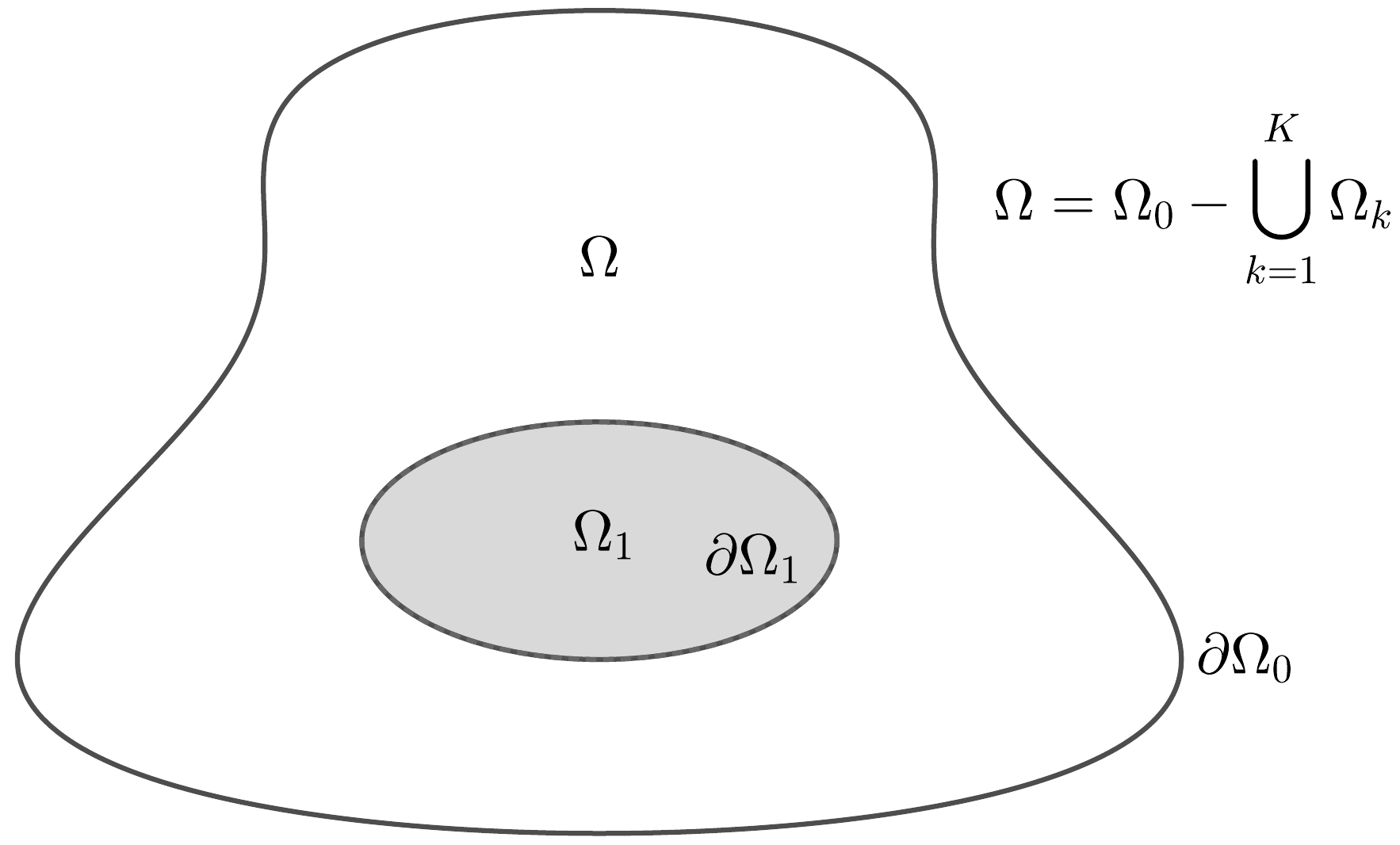}
\caption{We sketch the schematic diagram of bounded smooth domain $\Omega=\Omega_{0}-\bigcup_{k=1}^{K}\Omega_{k}$ and $K=1$, where $\Omega$ has $(K+1)$ boundaries $\partial\Omega_{k}$ for $k=0,1,\dots,K$.}
\label{figure:1}\end{figure}
The boundary condition of \eqref{eq.1.01} and \eqref{eq.1.02} is the Robin boundary condition given by
\begin{align}
\label{eq.1.04}
&\phi_{\epsilon}+\gamma_k\sqrt\epsilon\partial_\nu\phi_{\epsilon}=\phi_{bd,k}\quad\text{on}~\partial\Omega_k~\text{for}~k=0,1,\dots,K,
\end{align}
where $\phi_{bd,k}$ is a constant for the given external electric potential, and $\gamma_k>0$ is the ratio of Stern-layer width to the Debye screening length (cf. \cite{2005bazant,2010das,2006olesen,2010plouraboue}).
As $\phi_{bd,k}=\phi^{*}$ for $k=0,1,\dots,K$, equation \eqref{eq.1.01} with condition \eqref{eq.1.04} has only a trivial solution $\phi_{\epsilon}\equiv\phi^{*}$ so we study equation \eqref{eq.1.01} with condition \eqref{eq.1.04} under the assumption that $\phi_{bd,k}\neq\phi^{*}$ for $k=0,1,\dots,K$.
When constants $\phi_{bd,k}$ are equal, the solution $\phi_{\epsilon}$ to equation \eqref{eq.1.02} with condition \eqref{eq.1.04} is trivial. This leads us to assume that constants $\phi_{bd,k}$ are not equal. When the domain $\Omega=\Omega_{0}$ is simply connected (i.e., condition (A3) with $K=0$), the solution $\phi_{\epsilon}$ to equation \eqref{eq.1.02} with condition \eqref{eq.1.04} is also trivial.
Thus, we study equation \eqref{eq.1.02} with condition \eqref{eq.1.04} under the assumptions that $\phi_{bd,k}$ are not equal and the domain $\Omega$ satisfies (A3) with $K\in\N$.

To find the boundary layer solutions to PB type equations in domain $\Omega$, we assume $\epsilon>0$ (related to the Debye length) as a small parameter tending to zero (cf. \cite{1997barcilon,2015fellner,2015gebbie,2019kamysbayev,2007mori}), and study the singular perturbation problems of \eqref{eq.1.01} and \eqref{eq.1.02} with the Robin boundary condition \eqref{eq.1.04}.
We characterize the asymptotic expansions of solution $\phi_{\epsilon}$ to equation \eqref{eq.1.01} with condition \eqref{eq.1.04} across three distinct regions based on their distance from the boundary:
\begin{itemize}
\item Region I ($\Omega_{k,T,\epsilon}$), where the distance from the boundary $\partial\Omega_k$ is at most $T\sqrt\epsilon$,
\item Region II ($\Omega_{k,T,\epsilon,\beta}$), where the distance from the boundary $\partial\Omega_k$ ranges between $T\sqrt\epsilon$ and $\epsilon^\beta$,
\item Region III ($\Omega_{\epsilon,\beta}$), where the distance from the boundary $\partial\Omega$ is at least $\epsilon^\beta$,
\end{itemize}
for given parameters $T>0$ and $0<\beta<1/2$.
Regions I and II are called tubular neighborhoods around $\partial\Omega_{k}$ (cf. \cite{2004gray}).
A schematic illustration of these regions is shown in \Cref{figure:2}.
\begin{figure}[!htb]\centering\includegraphics[scale=0.27]{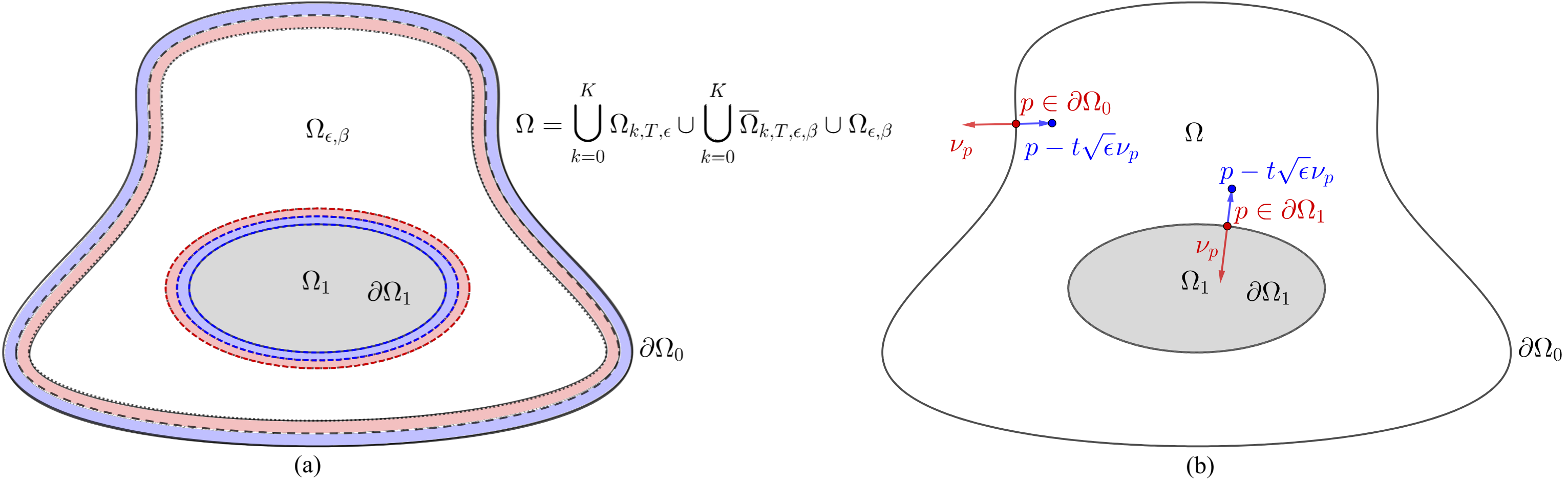}
\caption{Let $0<\beta<1/2$, $T>0$ and $K=1$. In (a), we present the schematic diagram of regions $\Omega_{k,T,\epsilon}=\{x\in\Omega\,:\,\dist(x,\partial\Omega_k)<T\sqrt\epsilon\}$ for $k=0,1$ (blue regions), $\overline\Omega_{k,T,\epsilon,\beta}=\{x\in\Omega\,:\,T\sqrt\epsilon\leq\dist(x,\partial\Omega_k)\leq\epsilon^\beta\}$ for $k=0,1$ (red regions), and $\Omega_{\epsilon,\beta}=\{x\in\Omega\,:\,\dist(x,\partial\Omega)>\epsilon^\beta\}$ (white region) as $\epsilon>0$ sufficiently small. The union of these regions constitutes $\Omega$. In (b), we sketch the point $p-t\sqrt{\epsilon}\nu_{p}$ near the boundary $\partial\Omega_k$ for $k=0,1$, where $p\in\partial\Omega_k$ and $0\leq t\leq\epsilon^{(2\beta-1)/2}$, as $\epsilon\to0^+$. Here $\nu_{p}$ is the unit outer normal at $p$ with respect to $\Omega$.}
\label{figure:2}\end{figure}

\newpage
The results for equation \eqref{eq.1.01} with condition \eqref{eq.1.04} are stated as follows.
\begin{theorem}
\label{theorem:1}
Assume that (A1)--(A3) hold true, and that $\phi_{bd,k}\neq\phi^{*}$ for $k=0,1,\dots,K$.
Let $\phi_{\epsilon}\in\C^\infty(\overline\Omega)$ be the unique solution to equation \eqref{eq.1.01} with condition \eqref{eq.1.04}.
Let $k\in\{0,1,\dots,K\}$, $p\in\partial\Omega_k$, and $T>0$ be arbitrary.
Then
\begin{itemize}
\item[(a)] 
\begin{align}
\label{eq.1.05}
&\phi_{\epsilon}(p-t\sqrt{\epsilon}\nu_{p})=u_{k}(t)+\sqrt\epsilon\Big((d-1)H(p)v_{k}(t)+o_{\epsilon}(1)\Big),\\
\label{eq.1.06}
&\nabla\phi_{\epsilon}(p-t\sqrt{\epsilon}\nu_{p})=-\left(\frac{1}{\sqrt\epsilon}u_{k}'(t)+(d-1)H(p)v_{k}'(t)\right)\nu_{p}+o_{\epsilon}(1)
\end{align}
for $0\leq t\leq T$ as $\epsilon\to0^+$, where $p-t\sqrt{\epsilon}\nu_{p}\in\overline\Omega_{k,T,\epsilon}=\{x\in\overline\Omega\,:\,0\leq\dist(x,\partial\Omega_k)\leq T\sqrt\epsilon\}$ (see \Cref{figure:2}).
Here $H(p)$ and $\nu_{p}$ denote the mean curvature and unit outer normal at $p$ with respect to $\Omega$, respectively.
The term $o_{\epsilon}(1)$, tending to zero as $\epsilon$ goes to zero, depends on $T$ but is independent of $p\in\partial\Omega_k$, which means
\begin{align*}&\lim_{\epsilon\to0^+}\sup_{p\in\partial\Omega_k,\,t\in[0,T]}\left|\frac{1}{\sqrt\epsilon}(\phi_{\epsilon}(p-t\sqrt{\epsilon}\nu_{p})-u_{k}(t))-(d-1)H(p)v_{k}(t)\right|=0,\\&\lim_{\epsilon\to0^+}\sup_{p\in\partial\Omega_k,\,t\in[0,T]}\left|\nabla\phi_{\epsilon}(p-t\sqrt{\epsilon}\nu_{p})+\left(\frac{1}{\sqrt\epsilon}u_{k}'(t)+(d-1)H(p)v_{k}'(t)\right)\nu_{p}\right|=0.\end{align*}
The functions $u_{k}(t)$ and $v_{k}(t)$ are the unique solutions to
\begin{align}
\label{eq.1.07}
&u_{k}''+f(u_{k})=0\quad\text{in}~(0,\infty),\\
\label{eq.1.08}
&u_{k}(0)-\gamma_k u_{k}'(0)=\phi_{bd,k},\\
\label{eq.1.09}
&\lim_{t\to\infty}u_{k}(t)=\phi^{*},
\end{align}and
\begin{align}
\label{eq.1.10}
&v_{k}''+f'(u_{k})v_{k}=u_{k}'\quad\text{in}~(0,\infty),\\
\label{eq.1.11}
&v_{k}(0)-\gamma_kv_{k}'(0)=0,\\
\label{eq.1.12}
&\lim_{t\to\infty}v_{k}(t)=0.
\end{align}
\item[(b)] There exists a positive constant $\epsilon^{*}>0$ such that
\begin{itemize}
\item[(i)]
\begin{align}
\label{eq.1.13}
|\phi_{\epsilon}(p-t\sqrt{\epsilon}\nu_{p})-\phi^{*}|\leq M'\exp(-Mt),\quad|\nabla\phi_{\epsilon}(p-t\sqrt{\epsilon}\nu_{p})|\leq\frac{M'}{\sqrt\epsilon}\exp(-Mt)
\end{align}
for $0<\epsilon<\epsilon^{*}$, $0<\beta<1/2$ and $T\leq t\leq\epsilon^{(2\beta-1)/2}$, i.e., $p-t\sqrt{\epsilon}\nu_{p}\in\overline\Omega_{k,T,\epsilon,\beta}=\{x\in\Omega\,:\,T\sqrt\epsilon\leq\dist(x,\partial\Omega_k)\leq\epsilon^\beta\}$;
\item[(ii)]
\begin{align}
\label{eq.1.14}
|\phi_{\epsilon}(x)-\phi^{*}|\leq M'\exp\left(-M\epsilon^{(2\beta-1)/2}\right),\qquad|\nabla\phi_{\epsilon}(x)|\leq\frac{M'}{\sqrt\epsilon}\exp\left(-M\epsilon^{(2\beta-1)/2}\right)
\end{align}
for $0<\epsilon<\epsilon^{*}$, $0<\beta<1/2$ and $x\in\overline\Omega_{\epsilon,\beta}=\{x\in\Omega\,:\,\dist(x,\partial\Omega)\geq\epsilon^\beta\}$.
\end{itemize}
Here $M>0$ and $M'>0$ are constants independent of $\epsilon$.\end{itemize}\end{theorem}

\begin{remark}
\label{remark:1}
In electrostatics, the solution $\phi_{\epsilon}$ represents the electric potential, which is of order $\mathcal{O}_\epsilon(1)$ (cf. \eqref{eq.1.05}), while the associated electric field $-\nabla\phi_{\epsilon}$ is of order $\mathcal{O}_\epsilon(1/\sqrt\epsilon)$ (cf. \eqref{eq.1.06}) in Region I, denoted $\Omega_{k,T,\epsilon}$.
Reflecting geometric enhancement or suppression, the boundary mean curvature $H(p)$ contributes to the second-order corrections in the asymptotic expansions \eqref{eq.1.05} and \eqref{eq.1.06}, where $u_{k}$ and $v_{k}$ are solutions to equations \eqref{eq.1.07}--\eqref{eq.1.12} and satisfy the following properties.
When the external electric potential $\phi_{bd,k}>\phi^{*}$ and $k\in\{0,1,\dots,K\}$, the function $u_{k}(t)\in(\phi^{*},\phi_{bd,k})$ is strictly decreasing for $t\in(0,\infty)$, and $v_{k}$ is positive on $(0,\infty)$, strictly increasing on $(0,t_k^{*})$ and strictly decreasing on $(t_k^{*},\infty)$, where $t_k^{*}0$ is a constant.
Conversely, when $\phi_{bd,k}<\phi^{*}$, $u_{k}(t)\in(\phi_{bd,k},\phi^{*})$ is strictly increasing for $t\in(0,\infty)$, and $v_{k}$ is negative on $(0,\infty)$, strictly decreasing on $(0,t_k^{*})$ and strictly increasing on $(t_k^{*},\infty)$ (cf. \Cref{proposition:B.1,proposition:C.1}).
The asymptotic formulas \eqref{eq.1.05}--\eqref{eq.1.06} thus provide approximations for the electric potential and electric field induced by an applied potential difference.
In addition, the exponential decay estimates are given by \eqref{eq.1.13} with respect to the variable $t$ in Region II ($\Omega_{k,T,\epsilon,\beta}$), and by \eqref{eq.1.14} with respect to the parameter $\epsilon$ in Region III ($\Omega_{\epsilon,\beta}$), highlighting the localized nature of boundary layer phenomena.
\end{remark}

The total ionic charge density plays crucial roles in the behavior of biological and physical systems (cf. \cite{1989russel}).
Using formulas \eqref{eq.1.05}, \eqref{eq.1.13} and \eqref{eq.1.14} along with Taylor expansion of $f$, we derive the asymptotic expansion of the total ionic charge density $f(\phi_{\epsilon})$ as follows.
Let $k\in\{0,1,\dots,K\}$, $p\in\partial\Omega_k$, $T>0$, and $0<\beta<1/2$ be arbitrary.
\begin{itemize}
\item In Region I ($\Omega_{k,T,\epsilon}$),
\begin{align}
\label{eq.1.15}
f(\phi_{\epsilon}(p-t\sqrt{\epsilon}\nu_{p}))=f(u_{k}(t))+\sqrt\epsilon[(d-1)H(p)f'(u_{k}(t))v_{k}(t)+o_{\epsilon}(1)]\quad\text{for}~0\leq t\leq T;\end{align}
\item In Region II ($\Omega_{k,T,\epsilon,\beta}$),
\[|f(\phi_{\epsilon}(p-t\sqrt{\epsilon}\nu_{p}))|\leq M'\exp(-Mt)\quad\text{for}~T\leq t\leq\epsilon^{(2\beta-1)/2};\]
\item In Region III ($\Omega_{\epsilon,\beta}$),
\[|f(\phi_{\epsilon}(x))|\leq M'\exp(-M\epsilon^{(2\beta-1)/2})\quad\text{for}~x\in\overline\Omega_{\epsilon,\beta},\]
\end{itemize}
as $0<\epsilon<\epsilon^{*}$, where $\epsilon^{*}$ comes from \Cref{theorem:1}(b), $M'$ and $M$ are generic positive constants independent of $\epsilon$.
Here we have used the condition $f(\phi^{*})=0$, which follows from assumption (A2).

The total ionic charge within regions $\overline\Omega_{k,T,\epsilon}$, $\overline\Omega_{k,T,\epsilon,\beta}$ and $\overline\Omega_{\epsilon,\beta}$ can be expressed by the integral of the total ionic charge density $f(\phi_{\epsilon})$ over $\overline\Omega_{k,T,\epsilon}$, $\overline\Omega_{k,T,\epsilon,\beta}$ and $\overline\Omega_{\epsilon,\beta}$, respectively. By \Cref{theorem:1}, we have the following results.
\begin{corollary}\label{corollary:1}
Under the hypothesis of \Cref{theorem:1}, if $\phi_{bd,k}>\phi^{*}$ (or $\phi_{bd,k}<\phi^{*}$), then we have
\begin{itemize}
\item[(a)] $\int_{\overline\Omega_{k,T,\epsilon}}\!f(\phi_{\epsilon}(x))\,\mathrm{d}x=\sqrt\epsilon|\partial\Omega_k|(u_{k}'(0)-u_{k}'(T))$\\${\color{white}\int_{\overline\Omega_{k,T,\epsilon}}\!f(\phi_{\epsilon}(x))\,\mathrm{d}x=}+\,\epsilon(d-1)\left(\int_{\partial\Omega_k}\!H(p)\,\mathrm{d}S_p\right)(Tu_{k}'(T)+v_{k}'(0)-v_{k}'(T))+\epsilon o_{\epsilon}(1)<0$ (or $>0$),
\item[(b)] $\int_{\overline\Omega_{k,T,\epsilon,\beta}}\!f(\phi_{\epsilon}(x))\,\mathrm{d}x=\sqrt\epsilon|\partial\Omega_k|u_{k}'(T)+\epsilon(d-1)\left(\int_{\partial\Omega_k}\!H(p)\,\mathrm{d}S_p\right)(-Tu_{k}'(T)+v_{k}'(T))+\epsilon o_{\epsilon}(1)<0$ (or $>0$),
\item[(c)] $\left|\int_{\overline\Omega_{\epsilon,\beta}}\!f(\phi_{\epsilon}(x))\,\mathrm{d}x\right|\leq\sqrt\epsilon M'\exp\left(-M\epsilon^{(2\beta-1)/2}\right)$
\end{itemize}
for $0<\epsilon<\epsilon^{*}$ and $0<\beta<1/2$, where $|\partial\Omega_k|$ is the surface area of $\partial\Omega_k$.
\end{corollary}
\noindent 
\Cref{corollary:1}(a) and (b) imply that the total ionic charge in the near-boundary regions $\Omega_{k,T,\epsilon}$ and $\Omega_{k,T,\epsilon,\beta}$ is negative (or positive) when $\phi_{bd,k}>\phi^{*}$ (or $\phi_{bd,k}<\phi^{*}$), the external electric potential is greater (or less) than the reference potential.
This gives a way for the redistribution of ionic species by the applied electric potential difference.
Moreover,
\begin{align}
\label{eq.1.16}
0<\frac{\int_{\overline\Omega_{k,T,\epsilon,\beta}}\!f(\phi_{\epsilon}(x))\,\mathrm{d}x}{\int_{\overline\Omega_{k,T,\epsilon}}\!f(\phi_{\epsilon}(x))\,\mathrm{d}x}\to\frac{u_{k}'(T)}{u_{k}'(0)-u_{k}'(T)}\quad\text{and}\quad\frac{|\overline\Omega_{k,T,\epsilon}|}{|\overline\Omega_{k,T,\epsilon,\beta}|}\to0\quad\text{as}~\epsilon\to0^+,\end{align}
for $T>0$, $0<\epsilon<1/2$ and $k=0,1,\dots,K$.
Since $u_{k}'(T)$ decays exponentially to zero as $T$ goes to infinity (cf. \Cref{proposition:B.2}(a)), then \eqref{eq.1.16} shows that for sufficiently large $T$, the majority of charged particles concentrate within the region $\Omega_{k,T,\epsilon}$, whose volume satisfies $|\Omega_{k,T,\epsilon}|\ll|\Omega_{k,T,\epsilon,\beta}|$ as $\epsilon\to0^+$.
In addition, \Cref{corollary:1}(c) implies that the total ionic charge in the region $\Omega_{\epsilon,\beta}$ decays exponentially as $\epsilon\to0^+$, demonstrating the emergence of electroneutrality in this bulk region.
This result provides a rigorous justification for the commonly used electroneutrality assumption in bulk electrolyte solutions, while also quantifying the rate at which neutrality is approached in terms of the parameter $\beta$.

Analogously to \Cref{theorem:1}, we characterize the asymptotic expansions of solution $\phi_{\epsilon}$ to equation \eqref{eq.1.02} with condition \eqref{eq.1.04} across three distinct regions.
The results can be stated as below.
\begin{theorem}
\label{theorem:2}
Assume that the domain $\Omega$ satisfies condition (A3) with $K\in\N$, $\phi_{bd,k}$ are not equal, and the charge neutrality condition \eqref{eq.1.03} holds true, where $m_{i}>0$ and $z_{i}\neq0$ for $i=1,\dots,I$.
Let $\phi_{\epsilon}\in\C^\infty(\overline\Omega)$ be the unique solution to equation \eqref{eq.1.02} with condition \eqref{eq.1.04}.
Let $k\in\{0,1,\dots,K\}$, $p\in\partial\Omega_k$, and $T>0$ be arbitrary.
Then
\begin{itemize}
\item[(a)]
\begin{align}
\label{eq.1.17}
&\phi_{\epsilon}(p-t\sqrt{\epsilon}\nu_{p})=u_{k}(t)+\sqrt\epsilon[(d-1)H(p)v_{k}(t)+w_{k}(t)+o_{\epsilon}(1)],\\
\label{eq.1.18}
&\nabla\phi_{\epsilon}(p-t\sqrt{\epsilon}\nu_{p})=-\left(\frac{1}{\sqrt\epsilon}u_{k}'(t)+(d-1)H(p)v_{k}'(t)+w_{k}'(t)\right)\nu_{p}+o_{\epsilon}(1)
\end{align}
for $0\leq t\leq T$ as $\epsilon\to0^+$, where $p-t\sqrt{\epsilon}\nu_{p}\in\overline\Omega_{k,T,\epsilon}$ (see \Cref{figure:2}).
Here $H(p)$ and $\nu_{p}$ denote the mean curvature and the unit outer normal at $p$ with respect to $\Omega$, respectively.
The term $o_{\epsilon}(1)$, tending to zero as $\epsilon\to0^+$, depends on $T$ but is independent of $p\in\partial\Omega_k$, which means
\begin{align*}&\lim_{\epsilon\to0^+}\sup_{p\in\partial\Omega_k,\,t\in[0,T]}\left|\frac{1}{\sqrt\epsilon}(\phi_{\epsilon}(p-t\sqrt{\epsilon}\nu_{p})-u_{k}(t))-(d-1)H(p)v_{k}(t)-w_{k}(t)\right|=0,\\&\lim_{\epsilon\to0^+}\sup_{p\in\partial\Omega_k,\,t\in[0,T]}\left|\nabla\phi_{\epsilon}(p-t\sqrt{\epsilon}\nu_{p})+\left(\frac{1}{\sqrt\epsilon}u_{k}'(t)+(d-1)H(p)v_{k}'(t)+w_{k}'(t)\right)\nu_{p}\right|=0.\end{align*}
The functions $u_{k}(t)$, $v_{k}(t)$, and $w_{k}(t)$ are the unique solutions to
\begin{align}
\label{eq.1.19}
&u_{k}''+f_{0}(u_{k})=0\quad\text{in}~(0,\infty),\\
\label{eq.1.20}
&u_{k}(0)-\gamma_ku_{k}'(0)=\phi_{bd,k},\\
\label{eq.1.21}
&\lim_{t\to\infty}u_{k}(t)=\phi_{0}^{*},
\end{align}
\begin{align}
\label{eq.1.22}
&v_{k}''+f_{0}'(u_{k})v_{k}=u_{k}'\quad\text{in}~(0,\infty)\\
\label{eq.1.23}
&v_{k}(0)-\gamma_kv_{k}'(0)=0,\\
\label{eq.1.24}
&\lim_{t\to\infty}v_{k}(t)=0,
\end{align}
and
\begin{align}
\label{eq.1.25}
&w_{k}''+f_{0}'(u_{k})w_{k}=-f_1(u_{k})\quad\text{in}~(0,\infty),\\
\label{eq.1.26}
&w_{k}(0)-\gamma_kw_{k}'(0)=0,\\
\label{eq.1.27}
&\lim_{t\to\infty}w_{k}(t)=Q.
\end{align}
The constants $\phi_{0}^{*}$ and $Q$ are determined uniquely by \eqref{eq.3.026}--\eqref{eq.3.027} and \eqref{eq.3.114}, respectively.
The functions $f_{0}$ and $f_1$ are given by
\begin{align}
\label{eq.1.28}
&f_{0}(\phi)=\frac{1}{|\Omega|}\sum_{i=1}^{I}m_{i}z_{i}\exp(-z_{i}(\phi-\phi_{0}^{*}))\quad\text{for}~\phi\in\R,\\
\label{eq.1.29}
&f_1(\phi)=-Qf_{0}'(\phi)+\hat f_1(\phi)\quad\text{for}~\phi\in\R,
\end{align}
where $|\Omega|$ is the volume of $\Omega$, and\begin{align*}&\hat{f}_{1}(\phi)=\frac{1}{|\Omega|}\sum_{i=1}^{I}\hat m_{i}z_{i}\exp(-z_{i}(\phi-\phi_{0}^{*}))\quad\text{for}~\phi\in\R,\\&\hat m_{i}=\frac{m_{i}}{|\Omega|}\sum_{k=0}^K|\partial\Omega_k|\int_{0}^{\infty}\![1-\exp(-z_{i}(u_{k}(s)-\phi_{0}^{*}))]\,\mathrm{d}s\quad\text{for}~i=1,\dots,I.\end{align*}
Furthermore, the constant  $Q$ satisfies
\begin{align}
\label{eq.1.30}
Q=\lim_{\epsilon\to0^+}\frac{\phi_{\epsilon}^{*}-\phi_{0}^{*}}{\sqrt\epsilon}=\frac{\dd\sum_{k=0}^{K}\frac{|\partial\Omega_{k}|\hat{F}_{1}(u_{k}(0))+(d-1)\left(\int_{\partial\Omega_{k}}\!H(p)\,\mathrm{d}S_{p}\right)\left(\int_{0}^{\infty}\!u_{k}'^2(s)\,\mathrm{d}s\right)}{u_{k}'(0)+\gamma_{k}f_{0}(u_{k}(0))}}{\dd\sum_{k=0}^{K}\frac{|\partial\Omega_{k}|f_{0}(u_{k}(0))}{u_{k}'(0)+\gamma_{k}f_{0}(u_{k}(0))}},\end{align}
where $\phi_{\epsilon}^{*}$ is the unique zero of function $f_{\epsilon}$ for $\epsilon>0$ (cf. \Cref{proposition:3.02}) and $\hat F_1(\phi)=\int_{\phi^*}^\phi\!\hat{f}_1(s)\,\mathrm ds=\frac1{|\Omega|}\sum_{i=1}^I\hat{m}_i[1-\exp(-z_i(\phi-\phi_0^*))]$ for $\phi\in\R$.
\item[(b)] There exists a positive constant $\epsilon^{*}>0$ such that
\begin{itemize}
\item[(i)] 
\begin{align}
\label{eq.1.31}
|\phi_{\epsilon}(p-t\sqrt{\epsilon}\nu_{p})-\phi_{\epsilon}^{*}|\leq M'\exp(-Mt),\quad|\nabla\phi_{\epsilon}(p-t\sqrt{\epsilon}\nu_{p})|\leq\frac{M'}{\sqrt\epsilon}\exp(-Mt)
\end{align}
for $0<\epsilon<\epsilon^{*}$, $0<\beta<1/2$ and $T\leq t\leq\epsilon^{(2\beta-1)/2}$, i.e., $p-t\sqrt{\epsilon}\nu_{p}\in\overline\Omega_{k,T,\epsilon,\beta}$;
\item[(ii)] \begin{align}
\label{eq.1.32}
|\phi_{\epsilon}(x)-\phi_{\epsilon}^{*}|\leq M'\exp\left(-M\epsilon^{(2\beta-1)/2}\right),\quad|\nabla\phi_{\epsilon}(x)|\leq\frac{M'}{\sqrt\epsilon}\exp\left(-M\epsilon^{(2\beta-1)/2}\right)
\end{align}
for $0<\epsilon<\epsilon^{*}$ and $0<\beta<1/2$.
\end{itemize}
Here $M>0$ and $M'>0$ are constants independent of $\epsilon$.\end{itemize}\end{theorem}
\newpage
The main difference of \Cref{theorem:1,theorem:2} comes from the integral terms in \eqref{eq.1.02} which give the solution $w_{k}$ to \eqref{eq.1.25}--\eqref{eq.1.27} and \eqref{eq.1.30}. This reflects the effect of nonlocal nonlinearity in \eqref{eq.1.02}.

\begin{remark}
\label{remark:2}
Recall that $f_{\epsilon}(\phi)=\sum\limits_{i=1}^{I}z_{i}c_{i,\epsilon}^{\mathrm{b}}\exp(-z_{i}\phi)$ represents the total ionic charge density and $c_{i,\epsilon}^{\mathrm{b}}=\dfrac{m_{i}}{\int_{\Omega}\!\exp(-z_{i}\phi_{\epsilon}(y))\,\mathrm{d}y}$ is the concentration of the $i$th ion species in the bulk.
From \eqref{eq.1.30}, we find
\begin{align}
\label{eq.1.33}
&c_{i,\epsilon}^{\mathrm b}=c_i^{\mathrm b}+\sqrt\epsilon\left(\frac{m_{i}z_{i}Q\exp(z_{i}\phi_{0}^{*})+\hat m_{i}\exp(z_{i}\phi_{0}^{*})}{|\Omega|}+o_{\epsilon}(1)\right)\quad\text{for}~i=1,\dots,I,\\
\label{eq.1.34}
&f_{\epsilon}(\phi)=f_{0}(\phi)+\sqrt\epsilon(f_1(\phi)+o_{\epsilon}(1))\quad\text{for}~\phi\in\R,
\end{align}
(cf. \Cref{remark:7} and \eqref{eq.3.096} in \Cref{section:3.4}), where $c_i^{\mathrm b}:=m_{i}\exp(z_{i}\phi_{0}^{*})/|\Omega|$ (is the concentration of the $i$th ion species in the bulk) may correspond to the bulk concentration in the classical PB equation $-\epsilon\Delta\phi_{\epsilon}=\sum_{i=1}^{I}z_{i}c_i^{\mathrm b}\exp(-z_{i}\phi_{\epsilon})=f_{0}(\phi_{\epsilon})$ in $\Omega$.
Here $f_{0}$ and $f_1$ are given in \eqref{eq.1.28}--\eqref{eq.1.29}. 
\end{remark}

As in \Cref{remark:1}, we have the following observation.
\begin{remark}
\label{remark:3}
In electrostatics, the solution $\phi_{\epsilon}$ represents the electric potential, and its gradient $-\nabla\phi_{\epsilon}$ represents the electric potential, with orders $\mathcal{O}_{\epsilon}(1)$ and $\mathcal{O}_{\epsilon}(1/\sqrt{\epsilon})$, respectively, in the region $\Omega_{k,T,\epsilon}$.
Similar to \Cref{remark:1}, the boundary curvature $H(p)$ contributes to the second-order corrections in the asymptotic expansions \eqref{eq.1.17}--\eqref{eq.1.18}, where $u_k$, $v_k$, and $w_k$ are the solutions to \eqref{eq.1.19}--\eqref{eq.1.27} and satisfy the following properties.
When the external electric potential $\phi_{bd,k}>\phi_{0}^{*}$ and $k\in\{0,1,\dots,K\}$, the function $u_{k}(t)\in(\phi_{0}^{*},\phi_{bd,k})$ is strictly decreasing for $t\in(0,\infty)$, and $v_{k}$ is positive on $(0,\infty)$, strictly increasing on $(0,t_k^{*})$ and strictly decreasing on $(t_{k}^{*},\infty)$, where $t_k^{*}$ is a constant.
Conversely, when $\phi_{bd,k}<\phi_{0}^{*}$, $u_{k}(t)\in(\phi_{bd,k},\phi_{0}^{*})$ is strictly increasing for $t\in(0,\infty)$, and $v_{k}$ is negative on $(0,\infty)$, strictly decreasing on $(0,t_k^{*})$ and strictly increasing on $(t_k^{*},\infty)$ (cf. \Cref{proposition:B.1,proposition:C.1}).
The asymptotic formulas \eqref{eq.1.17}--\eqref{eq.1.18} thus provide approximations for the electric potential and electric field induced by an applied potential difference.
Additionally, the exponential decay estimates are given by \eqref{eq.1.31} with respect to the variable $t$ in Region II ($\Omega_{k,T,\epsilon,\beta}$), and by \eqref{eq.1.32} with respect to the parameter $\epsilon$ in Region III ($\Omega_{\epsilon,\beta}$), highlighting the localized nature of boundary layer phenomena.
However, $w_{k}$, which arises from the nonlocal nonlinearity, may not behave like $u_{k}$ and $v_{k}$ because the constants $\hat m_{i}$ do not have the same sign.
\end{remark}

As the previous discussion for total ionic charge density of \eqref{eq.1.01}, we use formulas \eqref{eq.1.17}, \eqref{eq.1.31}, \eqref{eq.1.32}, \eqref{eq.1.34}, and the Taylor expansions of $f_0$ and $f_1$ to derive the asymptotic expansion of the total ionic charge density $f_{\epsilon}(\phi_{\epsilon})$ as follows.
Let $k\in\{0,1,\dots,K\}$, $p\in\partial\Omega_k$, $T>0$, and $0<\beta<1/2$ be arbitrary.
\begin{itemize}
\item In Region I ($\Omega_{k,T,\epsilon}$),
\begin{align}
\label{eq.1.35}
f_{\epsilon}(\phi_{\epsilon}(p-t\sqrt{\epsilon}\nu_{p}))=f_0(u_{k}(t))+\sqrt{\epsilon}[(d-1)H(p)f_{0}'(u_{k}(t))v_{k}(t)+f_{0}'(u_{k}(t))w_{k}(t)+f_{1}(u_{k}(t))+o_{\epsilon}(1)]\end{align}
for $0\leq t\leq T$;
\item In Region II ($\Omega_{k,T,\epsilon,\beta}$),
\[|f_{\epsilon}(\phi_{\epsilon}(p-t\sqrt{\epsilon}\nu_{p}))|\leq M'\exp(-Mt)\quad\text{for}~T\leq t\leq\epsilon^{(2\beta-1)/2};\]
\item In Region III ($\Omega_{\epsilon,\beta}$),
\[|f_{\epsilon}(\phi_{\epsilon}(x))|\leq M'\exp(-M\epsilon^{(2\beta-1)/2})\quad\text{for}~x\in\overline\Omega_{\epsilon,\beta},\]
\end{itemize}
as $0<\epsilon<\epsilon^{*}$, where $\epsilon^{*}$ comes from \Cref{theorem:2}(b), $M'$ and $M$ are generic positive constants independent of $\epsilon$.
Here we have used the condition $f_{\epsilon}(\phi_{\epsilon}^{*})=0$, which follows from \Cref{proposition:3.02}, and $f_0(\phi_0^*)=0$.

As in \Cref{corollary:1}, we use \Cref{theorem:2} to derive the asymptotic expansions of total ionic charge within regions $\overline\Omega_{k,T,\epsilon}$, $\overline\Omega_{k,T,\epsilon,\beta}$ and $\overline\Omega_{\epsilon,\beta}$ as below.
\begin{corollary}\label{corollary:2}
Under the hypothesis of \Cref{theorem:2}, if $\phi_{bd,k}>\phi_{0}^{*}$ (or $\phi_{bd,k}<\phi_{0}^{*}$), then we have
\begin{itemize}
\item[(a)] $\int_{\overline\Omega_{k,T,\epsilon}}\!f_{\epsilon}(\phi_{\epsilon}(x))\,\mathrm{d}x=\sqrt\epsilon|\partial\Omega_k|(u_{k}'(0)-u_{k}'(T))\\{\color{white}\int_{\overline\Omega_{k,T,\epsilon}}\!f_{\epsilon}(\phi_{\epsilon}}+\epsilon\left[(d-1)\left(\int_{\partial\Omega_k}\!H(p)\,\mathrm{d}S_p\right)(Tu_{k}'(T)+v_{k}'(0)-v_{k}'(T))+|\partial\Omega_k|(w_{k}'(0)-w_{k}'(T))+o_{\epsilon}(1)\right]<0$ (or $>0$),
\item[(b)] $\int_{\overline\Omega_{k,T,\epsilon,\beta}}\!f_{\epsilon}(\phi_{\epsilon}(x))\,\mathrm{d}x=\sqrt\epsilon|\partial\Omega_k|u_{k}'(T)$\\${\color{white}\int_{\overline\Omega_{k,T,\epsilon,\beta}}\!f_{\epsilon}(\phi_{\epsilon}(x))\,\mathrm{d}x=}+\epsilon\left[(d-1)\left(\int_{\partial\Omega_k}\!H(p)\,\mathrm{d}S_p\right)(-Tu_{k}'(T)+v_{k}'(T))+|\partial\Omega_k|w_{k}'(T)\right]+\epsilon o_{\epsilon}(1)\Big]<0$ (or $>0$),
\item[(c)] $\left|\int_{\overline\Omega_{\epsilon,\beta}}\!f_{\epsilon}(\phi_{\epsilon}(x))\,\mathrm{d}x\right|\leq\sqrt\epsilon M'\exp\left(-M\epsilon^{(2\beta-1)/2}\right)$
\end{itemize}
for $0<\epsilon<\epsilon^{*}$ and $0<\beta<1/2$, where $|\partial\Omega_k|$ is the surface area of $\partial\Omega_k$.
\end{corollary}
\noindent \Cref{corollary:2} is consistent with the charge neutrality condition \eqref{eq.1.03}, as demonstrated by equations \eqref{eq.3.027} and \eqref{eq.3.115} in \Cref{section:3}.
\Cref{corollary:2}(a) and \ref{corollary:2}(b) also indicate that the total ionic charge in the near-boundary regions $\Omega_{k,T,\epsilon}$ and $\Omega_{k,T,\epsilon,\beta}$ is negative (or positive) when $\phi_{bd,k}>\phi_0^{*}$ (or $\phi_{bd,k}<\phi_0^{*}$).
Such behavior is analogous to that described in \Cref{corollary:1}, where the sign of the total ionic charge is similarly determined by the comparison between $\phi_{bd,k}$ and $\phi_0^*$.
This redistribution of ionic species is driven by the applied electric potential difference, consistent with the behavior in \Cref{corollary:1}.
Moreover, analogously to \eqref{eq.1.16}, 
\[0<\frac{\int_{\overline{\Omega}_{k,T,\epsilon,\beta}}\!f_{\epsilon}(\phi_{\epsilon}(x))\,\mathrm dx}{\int_{\overline{\Omega}_{k,T,\epsilon}}\!f_{\epsilon}(\phi_{\epsilon}(x))\,\mathrm dx}\to\frac{u_{k}'(T)}{u_{k}'(0)-u_{k}'(T)}\quad\text{and}\quad\frac{|\overline{\Omega}_{k,T,\epsilon}|}{|\overline{\Omega}_{k,T,\epsilon,\beta}|}\to0\quad\text{as}~\epsilon\to0^+.\]
Since $u_{k}'(T)$ decays exponentially to zero as $T\to\infty$ (cf. \Cref{proposition:B.2}(a)), this implies that for sufficiently large $T$, the majority of charged particles are concentrated within the region $\Omega_{k,T,\epsilon}$, despite its volume being significantly smaller than that of the region $\Omega_{k,T,\epsilon,\beta}$.
Additionally, \Cref{corollary:2}(c) implies that the total ionic charge in the region $\Omega_{\epsilon,\beta}$ decays exponentially as $\epsilon\to0^+$ as in \Cref{corollary:1}(c).

For the proof of \Cref{theorem:1}, we first establish the uniform boundedness of the solution $\phi_{\epsilon}$ (cf. \Cref{proposition:2.1}), and then employ the principal coordinate system (cf. \cite[Section 14.6]{1977gilbarg}) and rescale spatial variables by $\sqrt\epsilon$ to obtain the local convergence and to derive equation \eqref{eq.2.10} in $\R_{+}^{d}$ with condition \eqref{eq.2.11} (cf. \Cref{lemma:2.3}).
Using the exponential-type estimate of \eqref{eq.2.10}--\eqref{eq.2.11} (cf. \Cref{lemma:2.4}) and the moving plane arguments (cf. \Cref{proposition:2.5}), we prove the first-order asymptotic expansion of $\phi_{\epsilon}$ (see \eqref{eq.2.39}--\eqref{eq.2.40}).
Using the first-order asymptotic expansion, we derive the exponential-type estimate of $\phi_{\epsilon}$ and $|\nabla\phi_{\epsilon}|$ (cf. \Cref{proposition:2.6}), which implies \Cref{theorem:1}(b) holds true.
For the second-order asymptotic expansion of $\phi_{\epsilon}$, we study $\varphi_{k,\epsilon}(x)=\epsilon^{-1/2}[\phi_{\epsilon}(x)-u_{k}(\delta_k(x)/\sqrt\epsilon)]$ for $x\in\overline\Omega_{k,\epsilon,\beta}:=\{x\in\overline\Omega\,:\,\delta_k(x)\leq\epsilon^{\beta}\}$ and $k=0,1,\dots,K$.
We prove the uniform boundedness of $\varphi_{k,\epsilon}$ (cf. \Cref{proposition:2.7}) and the local convergence to get equation \eqref{eq.2.67} in $\R_{+}^{d}$ with condition \eqref{eq.2.68} (cf. \Cref{lemma:2.8}).
Using the exponential-type estimate \eqref{eq.2.69} and the moving plane arguments (cf. \Cref{proposition:2.9}), we obtain \eqref{eq.1.05} and \eqref{eq.1.06} in \Cref{theorem:1}(a).

For the proof of \Cref{theorem:2}, we rewrite equation \eqref{eq.1.02} as $-\epsilon\Delta\phi_{\epsilon}=f_{\epsilon}(\phi_{\epsilon})$ in $\Omega$, where $f_{\epsilon}(\phi)=\sum_{i=1}^{I}(m_{i}z_{i}/A_{i,\epsilon})\exp(-z_{i}\phi)$ for $\phi\in\R$, $A_{i,\epsilon}=\int_\Omega\!\exp(-z_{i}\phi_{\epsilon}(y))\,\mathrm{d}y$ for $i=1,\dots,I$, and $\phi_{\epsilon}$ is the solution to equation \eqref{eq.1.02} with condition \eqref{eq.1.04}.
In \Cref{section:3.1}, we first prove the uniform boundedness of $\phi_{\epsilon}$ and show that $f_{\epsilon}$ also satisfies conditions (A1)--(A2) with the unique zero $\phi_{\epsilon}^{*}$.
As in the proof in \Cref{section:2.1}, we establish the first-order asymptotic expansions of $\phi_{\epsilon}$ and $\nabla\phi_{\epsilon}$.
Moreover, we also determine the unique value of $\phi_0^*$ by the algebraic equations \eqref{eq.3.026}--\eqref{eq.3.027}.
However, unlike the proof of \Cref{theorem:1}, the unique zero $\phi_{\epsilon}^*$ of $f_{\epsilon}$ may depend on the parameter $\epsilon$ so it is pivotal to prove, by contradiction, the uniform boundedness of $|\phi_{\epsilon}^{*}-\phi_0^*|/\sqrt{\epsilon}$ (cf. \Cref{section:3.3}).
More specifically, in \Cref{section:3.3}, under the assumption $|\phi_{\epsilon}^{*}-\phi_0^*|/\sqrt{\epsilon}\to\infty$ as $\epsilon$ goes to zero, we obtain the second-order asymptotic expansions of $A_{i,\epsilon}$, $f_{\epsilon}$, and $\phi_{\epsilon}$ and the total ionic charge over three distinct regions, which contradicts the electroneutrality condition \eqref{eq.1.03}.
Then in \Cref{section:3.4}, we use the uniform boundedness of $|\phi_{\epsilon}^{*}-\phi_0^*|/\sqrt{\epsilon}$ to get the asymptotic expansion of $A_{i,\epsilon}$ (see \Cref{remark:7}). This implies that the function $f_{\epsilon}$ has the asymptotic expansion $f_{\epsilon}(\phi)=f_{0}(\phi)+\sqrt\epsilon(f_1(\phi)+o_{\epsilon}(1))$ for $\phi\in\R$, where $f_{0}$ and $f_1$ are smooth functions (see \eqref{eq.1.28}--\eqref{eq.1.29}).
Thus, equation \eqref{eq.1.02} can be transformed into $-\epsilon\Delta\phi_{\epsilon}=f_{0}(\phi_{\epsilon})+\sqrt\epsilon(f_1(\phi_{\epsilon})+o_{\epsilon}(1))$ in $\Omega$ (cf. \eqref{eq.3.097}), where $f_{0}$ satisfies (A1)--(A2) with a unique zero $\phi_{0}^{*}$ and $f_1$ is bounded on any compact subsets of $\R$.
Therefore, as in the proof of \Cref{theorem:1}, we may apply the principal coordinate system and the moving plane arguments to prove \Cref{theorem:2}(a).

The electric forces play a central role in the behavior of biological and physical systems (cf. \cite{2014wan}).
In \cite{2017griffiths,1989russel}, the Maxwell stress tensor $\mathscr{T}_{\text{M}}$ of equation \eqref{eq.1.01} is defined as
\begin{align}
\label{eq:2025.0609.1538}
\mathscr{T}_{\text{M}}=\varepsilon\nabla\phi_\varepsilon\otimes\nabla\phi_\varepsilon-\frac\varepsilon2|\nabla\phi_\varepsilon|^2I_d,
\end{align}
where $\otimes$ denotes the tensor product.
Then by \eqref{eq.1.06} and \eqref{eq:2025.0609.1538}, we get the Maxwell stress tensor $\mathscr T_{\text M}$ at $p-t\sqrt{\varepsilon}\nu_p$:
\[\mathscr T_{\text M}(p-t\sqrt{\varepsilon}\nu_p)=[u_k'^2(t)+2\sqrt{\varepsilon}(d-1)H(p)u_k'(t)v_k'(t)]\left(\nu_p\otimes\nu_p-\frac12I_d\right)+\sqrt{\varepsilon}o_\varepsilon(1)\]
for $0\leq t\leq T$ and $p\in\partial\Omega_k$.
By \eqref{eq.B.08} and \eqref{eq.B.13} with $U=u_k$, we have $u_k'^2(t)=2(F(\phi^*)-F(u_k(t)))=-2F(u_k(t))$ for all $t\geq0$, where $F(\phi)=\int_{\phi^*}^\phi\!f(s)\,\mathrm ds$ and $\phi^*$ is the unique zero of $f$.
Hence
\begin{align*}
\mathscr{T}_{\text M}(p-t\sqrt{\varepsilon}\nu_p)=[-F(u_k(t))+\sqrt\varepsilon(d-1)H(p)u_k'(t)v_k'(t)](2\nu_p\otimes\nu_p-I_d)+\sqrt\varepsilon o_\varepsilon(1),
\end{align*}
which gives
\begin{align}
\label{eq:2025.0609.1626}
\mathscr{T}_{\text{M}}(p-t\sqrt{\varepsilon}\nu_p)\nu_p=-F(u_k(t))\nu_p+\sqrt{\varepsilon}[(d-1)H(p)u_k'(t)v_k'(t)\nu_p+o_\varepsilon(1)]
\end{align}
for $0\leq t\leq T$ and $p\in\partial\Omega_k$.
Note that $F(u_k(t))<0$ for $t\geq0$ because $F(\phi^*)=F'(\phi^*)=0$ and $F''(\phi)=f'(\phi)<0$ for $\phi\in\R$.
In \eqref{eq:2025.0609.1626}, $\mathscr{T}_{\text{M}}(p-t\sqrt{\varepsilon}\nu_p)\nu_p$ is the stress tensor $\mathscr{T}_{\text{M}}(p-t\sqrt{\varepsilon}\nu_p)$ acting on $\nu_p$, which provides the electric force at $p-t\sqrt{\varepsilon}\nu_p$ for $0\leq t\leq T$, $p\in\partial\Omega_k$, and $k\in\{0,1,\dots,K\}$.

For equation \eqref{eq.1.02} with condition \eqref{eq.1.04}, we can use \eqref{eq.1.05}--\eqref{eq.1.12}, \eqref{eq.1.18}, and \eqref{eq:2025.0609.1538} directly to derive the asymptotic formula of the Maxwell stress tensor $\mathscr{T}_{\text{M}}$:
\begin{align*}
\mathscr{T}_{\text{M}}(p-t\sqrt{\varepsilon}\nu_p)=\{u_k'^2(t)+2\sqrt{\varepsilon}u_k'(t)[(d-1)H(p)v_k'(t)+w_k'(t)]\}\left(\nu_p\otimes\nu_p-\frac12I_d\right)+\sqrt{\varepsilon}o_\varepsilon(1)
\end{align*}
for $0\leq t\leq T$ and $p\in\partial\Omega_k$.
By \eqref{eq.B.08} and \eqref{eq.B.13} with $U=u_k$, we have $u_k'^2(t)=2(F_0(\phi^*)-F_0(u_k(t)))=-2F_0(u_k(t))$ for all $t\geq0$, where $F_0(\phi)=\int_{\phi_0^*}^\phi\!f_0(s)\,\mathrm ds$ and $\phi_0^*$ is the unique zero of $f_0$.
Hence
\begin{align*}
\mathscr{T}_{\text M}(p-t\sqrt{\varepsilon}\nu_{p})=\{-F_0(u_k(t))+\sqrt\varepsilon u_k'(t)[(d-1)H(p)v_k'(t)+w_k'(t)]\}(2\nu_p\otimes\nu_p-I_d)+\sqrt\varepsilon o_\varepsilon(1),
\end{align*}
which implies
\begin{align}
\label{eq:2025.0609.1645}
\mathscr{T}_{\text{M}}(p-t\sqrt{\varepsilon}\nu_{p})\nu_p=-F_0(u_k(t))\nu_p+\sqrt{\varepsilon}u_k'(t)\{[(d-1)H(p)v_k'(t)+w_k'(t)]\nu_p+o_\varepsilon(1)\}
\end{align}
for $0\leq t\leq T$ and $p\in\partial\Omega_k$.
Note that $F_0(u_k(t))<0$ for $t\geq0$ because $F_0(\phi_0^*)=F_0'(\phi_0^*)=0$ and $F_0''(\phi)=f_0'(\phi)<0$ for $\phi\in\R$.
In \eqref{eq:2025.0609.1645}, $\mathscr{T}_{\text{M}}(p-t\sqrt{\varepsilon}\nu_p)\nu_p$ is the stress tensor $\mathscr{T}_{\text{M}}(p-t\sqrt{\varepsilon}\nu_p)$ acting on $\nu_p$, which provides the electric force at $p-t\sqrt{\varepsilon}\nu_p$ for $0\leq t\leq T$, $p\in\partial\Omega_k$, and $k\in\{0,1,\dots,K\}$.

\newpage
Throughout the paper we shall use the following notations.
\begin{itemize}
\item $\mathcal{O}_\epsilon(1)$ is a bounded quantity independent of $\epsilon$, and $o_{\epsilon}(1)$ is a small quantity tending to zero as $\epsilon$ goes to zero.
\item $\R_{+}^{d}=\{(z=(z',z^{d})\in\R^d:z^{d}>0\}$, $\partial\R_{+}^{d}=\{z=(z',z^{d})\in\R^d:z^{d}=0\}$, and $\overline\R_{+}^{d}=\R_{+}^{d}\cup\partial\R_{+}^{d}$.
\item $e_d$ stands for the unit vector in the positive direction of $z^{d}$-axis.
\item $|\Omega|$ denotes the volume of $\Omega$, and $|\partial\Omega_k|$ stands for the surface area of $\partial\Omega_k$ for $k=0,1,\dots,K$.
\item $\Omega_{k,T,\epsilon}=\{x\in\Omega\,:\,\dist(x,\partial\Omega_k)<T\sqrt\epsilon\}$, $\Omega_{k,T,\epsilon,\beta}=\{x\in\Omega\,:\,T\sqrt\epsilon<\dist(x,\partial\Omega_k)<\epsilon^\beta\}$, $\Omega_{k,\epsilon,\beta}=\{x\in\Omega\,:\,\dist(x,\partial\Omega_k)<\epsilon^\beta\}$ and $\Omega_{\epsilon,\beta}=\{x\in\Omega\,:\,\dist(x,\partial\Omega)>\epsilon^\beta\}$, where $\dist$ denotes the distance.
\item $\delta_{k}(x)=\dist(x,\partial\Omega_k)$ and $\delta(x)=\dist(x,\partial\Omega)=\min\{\delta_{0}(x),\delta_{1}(x),\dots,\delta_{K}(x)\}$.
\item $\dd\underline{\phi_{bd}}=\min_{0\leq k\leq K}\phi_{bd,k}$ and $\dd\overline{\phi_{bd}}=\max_{0\leq k\leq K}\phi_{bd,k}$.
\item $\epsilon^{*}>0$ is a sufficiently small constant, and $M>0$ is a generic constant independent of $\epsilon>0$.
\end{itemize}

The rest of the paper is organized as follows.
The proofs of \Cref{theorem:1,theorem:2} are given in \Cref{section:2,section:3}, respectively.
For the analysis of ordinary differential equations \eqref{eq.1.07}--\eqref{eq.1.12}, \eqref{eq.1.19}--\eqref{eq.1.21} and \eqref{eq.1.22}--\eqref{eq.1.27}, one may refer to \Cref{appendix:A,appendix:B,appendix:C,appendix:D}.

\noindent{\bf Acknowledgement}. The research of T.-C. Lin is partially supported by the National Center for Theoretical Sciences (NCTS) and NSTC grants 112-2115M-002-012MY3 and 113-2115-M-002-002.

\section{Proof of \texorpdfstring{\Cref{theorem:1}}{Theorem 1}}
\label{section:2}

In this section, we derive the first- and second-order terms in the asymptotic expansion of the solution $\phi_{\epsilon}$ to equation \eqref{eq.1.01} with condition \eqref{eq.1.04} near $\partial\Omega_k$ for $k=0,1,\dots,K$.
For the first-order term, we use the principal coordinate system \eqref{eq.2.01} to straighten the boundary $\partial\Omega_k$ locally and rescale spatial variables by $\sqrt\epsilon$ so that equations \eqref{eq.1.01} with condition \eqref{eq.1.04} can be transformed into \eqref{eq.2.08}--\eqref{eq.2.09} in the upper half-ball $B_{b/\sqrt\epsilon}^+$.
We then show that the solution to \eqref{eq.2.08}--\eqref{eq.2.09} converges (locally) to the solution to \eqref{eq.2.10}--\eqref{eq.2.11} in the upper half-space $\R_{+}^{d}$ (cf. \Cref{lemma:2.3}).
By the moving plane arguments on \eqref{eq.2.10}--\eqref{eq.2.11}, we prove that the solution to \eqref{eq.2.10}--\eqref{eq.2.11} depends only on the single variable $z^{d}$ (cf. \Cref{proposition:2.5}), which is the unique solution to \eqref{eq.1.07}--\eqref{eq.1.09}.
Here we have to assume that $\phi_{bd,k}$ are constants.
Hence we obtain the asymptotic expansions $\phi_{\epsilon}(p-t\sqrt{\epsilon}\nu_{p})=u_{k}(t)+o_{\epsilon}(1)$ and $\nabla\phi_{\epsilon}(p-t\sqrt{\epsilon}\nu_{p})=-\epsilon^{-1/2}(u_{k}'(t)\nu_{p}+o_{\epsilon}(1))$ for $T>0$, $p\in\partial\Omega_k$ and $0\leq t\leq T$, where $u_{k}$ is the unique solution to \eqref{eq.1.07}--\eqref{eq.1.09} and $\nu_{p}$ is the unit outer normal at $p$ (cf. \Cref{section:2.1}).

For the second-order term, we set $\varphi_{k,\epsilon}(x)=\epsilon^{-1/2}(\phi_{\epsilon}(x)-u_{k}(\delta_k(x)/\sqrt\epsilon))$ for $x\in\overline\Omega_{k,\epsilon,\beta}$ and $k=0,1,\dots,K$, where $\Omega_{k,\epsilon,\beta}=\{x\in\Omega:\delta_k(x)<\epsilon^\beta\}$ and $0<\beta<1/2$.
Then by the exponential-type estimate of $u_{k,p}$ and $\phi_{\epsilon}$ (cf. \Cref{lemma:2.4} and \Cref{proposition:2.6}), we may prove the uniform boundedness of $\varphi_{k,\epsilon}$ (cf. \Cref{proposition:2.7}).
Moreover, following similar arguments as in \Cref{section:2.1}, we use the principal coordinate system and the moving plane arguments to prove the local convergence of $\varphi_{k,\epsilon}$ and obtain $\varphi_{k,\epsilon}(p-t\sqrt{\epsilon}\nu_{p})=(d-1)H(p)v_{k}(t)+o_{\epsilon}(1)$ and $\nabla\varphi_{k,\epsilon}(p-t\sqrt{\epsilon}\nu_{p})=-\epsilon^{-1/2}[(d-1)H(p)v_{k}'(t)\nu_{p}+o_{\epsilon}(1)]$ for $T>0$, $p\in\partial\Omega_k$ and $0\leq t\leq T$ (as $\epsilon>0$ sufficiently small), where $v_{k}$ is the unique solution to \eqref{eq.1.10}--\eqref{eq.1.12} (cf. \Cref{section:2.2}).
Therefore, we obtain equations \eqref{eq.1.05}--\eqref{eq.1.06} and complete the proof of \Cref{theorem:1}.

\subsection{First-order asymptotic expansion of \texorpdfstring{$\phi_{\epsilon}$}{ϕ.ϵ}}
\label{section:2.1}

Equation \eqref{eq.1.01} with the boundary condition \eqref{eq.1.04} is the Euler--Lagrange equation of the following energy functional
\[E[\phi]=\int_\Omega\!\left(\frac{\epsilon}2|\nabla\phi|^2-F(\phi)\right)\,\mathrm{d}x+\sqrt\epsilon\sum_{k=0}^{K}\frac{1}{2\gamma_k}\int_{\partial\Omega_k}\!(\phi-\phi_{bd,k})^2\,\mathrm{d}S\quad\text{for}~\phi\in H^1(\Omega).\]
Since $F'' = f' < 0$ in $\R$, we can apply the direct method in the calculus of variations to get the unique minimizer of $E[\phi]$ over $H^1(\Omega)$ (cf. \cite{2008struwe}).
Then by the standard bootstrap argument and smoothness of $f$, we obtain the existence and uniqueness of the smooth solution $\phi\in\C^{\infty}(\overline{\Omega})$ to equation \eqref{eq.1.01} with condition \eqref{eq.1.04} (cf. \cite{1977gilbarg}).
Because the nonlinear term $f$ is strictly decreasing (cf. (A1)), we can prove the uniform boundedness of the solution $\phi_{\epsilon}$ as follows.

\begin{proposition}[Uniform boundedness of $\phi_{\epsilon}$]
\label{proposition:2.1}
For $\epsilon>0$, let $\phi_{\epsilon}\in\C^\infty(\overline\Omega)$ be the solution to equation \eqref{eq.1.01} with condition \eqref{eq.1.04}.
Then $|\phi_{\epsilon}(x)-\phi^{*}|\leq C_{1}$ for $x\in\overline\Omega$ and $\epsilon>0$, where $\phi^{*}$ is the unique zero of $f$ (cf. (A2)) and $C_{1}>0$ is a constant independent of $\epsilon$.
\end{proposition}
\begin{proof}
We first prove that $\phi_{\epsilon}(x)\leq\max\{\phi^{*},\overline{\phi_{bd}}\}$ for $x\in\overline\Omega$ and $\epsilon>0$.
Suppose that $\phi_{\epsilon}$ attains its maximum value at an interior point $x_{0}$, which implies $\Delta\phi_{\epsilon}(x_{0})\leq0$.
Then by \eqref{eq.1.01} and (A2), we obtain $f(\phi^{*})=0\leq-\epsilon\Delta\phi_{\epsilon}(x_{0})=f(\phi_{\epsilon}(x_{0}))$.
Due to the strict decrease of $f$ (cf. (A1)), we have $\phi_{\epsilon}(x_{0})\leq\phi^{*}$, which implies $\phi_{\epsilon}(x)\leq\phi^{*}$ for $x\in\overline\Omega$.
On the other hand, we suppose that $\phi_{\epsilon}$ attains its maximum value at a boundary point $x_{0}\in\partial\Omega_{k_{0}}$ for some $k_{0}\in\{0,1,\dots,K\}$, which implies $\partial_\nu\phi_{\epsilon}(x_{0})\geq0$.
By the boundary condition \eqref{eq.1.04} with $\gamma_{k_{0}}>0$, we get
\[\phi_{\epsilon}(x)\leq\phi_{\epsilon}(x_{0})=\phi_{bd,k_{0}}-\gamma_{k_{0}}\sqrt\epsilon\partial_\nu\phi_{\epsilon}(x_{0})\leq\phi_{bd,k_{0}}\leq\overline{\phi_{bd}}:=\max_{0\leq k\leq K}\phi_{bd,k}\quad\text{for}~x\in\overline\Omega.\]
Thus, $\phi_{\epsilon}(x)\leq\max\{\phi^{*},\overline{\phi_{bd}}\}$ for $x\in\overline\Omega$ and $\epsilon>0$.
Similarly, we can get $\phi_{\epsilon}(x)\geq\min\{\phi^{*},\underline{\phi_{bd}}\}$ for $x\in\overline\Omega$ and $\epsilon>0$, where $\underline{\phi_{bd}}:=\min\limits_{0\leq k\leq K}\phi_{bd,k}$.
Hence $\min\{\phi^{*},\underline{\phi_{bd}}\}\leq\phi_{\epsilon}(x)\leq\max\{\phi^{*},\overline{\phi_{bd}}\}$, i.e., $|\phi_{\epsilon}(x)-\phi^{*}|\leq C_{1}$ for $x\in\overline\Omega$ and $\epsilon>0$, where $C_{1}>0$ is a constant independent of $\epsilon>0$.
Therefore, we complete the proof of \Cref{proposition:2.1}.
\end{proof}

Now we introduce the principal coordinate system and the associated diffeomorphism $\Psi_p$, which straightens the boundary portion near each point $p\in\partial\Omega_k$ for $k\in\{0,1,\dots,K\}$.
Fix $k\in\{0,1,\dots,K\}$ arbitrarily.
After a suitable rotation of the coordinate system, we may write the boundary point as $p=(p',p^d)\in\R^{d-1}\times\R$ and assume that the inner normal to $\partial\Omega_k$ at $p$ points in the direction of the positive $x^d$-axis.
To straighten the portion of $\partial\Omega_k$, we employ the principal coordinate system $y=(y',y^d)$ (cf. \cite[Section 14.6]{1977gilbarg}).
Thus, there exist a neighborhood $\mathcal{N}_p$ of $p$, a constant $b_p'>0$ (depending on $p$ and $\partial\Omega_k$) and a smooth function $\psi_p=\psi_p(x')$ (defined for $|x'-p'|<b_p'$) satisfying
\begin{enumerate}
\item[(D1)] $\psi_p(p')=p^d$,
\item[(D2)] $\nabla\psi_p(p')=0$,
\item[(D3)] $\partial\Omega_k\cap\mathcal{N}_p=\{x=(x',x^d)\in\R^d:x^d=\psi_p(x'),\,|x'-p'|<b_p'\}$,
\item[(D4)] $\mathcal{N}_p^+:=\Omega\cap\mathcal{N}_p=\{x=(x',x^d)\in\R^d:x^d>\psi_p(x'),\,|x'-p'|<b_p'\}\cap\mathcal{N}_p$ and $\overline{\mathcal{N}_p}^+=\Psi_p(\overline B_b^+)$.
\end{enumerate}
Here the diffeomorphism $\Psi_p:\overline{B}_b^+\to\R^d$ is defined by
\begin{align}
\label{eq.2.01}
&x=(x',x^d)=\Psi_p(y)=(p'+y',\psi_p(p'+y'))-y^d\nu\quad\text{for}~y=(y',y^d)\in\overline{B}_b^+,
\end{align}
where $\overline{B}_b^+=\{y=(y',y^d)\in\R^d\,:\,|y|\leq b,\,y^d\geq0\}$, $b$ is a positive constant (depending on $\Omega$) and $\nu=(\nu',\nu^d)=(\nabla\psi_p(p'+y'),-1)/\sqrt{1+|\nabla\psi_p(p'+y')|^2}$ is the unit outer normal at $(p'+y',\psi_p(p'+y'))$ with respect to $\Omega$ (see \Cref{figure:3}). 
Moreover, $y^d=\dist(x,\partial\Omega_k)$ and $(y',y^d)=\Psi_p^{-1}(x',x^d)$ for $x=(x',x^d)\in\overline{\mathcal{N}}_p^+$.

\begin{figure}[!htb]\centering\includegraphics[scale=0.29]{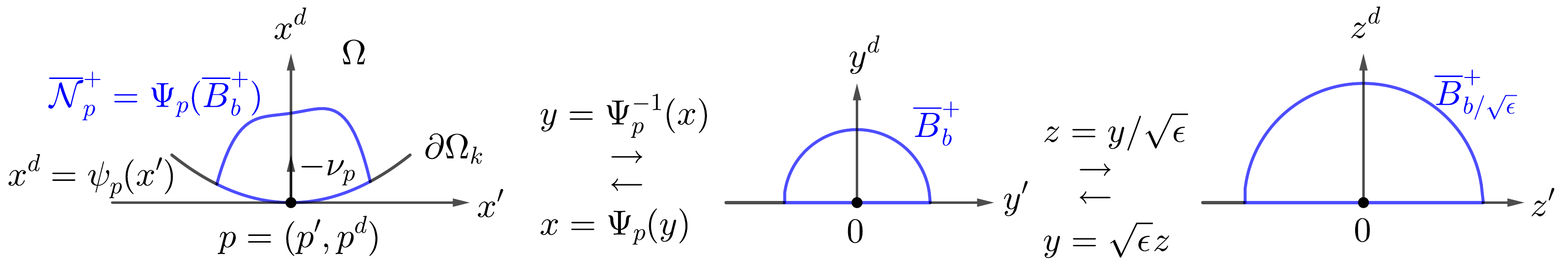}
\caption{The function $\psi_p$ describes the portion of $\partial\Omega_k$ near $p\in\partial\Omega_k$.
The diffeomorphism $x=\Psi_p(y)$ maps an upper half ball $\overline{B}_b^+$ to the neighborhood $\overline{\mathcal{N}}_p^+$, and $z=y/\sqrt\epsilon$ dilates $\overline{B}_b^+$ into $\overline{B}_{b/\sqrt\epsilon}^+$.}
\label{figure:3}
\end{figure}

By \eqref{eq.2.01} and rescaling $y$ by $\sqrt{\epsilon}$, we obtain
\begin{lemma}
\label{lemma:2.2}
For $\phi\in\C^\infty(\overline{\Omega})$, $p\in\partial\Omega=\bigcup_{k=0}^K\partial\Omega_k$ and $\epsilon>0$, let $u_{p,\epsilon}:\overline{B}_{b/\sqrt\epsilon}^+\to\R$ satisfy $u_{p,\epsilon}(z)=\phi(\Psi_p(\sqrt\epsilon z))$ for $z\in\overline B_{b/\sqrt\epsilon}^+$ (see \Cref{figure:3}).
Then $u_{p,\epsilon}\in\C^\infty(\overline{B}_{b/\sqrt\epsilon}^+)$, and
\begin{align}
\label{eq.2.02}
&\epsilon\Delta_x\phi(\Psi_p(\sqrt{\epsilon}z))=\sum_{i,j=1}^da_{ij}(z)\frac{\partial^2u_{p,\epsilon}}{\partial z^{i}\partial z^j}(z)+\sum_{j=1}^db_j(z)\frac{\partial u_{p,\epsilon}}{\partial z^j}(z)\quad\text{for}~z\in B_{b/\sqrt\epsilon}^+,\\
\label{eq.2.03}
&\partial_{\nu_{p}}\phi(\Psi_p(\sqrt{\epsilon}(z^{d}e_d)))=-\epsilon^{-1/2}\partial_{z^{d}}u_{p,\epsilon}(z^{d}e_d)\quad~\text{for}~0\leq z^{d}<b/\sqrt\epsilon,\\
\label{eq.2.04}
&\phi(p-t\sqrt{\epsilon}\nu_{p})=u_{p,\epsilon}(te_d)\quad\text{for}~0\leq t<b/\sqrt\epsilon,
\end{align}
where $e_d$ and $\nu_{p}$ are the unit vector in the positive direction of $z^{d}$-axis and the unit outer normal at $p$ with respect to $\Omega$, respectively. The coefficients $a_{ij}$ and $b_j$ satisfy
\begin{align}
\label{eq.2.05}
a_{ij}(z)&=\delta_{ij}+2\psi_{ij}\sqrt{\epsilon}z^{d}+\epsilon|z|^2\mathcal{O}(1)\quad\text{for}~1\leq i,j\leq d~\text{and}~z\in B_{b/\sqrt\epsilon}^+,\\
\label{eq.2.06}
b_j(z)&=-\delta_{jd}\sqrt\epsilon(d-1)H(p)+\epsilon|z|\mathcal{O}(1)\quad\text{for}~1\leq j\leq d~\text{and}~z\in B_{b/\sqrt\epsilon}^+.
\end{align}
Here $\delta_{ij}=\begin{cases}1&\text{if}~i=j\\0&\text{if}~i\neq j\end{cases}$ is the Kronecker symbol, $\mathcal{O}(1)$ is a bounded quantity independent of $\epsilon$ and $z$, and
\begin{align*}
&\psi_{ij}=\psi_{ji}:=\partial_{x^ix^j}\psi_p(p')\quad\text{for}~1\leq i,j\leq d-1,\\
&\psi_{id}=\psi_{di}:=\frac12\sum_{n=1}^{d-1}\psi_{in}\quad\text{for}~1\leq i\leq d-1,\\
&\psi_{dd}:=0.
\end{align*}
Besides, $H(p)=(d-1)^{-1}\sum_{i=1}^{d-1}\psi_{ii}$ stands for the mean curvature at the boundary point $p$ with respect to $\Omega$ due to (D2) (see \cite[(14.92)]{1977gilbarg}).
\end{lemma}
\noindent Note that \eqref{eq.2.05}--\eqref{eq.2.06} imply that $a_{ij}$ and $b_j$ satisfy
\begin{align*}
\lim_{\epsilon\to0^+}a_{ij}(z)=\delta_{ij}\quad\lim_{\epsilon\to0^+}b_j(z)=0\quad\text{for}~R>0,\,1\leq i,j\leq d,~|z|<R~\text{and}~z^{d}>0.
\end{align*}
The proof of \Cref{lemma:2.2} follows standard calculations of the principal coordinate system (cf. \cite{1977gilbarg}) so we omit it here.

\newpage

To get the first-order asymptotic expansion of $\phi_{\epsilon}$ near the boundary $\partial\Omega_k$, we define
\begin{align}
\label{eq.2.07}
u_{k,p,\epsilon}(z)=\phi_{\epsilon}(\Psi_p(\sqrt\epsilon z))\quad\text{for}~z\in\overline{B}_{b/\sqrt\epsilon}^+,~p\in\partial\Omega_k~\text{and}~k=0,1,\dots,K,\end{align}
where $\phi_{\epsilon}$ is the solution to equation \eqref{eq.1.01} with condition \eqref{eq.1.04}.
Then by \eqref{eq.2.02}, \eqref{eq.2.03}, and \eqref{eq.2.07}, $u_{k,p,\epsilon}$ satisfies
\begin{alignat}{2}
\label{eq.2.08}
\sum_{i,j=1}^da_{ij}(z)\frac{\partial^2u_{k,p,\epsilon}}{\partial z^{i}\partial z^j}+\sum_{j=1}^db_j(z)\frac{\partial u_{k,p,\epsilon}}{\partial z^j}+f(u_{k,p,\epsilon})&=0&&\quad\text{in}~B_{b/\sqrt\epsilon}^+,\\
\label{eq.2.09}
u_{k,p,\epsilon}-\gamma_k\partial_{z^{d}}u_{k,p,\epsilon}&=\phi_{bd,k}&&\quad\text{on}~\overline{B}_{b/\sqrt\epsilon}^+\cap\partial\R_{+}^{d},
\end{alignat}
where $\partial\R_{+}^{d}=\{z=(z',z^{d})\in\R^d:z^{d}=0\}$, $a_{ij}$ and $b_j$ are given in \eqref{eq.2.05}--\eqref{eq.2.06}.
Based on the uniform boundedness of $\phi_{\epsilon}$ (cf. \Cref{proposition:2.1}), together with the standard elliptic $L^q$-estimate and Schauder's estimate, we obtain the $\C_{\text{loc}}^{2,\alpha}$-estimate of $u_{k,p,\epsilon}$.
Using a diagonal process, we can prove the convergence of $u_{k,p,\epsilon}$ in $\C^{2,\alpha}(\overline{B}_m^+)$ as $\epsilon$ tends to zero (up to a subsequence) for $m\in\N$.
The result is stated below.
\begin{lemma}
\label{lemma:2.3}
For any sequence $\{\epsilon_{n}\}_{n=1}^\infty$ of positive numbers with $\dd\lim_{n\to\infty}\epsilon_{n}=0$, $\alpha\in(0,1)$ and $p\in\partial\Omega_k$ ($k\in\{0,1,\dots,K\}$), there exists a subsequence $\{\epsilon_{nn}\}_{n=1}^\infty$ such that $\dd\lim_{n\to\infty}\|u_{k,p,\epsilon_{nn}}-u_{k,p}\|_{\C^{2,\alpha}(\overline{B}_m^+)}=0$ for $m\in\N$, where $\overline{B}_m^+=\{z=(z',z^{d})\in\R^d\,:\,|z|\leq m,\,z^{d}\geq0\}$ and $u_{k,p}\in\C^{2,\alpha}_{\text{loc}}(\overline\R_{+}^{d})$ satisfies
\begin{alignat}{2}
\label{eq.2.10}
\Delta u_{k,p}+f(u_{k,p})&=0\quad&&\text{in}~\R_{+}^{d},\\
\label{eq.2.11}
u_{k,p}-\gamma_k\partial_{z^{d}}u_{k,p}&=\phi_{bd,k}\quad&&\text{on}~\partial\R_{+}^{d}.
\end{alignat}
Here $\R_{+}^{d}=\{z=(z',z^{d})\in\R^{d}\,:\,z^{d}>0\}$.
\end{lemma}

For the equation \eqref{eq.2.10} together with the boundary condition \eqref{eq.2.11}, we can establish the following exponential-type estimate of $u_{k,p}$.
\begin{lemma}
\label{lemma:2.4}
For $p\in\partial\Omega_k$ and $k\in\{0,1,\dots,K\}$, the solution $u_{k,p}$ to \eqref{eq.2.10}--\eqref{eq.2.11} satisfies
\begin{align}
\label{eq.2.12}
|u_{k,p}(z)-\phi^{*}|\leq2C_{1}\exp(-C_{2}z^{d})\quad\text{for}~z=(z',z^{d})\in\overline\R_{+}^{d}~\text{and}~z^{d}\geq(d-1)/(4C_{2}),
\end{align}
where $C_{1}>0$ is a constant (independent of $\epsilon$) given in \Cref{proposition:2.1} such that $|\phi_{\epsilon}(x)-\phi^{*}|\leq C_{1}$ for $x\in\overline\Omega$, and $C_{2}=m_f([\phi^{*}-C_{1},\phi^{*}+C_{1}])/8>0$.
\end{lemma}
\begin{proof}Let $\overline{z}=(\overline{z}',\overline{z}^{d})\in\overline\R_{+}^{d}$ arbitrarily with $\overline{z}^{d}\geq(d-1)/(4C_{2})$.
Then let $u^\pm$ be the solution to
\begin{alignat}{2}
\label{eq.2.13}
\Delta u^{\pm}+f(u^{\pm})&=0\quad&&\text{in}~B_{\overline{z}^{d}}(\overline{z}),\\
\label{eq.2.14}
u^{\pm}&=\phi^{*}\pm C_{1}\quad&&\text{on}~\partial B_{\overline{z}^{d}}(\overline{z}).
\end{alignat}
We claim that
\begin{align}
\label{eq.2.15}
u^{-}\leq u_{k,p}\leq u^{+}\quad\text{in}~\overline B_{\overline{z}^{d}}(\overline{z}).
\end{align}
We first prove $u^{+}\geq u_{k,p}$ and set $\overline{u}_{k,p}^{+}:=u^{+}-u_{k,p}$ in $\overline{B}_{\overline{z}^{d}}(\overline{z})$.
Then by \eqref{eq.2.10} and \eqref{eq.2.13}--\eqref{eq.2.14}, $\overline{u}_{k,p}^{+}$ satisfies
\begin{align}
\label{eq.2.16}
&\Delta\overline{u}_{k,p}^{+}+c(z)\overline{u}_{k,p}^{+}=0\quad\text{in}~B_{\overline{z}^{d}}(\overline{z}),\\
\label{eq.2.17}
&\overline{u}_{k,p}^{+}=C_{1}-(u_{k,p}-\phi^{*})\geq0\quad\text{on}~\partial B_{\overline{z}^{d}}(\overline{z}),\end{align}
where $c(z)=(f(u^{+}(z))-f(u_{k,p}(z)))/(u^{+}(z)-u_{k,p}(z))$ if $u^+(z)\neq u_{k,p}(z)$; $c(z)=f'(u_{k,p}(z))$ if $u^+(z)=u_{k,p}(z)$.
Here we have used \eqref{eq.2.07}, \Cref{proposition:2.1} and \Cref{lemma:2.3}.
Note that $c(z)<0$ for $z\in B_{\overline{z}^{d}}(\overline{z})$ because $f$ is strictly decreasing on $\R$ (cf. (A1)).
Thus, we apply the maximum principle to \eqref{eq.2.16}--\eqref{eq.2.17} and obtain $\overline u_{k,p}^+\geq0$ in $\overline B_{\overline{z}^{d}}(\overline{z})$, which implies $u^+\geq u_{k,p}$ in $\overline B_{\overline{z}^{d}}(\overline{z})$.
Similarly, we can prove $u^-\leq u_{k,p}$ in $B_{\overline{z}^{d}}(\overline{z})$, and get \eqref{eq.2.15}.
Since $u^\pm$ are the radial solutions to \eqref{eq.2.13}--\eqref{eq.2.14}, then by \Cref{proposition:A.2} in \Cref{appendix:A} with $\epsilon=1$ and replacing $B_R(0)$, $r$, $\Phi_\epsilon$, $\Phi_{bd}$, $m_f$ by $B_{\overline{z}^{d}}(\overline{z})$, $|z-\overline{z}|$, $u^\pm$, $\phi^{*}\pm C_{1}$, we obtain
\begin{align}
\label{eq.2.18}
|u^\pm(z)-\phi^{*}|\leq2C_{1}\exp(-C_{2}(\overline{z}^{d}-|z-\overline{z}|))\quad\text{for}~z\in B_{\overline{z}^{d}}(\overline{z}).
\end{align}
Here we have used the fact that $\overline{z}^{d}\geq(d-1)/(4C_{2})$.
Along with \eqref{eq.2.15} and \eqref{eq.2.18}, we obtain
\[|u_{k,p}(\overline{z})-\phi^{*}|=|u_{k,p}(\overline{z}',\overline{z}^{d})-\phi^{*}|\leq|u^\pm(\overline{z}',\overline{z}^{d})-\phi^{*}|\leq2C_{1}\exp(-C_{2}\overline{z}^{d})\]
for $\overline{z}\in\R_{+}^{d}$ and $\overline{z}^{d}\geq(d-1)/(4C_{2})$, which gives \eqref{eq.2.12}.
Therefore, we complete the proof of \Cref{lemma:2.4}.
\end{proof}

\newpage

From \Cref{lemma:2.3}, the solution $u_{k,p}$ to \eqref{eq.2.10}--\eqref{eq.2.11} may, a priori, depend on the sequence $\{\epsilon_{n}\}_{n=1}^\infty$ and on the point $p\in\partial\Omega_k$.
To establish the independence from the sequence and the point, we employ the moving plane arguments to prove that $u_{k,p}$ satisfies $u_{k,p}(z)=u_{k}(z^{d})$ for $z=(z',z^{d})\in\overline\R_{+}^{d}$ (see \Cref{proposition:2.5}), where $u_{k}$ is the unique solution to \eqref{eq.1.07}--\eqref{eq.1.09}, and $u_{k}$ is independent of the sequence $\{\epsilon_{n}\}_{n=1}^{\infty}$ and the point $p\in\partial\Omega_{k}$. Consequently, we improve the results of \Cref{lemma:2.3} as
\begin{align}
\label{eq.2.19}
\lim_{\epsilon\to0^+}\|u_{k,p,\epsilon}-u_{k}\|_{\C^{2,\alpha}(\overline B_m^+)}=0\quad\text{for}~m\in\N~\text{and}~\alpha\in(0,1),
\end{align}
and
\begin{align}
\label{eq.2.20}
\lim_{\epsilon\to0^+}u_{k,p,\epsilon}(z)=u_{k}(z^{d})\quad\text{for}~z=(z',z^{d})\in\overline\R_{+}^{d},
\end{align}
where $u_{k,p,\epsilon}$ is defined in \eqref{eq.2.07}. Below are the details.
\begin{proposition}
\label{proposition:2.5}
For $p\in\partial\Omega_k$ and $k\in\{0,1,\dots,K\}$, the solution $u_{k,p}$ to \eqref{eq.2.10}--\eqref{eq.2.11} satisfies
\begin{itemize}
\item[(a)] $u_{k,p}$ depends only on the variable $z^{d}$, i.e., $u_{k,p}(z)=u_{k,p}(z^{d})$ for $z=(z',z^{d})\in\overline\R_{+}^{d}$.
\item[(b)] $u_{k,p}$ is independent of $p$ and depends only on $k$, i.e., $u_{k,p}(z^{d})=u_{k}(z^{d})$ for $z^{d}\in[0,\infty)$, where $u_{k}$ is the unique solution to \eqref{eq.1.07}--\eqref{eq.1.09}.
\end{itemize}
\end{proposition}
\begin{proof}
To prove (a), we fix $p\in\partial\Omega_k$, $k\in\{0,1,\dots,K\}$ and $h\in\R^{d-1}-\{0\}$ arbitrarily.
For notational convenience, we denote $u_{k,p}$ by $u$ in the proof of (a).
To use the moving plane arguments, we define $\tilde u$ by
\begin{align}
\label{eq.2.21}
\tilde u(z)=u(z'+h,z^{d})-u(z)\quad\text{for}~z=(z',z^{d})\in\overline\R_{+}^{d}.
\end{align}
For the proof of (a), it suffices to show $\tilde u\equiv0$, which is equivalent to showing $\dd\sup_{\overline\R_{+}^{d}}\tilde u=0$ and $\dd\inf_{\overline\R_{+}^{d}}\tilde u=0$.
We first prove $\dd\sup_{\overline\R_{+}^{d}}\tilde u=0$.
Suppose by contradiction that
\begin{align*}
\zeta:=\sup_{\overline\R_{+}^{d}}\tilde u>0.
\end{align*}
By \eqref{eq.2.12} and \eqref{eq.2.21}, we have
\begin{align*}
|\tilde u(z)|\leq|u(z'+h,z^{d})-\phi^{*}|+|u(z)-\phi^{*}|\leq4C_{1}\exp(-C_{2}z^{d})
\end{align*}
for $z=(z',z^{d})\in\overline\R_{+}^{d}$ and $z^{d}\geq(d-1)/(4C_{2})$.
Clearly, there exists $L>(d-1)/(4C_{2})$ such that $4C_{1}\exp(-C_{2}L)\leq\zeta/2$ and
\begin{align}
\label{eq.2.22}
|\tilde u(z)|\leq\zeta/2\quad\text{for}~z=(z',z^{d})\in\R_{+}^{d}\quad\text{and}\quad z^{d}\geq L.\end{align}
Note that $\tilde u$ (defined in \eqref{eq.2.21}) may depend on $h$, but $L$ is independent of $h$ (because $C_{1}$ and $C_{2}$ are independent of $h$).
Thus
\begin{align}
\label{eq.2.23}
\zeta=\sup_{\overline\R_{+}^{d}}\tilde u=\sup_{\R^{d-1}\times[0,L]}\tilde u>0.
\end{align}
Then from \eqref{eq.2.10}--\eqref{eq.2.11}, $\tilde u$ satisfies
\begin{alignat}{2}
\label{eq.2.24}
\Delta\tilde u+c_1(z)\tilde u&=0\quad&&\text{in}~\R_{+}^{d},\\
\label{eq.2.25}
\tilde u-\gamma_k\partial_{z^{d}}\tilde u&=0\quad&&\text{on}~\partial\R_{+}^{d},
\end{alignat}
where
\[c_1(z)=\begin{cases}
\dd\frac{f(u(z'+h,z^{d}))-f(u(z))}{u(z'+h,z^{d})-u(z)}&\text{if}~u(z'+h,z^{d})\neq u(z),\\
f'(u(z))&\text{if}~u(z'+h,z^{d})=u(z).
\end{cases}\]
Note that $c_1(z)<0$ for $z\in\R_{+}^{d}$ because of the strict decrease of $f$ (cf. (A1)).
By the strong maximum principle applied to \eqref{eq.2.24} in $\R^{d-1}\times(0,L)$ (cf. \cite[Theorem 3.5]{1977gilbarg}), $\tilde u$ has no interior maximum point in $\R^{d-1}\times(0,L)$.
Moreover, by \eqref{eq.2.22} and \eqref{eq.2.23}, the maximum point of $\tilde u$ cannot be located on $\R^{d-1}\times\{L\}$ .
On the other hand, by \eqref{eq.2.25}, the maximum point of $\tilde u$ cannot be located on $\partial \R_{+}^{d}$ because $\partial_{z^{d}}\tilde u(z',0)=\tilde u(z',0)/\gamma_k=\zeta/\gamma_k>0$ if $(z',0)$ is the maximum point of $\tilde u$, i.e., $\tilde u(z',0)=\zeta$.
Hence there exists a sequence $z_n=(z_n',z_n^d)\in\R^{d-1}\times[0,L]$ with $\dd\lim_{n\to\infty}|z'_n|=\infty$ and
\begin{align}
\label{eq.2.26}
\lim_{n\to\infty}z_n^d=z_\infty^d\in[0,L]
\end{align}such that
\begin{align}
\label{eq.2.27}
\lim_{n\to\infty}\tilde u(z_n',z_n^d)=\zeta.
\end{align}
Let
\begin{align}
\label{eq.2.28}
u_n(z)=u(z'+z_n',z^{d})\quad\text{for}~n\in\N~\text{and}~z=(z',z^{d})\in\overline\R_{+}^{d}.
\end{align}
By the translation invariance (to $z'$) of \eqref{eq.2.10}--\eqref{eq.2.11}, $u_n$ satisfies
\begin{alignat*}{2}
\Delta u_n+f(u_n)&=0\quad&&\text{in}~\R_{+}^{d},\\
u_n-\gamma_k\partial_{z^{d}}u_n&=\phi_{bd,k}\quad&&\text{on}~\partial\R_{+}^{d}.
\end{alignat*}
Note that by \eqref{eq.2.07}, \Cref{proposition:2.1}, and \Cref{lemma:2.3}, $\|u_n\|_{L^\infty(\R_{+}^{d})}\leq C_{1}+|\phi^{*}|$ for $n\in\N$.
Then similar to \Cref{lemma:2.3}, we can apply the elliptic $L^q$-estimate, Sobolev's compact embedding theorem, Schauder's estimate and the diagonal process to get a subsequence $\{u_{n_{jj}}\}_{j=1}^\infty$ of $\{u_n\}_{n=1}^\infty$ such that
\begin{align}
\label{eq.2.29}
\lim_{j\to\infty}\|u_{n_{jj}}-u_\infty\|_{\C^{2,\alpha}(\overline B_m^+)}=0\quad\text{for}~m\in\N~\text{and}~\alpha\in(0,1),
\end{align}
where $u_\infty\in\C^{2,\alpha}_{\text{loc}}(\overline\R_{+}^{d})$ satisfies
\begin{alignat}{2}
\label{eq.2.30}
\Delta u_\infty+f(u_\infty)&=0\quad&&\text{in}~\R_{+}^{d},\\
\label{eq.2.31}
u_\infty-\gamma_k\partial_{z^{d}}u_\infty&=\phi_{bd,k}\quad&&\text{on}~\partial\R_{+}^{d}.
\end{alignat}

Now we derive the contradiction and conclude that $\dd\zeta=\sup_{\overline\R_{+}^{d}}\tilde u=\sup_{\R^{d-1}\times[0,L]}\tilde u=0$ (cf. \eqref{eq.2.23}). Let
\begin{align}
\label{eq.2.32}
\tilde u_\infty(z)=u_\infty(z'+h,z^{d})-u_\infty(z)\quad\text{for}~z=(z',z^{d})\in\R^{d-1}\times[0,L].
\end{align}
From \eqref{eq.2.23}, we have $\zeta\geq\tilde u(z)$ for $z\in\R^{d-1}\times[0,L]$.
Then by \eqref{eq.2.21} and \eqref{eq.2.28} with $n=n_{jj}$, we get
\begin{align}
\label{eq.2.33}
\zeta\geq u_{n_{jj}}(z'+h,z^{d})-u_{n_{jj}}(z)\quad\text{for}~z\in\R^{d-1}\times[0,L]~\text{and}~j\in\N.
\end{align}
We prove that $\zeta$ is the maximum value of $\tilde u_\infty$ as follows.
\begin{claim}
\label{claim:1}
$\tilde u_\infty$ attains its maximum value $\zeta$ at $(0,z_\infty^d)\in\R^{d-1}\times[0,L]$.
\end{claim}
\begin{proof}[Proof of \Cref{claim:1}]
Fix $z=(z',z^{d})\in\R^{d-1}\times[0,L]$ arbitrarily. Then there exists $m_1\in\N$ such that $z\in\overline B_{m_1}^+$ and $(z'+h,z^{d})\in\overline B_{m_1}^+$.
By \eqref{eq.2.29} with $m=m_1$, we have
\begin{align}
\label{eq.2.34}
\lim_{j\to\infty}u_{n_{jj}}(z'+h,z^{d})=u_\infty(z'+h,z^{d})\quad\text{and}\quad\lim_{j\to\infty}u_{n_{jj}}(z)=u_\infty(z).
\end{align}
Taking $j\to\infty$ in \eqref{eq.2.33}, we use \eqref{eq.2.32} and \eqref{eq.2.34} to obtain
\begin{equation}
\label{eq.2.35}
\zeta\geq u_\infty(z'+h,z^{d})-u_\infty(z)=\tilde u_\infty(z).
\end{equation}
By \eqref{eq.2.26}, $\dd\lim_{j\to\infty}z_{n_{jj}}^d=z_\infty^d\in[0,L]$ and hence there exists a positive integer $m_2$ such that $(h,z_{n_{jj}}^d)\in\overline B_{m_2}^+$ and $(0,z_{n_{jj}}^d)\in\overline B_{m_2}^+$ for $j\in\N$. Then we use \eqref{eq.2.29} with $m=m_2$ to get
\begin{align}
\label{eq.2.36}
\lim_{j\to\infty}u_{n_{jj}}(h,z_{n_{jj}}^d)=u_\infty(h,z_\infty^d)\quad\text{and}\quad\lim_{j\to\infty}u_{n_{jj}}(0,z_{n_{jj}}^d)=u_\infty(0,z_\infty^d).
\end{align}
By \eqref{eq.2.21}, \eqref{eq.2.27} and \eqref{eq.2.28} with $n=n_{jj}$,
\[\zeta=\lim_{j\to\infty}[u_{n_{jj}}(h,z_{n_{jj}}^d)-u_{n_{jj}}(0,z_{n_{jj}}^d)],\]
and then along with \eqref{eq.2.32}, \eqref{eq.2.35} and \eqref{eq.2.36}, we have
\[\zeta=u_\infty(h,z_\infty^d)-u_\infty(0,z_\infty^d)=\tilde u_\infty(0,z_\infty^d),\]
which completes the proof of \Cref{claim:1}.\end{proof}
\noindent By \eqref{eq.2.30}--\eqref{eq.2.32}, $\tilde u_\infty$ satisfies
\begin{alignat}{2}
\label{eq.2.37}
\Delta\tilde u_\infty+c_2(z)\tilde u_\infty&=0\quad&&\text{in}~\R^{d-1}\times(0,L),\\
\label{eq.2.38}
\tilde{u}_{\infty}-\gamma_{k}\partial_{z^{d}}\tilde{u}_{\infty}&=0\quad&&\text{on}~\partial\R_{+}^{d},
\end{alignat}
where
\[c_{2}(z)=\begin{cases}\dd\frac{f(u_\infty(z'+h,z^{d}))-f(u_{\infty}(z))}{u_{\infty}(z'+h,z^{d})-u_\infty(z)}&\text{if}~u_\infty(z'+h,z^{d})\neq u_{\infty}(z),\\f'(u_{\infty}(z))&\text{if}~u_\infty(z'+h,z^{d})=u_\infty(z).
\end{cases}\]
Note that $c_{2}(z)<0$ for $z\in\R_{+}^{d}$ because of (A1).
Then by the strong maximum principle to \eqref{eq.2.37} (cf. \cite[Theorem 3.5]{1977gilbarg}), the maximum point $(0,z_{\infty}^{d})$ of $\tilde u_\infty$ cannot be an interior point of $\R^{d-1}\times(0,L)$.
Hence either $z_{\infty}^{d}=0$ or $z_{\infty}^{d}=L$.
From \eqref{eq.2.22}, \eqref{eq.2.28} with $n=n_{jj}$, \eqref{eq.2.32} and \eqref{eq.2.36}, we have $\tilde{u}_{\infty}(0,L)\leq\zeta/2$.
Thus, by \Cref{claim:1}, $z_{\infty}^{d}\neq L$, $z_{\infty}^{d}=0$, $\tilde{u}_{\infty}(0,0)=\zeta$ and $(0,0)$ is the maximum point of $\tilde{u}_{\infty}$ in $\R^{d-1}\times[0,L]$.
However, \eqref{eq.2.38} gives
\[0<\zeta=\tilde{u}_{\infty}(0,0)=\gamma_k\partial_{z^{d}}\tilde{u}_{\infty}(0,0)\leq0,\]
which leads to a contradiction.
Hence we obtain $\zeta=\dd\sup_{\overline{\R}_{+}^{d}}\tilde u=0$.
Similarly, we can prove $\dd\inf_{\overline{\R}_{+}^{d}}\tilde u=0$ by contradiction. Suppose that $\dd\zeta':=-\inf_{\overline{\R}_{+}^{d}}\tilde{u}>0$. Then by \eqref{eq.2.12} and \eqref{eq.2.21}, we have $|\tilde{u}(z)|\leq\zeta'/2$ for $z=(z',z^{d})\in\R_{+}^{d}$ and $z^{d}\geq L$ for some $L>0$ independent of $h$, and thus $\dd\zeta'=-\inf_{{\R}^{d-1}\times[0,L]}\tilde u$.
Arguing as in \eqref{eq.2.24}--\eqref{eq.2.36}, we can prove $\tilde u_\infty$ attains its minimum value at $(0,z_{\infty}^{d})$, where $z_\infty^d\in[0,L]$. Then we can apply the maximum principle to \eqref{eq.2.37}--\eqref{eq.2.38} to get a contradiction and complete the proof of (a).

To complete the proof, it remains to prove (b). By (a), $u_{k,p}=u_{k,p}(z^{d})$ solves the ordinary differential equation \eqref{eq.1.07} with the initial condition \eqref{eq.1.08}. Moreover, by the exponential-type estimate \eqref{eq.2.12}, $u_{k,p}$ satisfies the asymptotic formula \eqref{eq.1.09}.
Note that the equations \eqref{eq.1.07}--\eqref{eq.1.09} are independent of $p$.
Hence by the uniqueness of \eqref{eq.1.07}--\eqref{eq.1.09} (cf. \Cref{appendix:B}), $u_{k,p}$ is independent of $p$; we therefore denote it by $u_{k}$ for $k\in\{0,1,\dots,K\}$, which finalizes the proof of (b). Therefore, we complete the proof of \Cref{proposition:2.5}.
\end{proof}

Using \Cref{proposition:2.5}, we prove \eqref{eq.2.19} and \eqref{eq.2.20}.
Let $T>0$, $k\in\{0,1,\dots,K\}$, and $p\in\partial\Omega_k$.
Then by \eqref{eq.2.04}, \eqref{eq.2.07}, and \eqref{eq.2.19}--\eqref{eq.2.20}, we find the first-order asymptotic expansions of $\phi_{\epsilon}$ and $\nabla\phi_{\epsilon}$:
\begin{align}
\label{eq.2.39}
\phi_{\epsilon}(p-t\sqrt{\epsilon}\nu_{p})=u_{k,p,\epsilon}(te_{d})=u_{k}(t)+o_{\epsilon}(1),\\
\label{eq.2.40}
\nabla\phi_{\epsilon}(p-t\sqrt{\epsilon}\nu_{p})=-\frac{1}{\sqrt{\epsilon}}(u_{k}'(t)\nu_{p}+o_{\epsilon}(1))
\end{align}
for $0\leq t\leq T$ as $\epsilon\to0^+$.

To derive the second-order asymptotic expansions of $\phi_{\epsilon}$ and $\nabla\phi_{\epsilon}$ and capture the asymptotic behavior of the term $o_{\epsilon}(1)/\sqrt{\epsilon}$ in the next section, it is crucial to establish the following exponential-type estimates of $\phi_{\epsilon}$ and $\nabla\phi_{\epsilon}$:
\begin{proposition}[Exponential-type estimates of $\phi_{\epsilon}$ and $\nabla\phi_{\epsilon}$]
\label{proposition:2.6}
Let $\Omega$ satisfy the uniform interior sphere condition, i.e., there exists $R>0$ such that $B_R(p-R\nu_{p})\subseteq\Omega$ and $\partial B_R(p-R\nu_{p})\cap\partial\Omega=\{p\}$ for $p\in\partial\Omega$, where $\nu_{p}$ is the unit outer normal to $\partial\Omega$ at $p$.
Then we have
\begin{align}
\label{eq.2.41}
&|\phi_{\epsilon}(x)-\phi^{*}|\leq2C_{1}\exp\left(-\frac{C_{2}\delta(x)}{\sqrt{\epsilon}}\right),\\
\label{eq.2.42}
&|\nabla\phi_{\epsilon}(x)|\leq\frac{M}{\sqrt\epsilon}\exp\left(-\frac{C_{2}\delta(x)}{\sqrt{2\epsilon}}\right)\end{align}
for $x\in\overline\Omega$ and $0<\epsilon<\epsilon^{*}$, where $C_{1}$ and $C_{2}$ are positive constants given in \Cref{proposition:2.1} and \Cref{lemma:2.4}, respectively.
Here $\epsilon^{*}$ is a sufficiently small constant, $M>0$ is a positive constant independent of $\epsilon$, and $\delta(x)=\dist(x,\partial\Omega)$ for $x\in\overline\Omega$.
\end{proposition}
\noindent Hereafter $\epsilon^{*}>0$ is a sufficiently small constant, and $M>0$ is a generic constant independent of $\epsilon>0$.
\begin{proof}[Proof of \Cref{proposition:2.6}]
For $x\in\Omega$, then by the uniform interior sphere condition of $\Omega$, there exists $B_{R_{0}}(x_{0})\subseteq\Omega$ with $R_{0}\geq R$ such that $x\in B_{R_{0}}(x_{0})$ and $\partial B_{R_{0}}(x_{0})\cap\partial\Omega\neq\emptyset$.
Note that if $\delta(x)\geq R$, then we may set $x_{0}=x$ and $R_{0}=\delta(x)$; if $\delta(x)<R$, then we take $x_{0}=p_{x}+(x-p_{x})R/|x-p_{x}|$ and $R_{0}=R$, where $p_{x}\in\partial\Omega$ is the unique closest point to $x$.
Hence we have
\begin{align}
\label{eq.2.43}
R_{0}-|x-x_{0}|=\dist(x,\partial B_{R_{0}}(x_{0}))=\delta(x)\quad\text{for}~x\in\Omega.
\end{align}

To show \eqref{eq.2.41}, we let $\phi_{\epsilon}^{\pm}$ be radial solutions to
\begin{alignat}{2}
\label{eq.2.44}
-\epsilon\Delta\phi_{\epsilon}^{\pm}&=f(\phi_{\epsilon}^{\pm})&&\quad\text{in}~B_{R_{0}}(x_{0}),\\
\label{eq.2.45}
\phi_{\epsilon}^{\pm}&=\phi^{*}\pm C_{1}&&\quad\text{on}~\partial B_{R_{0}}(x_{0}).
\end{alignat}
The existence of solutions to \eqref{eq.2.44}--\eqref{eq.2.45} follows from the standard variational method. Now we claim
\begin{align}
\label{eq.2.46}
\phi_{\epsilon}^{-}(x)\leq\phi_{\epsilon}(x)\leq\phi_{\epsilon}^{+}(x)\quad\text{for}~x\in\overline B_{R_{0}}(x_{0}).
\end{align}
We first prove $\phi_{\epsilon}^+\geq\phi_{\epsilon}$.
Let $\overline\phi_{\epsilon}^+:=\phi_{\epsilon}^+-\phi_{\epsilon}$ in $\overline B_{R_{0}}(x_{0})$.
Then by \eqref{eq.1.01}, \eqref{eq.2.44}--\eqref{eq.2.45} and \Cref{proposition:2.1}, $\overline\phi_{\epsilon}^+$ satisfies
\begin{alignat}{2}
\label{eq.2.47}
&\epsilon\Delta\overline\phi_{\epsilon}^++c(x)\overline\phi_{\epsilon}^+=0\quad&&\text{in}~B_{R_{0}}(x_{0}),\\
\label{eq.2.48}
&\overline\phi_{\epsilon}^+=C_{1}-(\phi_{\epsilon}-\phi^{*})\geq0\quad&&\text{on}~\partial B_{R_{0}}(x_{0}),
\end{alignat}
where $c(x)=(f(\phi_{\epsilon}^+(x))-f(\phi_{\epsilon}(x)))/(\phi_{\epsilon}^+(x)-\phi_{\epsilon}(x))$ if $\phi_{\epsilon}^+(x)\neq\phi_{\epsilon}(x)$; $c(x)=f'(\phi_{\epsilon}(x))$ if $\phi_{\epsilon}^+(x)=\phi_{\epsilon}(x)$.
Note that $c(x)<0$ for $x\in B_{R_{0}}(x_{0})$ because $f$ is strictly decreasing on $\R$ (cf. (A1)).
Thus, we apply the maximum principle to \eqref{eq.2.47}--\eqref{eq.2.48} and obtain $\overline\phi_{\epsilon}^+\geq0$ in $\overline B_{R_{0}}(x_{0})$, which implies $\phi_{\epsilon}^+\geq\phi_{\epsilon}$ in $\overline B_{R_{0}}(x_{0})$.
Similarly, we prove $\phi_{\epsilon}^-\leq\phi_{\epsilon}$ in $\overline B_{R_{0}}(x_{0})$, and get \eqref{eq.2.46}.
Since $\phi_{\epsilon}^\pm$ are the radial solutions to \eqref{eq.2.44}--\eqref{eq.2.45}, then by \eqref{eq.2.43} and \Cref{proposition:A.2} (in \Cref{appendix:A}) by replacing $B_{R}(0)$ with $B_{R_{0}}(x_{0})$, we may use the fact that $R_{0}\geq R$ and obtain
\begin{align}
\label{eq.2.49}
|\phi_{\epsilon}^\pm(x)-\phi^{*}|\leq2C_{1}\exp\left(-\frac{C_{2}\dist(x,\partial B_{R_{0}}(x_{0}))}{\sqrt{\epsilon}}\right)=2C_{1}\exp\left(-\frac{C_{2}\delta(x)}{\sqrt\epsilon}\right)
\end{align}
for $x\in\overline B_{R_{0}}(x_{0})$ and $0<\epsilon<8(C_{2}R)^2/(d-1)^2$, which gives \eqref{eq.2.41}.

To get \eqref{eq.2.42}, we first claim
\begin{align}
\label{eq.2.50}
|\nabla\phi_{\epsilon}(x)|\leq M/\sqrt\epsilon\quad\text{for}~x\in\overline\Omega\quad\text{and}\quad0<\epsilon<(\delta(x))^2,
\end{align}
where $M>0$ is independent of $\epsilon>0$.
When $x\in\partial\Omega$, we may use \eqref{eq.2.40} to obtain \eqref{eq.2.50}.
Note that $B_{\sqrt\epsilon}(x_1)\subseteq\Omega$ for $x_1\in\Omega$ and $0<\epsilon<(\delta(x_1))^2$.
Set $y=(x-x_1)/\sqrt\epsilon$ and $\tilde\phi_{\epsilon}(y)=\phi_{\epsilon}(x_1+\sqrt\epsilon y)$.
Then from \eqref{eq.1.01}, we have $-\Delta\tilde\phi_{\epsilon}=f(\tilde\phi_{\epsilon})$ in $B_1(0)$.
By the uniform boundedness of $\phi_{\epsilon}$ (cf. \Cref{proposition:2.1}), we can apply the elliptic $L^q$-estimate and obtain $\|\tilde\phi_{\epsilon}\|_{W^{2,q}(B_{1/2}(0))}\leq M$ for $q>1$.
Using Sobolev's compact embedding theorem, we get $\|\tilde\phi_{\epsilon}\|_{\C^{1,\alpha}(B_{1/4}(0))}\leq M$ for $\alpha\in(0,1)$.
In particular, we have $|\nabla\tilde\phi_{\epsilon}(0)|\leq M$, which implies \eqref{eq.2.50} holds true for $x\in\Omega$.

Now we prove \eqref{eq.2.42}.
By \eqref{eq.1.01}, we have
\begin{align}
\label{eq.2.51}
\begin{aligned}
\epsilon\Delta|\nabla\phi_{\epsilon}|^2&=2\epsilon\sum_{i,j=1}^d\left(\frac{\partial^2\phi_{\epsilon}}{\partial x^j\partial x^i}\right)^2+2\epsilon\sum_{j=1}^d\frac{\partial\phi_{\epsilon}}{\partial x^j}\frac\partial{\partial x^j}\Delta\phi_{\epsilon}\\&\geq-2\nabla\phi_{\epsilon}\cdot\nabla(f(\phi_{\epsilon}))=-2f'(\phi_{\epsilon})|\nabla\phi_{\epsilon}|^2\quad\text{in}~B_{R_{0}}(x_{0}).
\end{aligned}
\end{align}
Along with \Cref{proposition:2.1} and (A1), we have $\epsilon\Delta|\nabla\phi_{\epsilon}|^2\geq128C_{2}^2|\nabla\phi_{\epsilon}|^2$ in $B_{R_{0}}(x_{0})$, where $C_{2}=m_f([\phi^{*}-C_{1},\phi^{*}+C_{1}])/8>0$ is given in \Cref{lemma:2.4}.
Let $\overline\phi_{\epsilon}$ be the solution to $\epsilon\Delta\overline\phi_{\epsilon}=128C_{2}^2\overline\phi_{\epsilon}$ in $B_{R_{0}}(x_{0})$ with the Dirichlet boundary condition $\overline\phi_{\epsilon}=\max\limits_{\partial B_{R_{0}}(x_{0})}|\nabla\phi_{\epsilon}|^2$ on $\partial B_{R_{0}}(x_{0})$.
Then by the standard comparison principle, we obtain
\[|\nabla\phi_{\epsilon}(x)|^2\leq|\overline\phi_{\epsilon}(x)|\leq2\left(\max_{\partial B_{R_{0}}(x_{0})}|\nabla\phi_{\epsilon}|^2\right)\exp\left(-\frac{\sqrt2C_{2}\dist(x,\partial B_{R_{0}}(x_{0}))}{\sqrt\epsilon}\right)\]
for $x\in\overline B_{R_{0}}(x_{0})$ and $0<\epsilon<16(C_{2}R)^2/(d-1)^2$.
Combining \eqref{eq.2.50}, we obtain
\[|\nabla\phi_{\epsilon}(x_{0})|\leq\frac{M}{\sqrt\epsilon}\exp\left(-\frac{C_{2}\delta(x_{0})}{\sqrt{2\epsilon}}\right)\quad\text{for}~0<\epsilon<16(C_{2}R)^2/(d-1)^2,\]
which implies \eqref{eq.2.42}.
Therefore, we complete the proof of \Cref{proposition:2.6}.
\end{proof}

By \Cref{proposition:2.6}, it is clear that \Cref{theorem:1}(b) holds true.

\subsection{Second-order asymptotic expansion of \texorpdfstring{$\phi_{\epsilon}$}{ϕ.ϵ} in \texorpdfstring{$\Omega_{k,\epsilon}$}{Ω.k,ϵ}}
\label{section:2.2}

In this section, we prove the second-order asymptotic expansion of $\phi_{\epsilon}$ near the boundary $\partial\Omega_{k}$ for $k\in\{0,1,\dots,K\}$.
Since $\dist(\Omega_{i},\Omega_{j})>0$ for $i\neq j\in\{0,1,\dots,K\}$ (cf. (A3)), there exists $\epsilon^{*}>0$ sufficiently small such that $\dist(\Omega_{i,\epsilon,\beta},\Omega_{j,\epsilon,\beta})\geq d_{0}>0$ for $i\neq j$ and $\Omega_{k,\epsilon,\beta}=\{x\in\Omega:\,\delta_k(x):=\dist(x,\partial\Omega_{k})<\epsilon^{\beta}\}$ contains no focal points for $0<\beta<1/2$ and $0<\epsilon<\epsilon^{*}$ (cf. \cite{1994cecil}), where $d_{0}>0$ is independent of $\epsilon$.
Then for $x\in\overline\Omega_{k,\epsilon,\beta}$, there exists a unique $p_{x}\in\partial\Omega_k$ (the closest point to $x$) and $0\leq t_{x}=\delta_{k}(x)/\sqrt{\epsilon}\leq\epsilon^{(2\beta-1)/2}$ such that $x=p_{x}-t_{x}\sqrt{\epsilon}\nu_{p_{x}}$, where $\nu_{p_{x}}$ is the unit outer normal at $p_{x}$ with respect to $\Omega=\Omega_0-\bigcup_{k=1}^{K}\Omega_{k}$.
Hence \eqref{eq.2.39}--\eqref{eq.2.40} become
\begin{align}
\label{eq.2.52}
&\phi_{\epsilon}(x)=u_{k}(\delta_{k}(x)/\sqrt{\epsilon})+o_{\epsilon}(1),\\
\label{eq.2.53}
&\nabla\phi_{\epsilon}(x)=-\frac{1}{\sqrt{\epsilon}}[u_{k}'(\delta_k(x)/\sqrt{\epsilon})\nu_{p_{x}}+o_{\epsilon}(1)].
\end{align}

For the second-order term of $\phi_{\epsilon}$, we use \eqref{eq.2.52} to define the following function:
\begin{align}
\label{eq.2.54}
\varphi_{k,\epsilon}(x)=\frac{\phi_{\epsilon}(x)-u_{k}(\delta_{k}(x)/\sqrt{\epsilon})}{\sqrt{\epsilon}}\quad\text{for}~x\in\overline\Omega_{k,\epsilon,\beta}~\text{and}~k=0,1,\dots,K,
\end{align}
where $u_{k}$ is the unique solution to \eqref{eq.1.07}--\eqref{eq.1.09}, and $\delta_k(x)=\dist(x,\partial\Omega_{k})$ for $k=0,1,\dots,K$.
Then by \eqref{eq.1.01}, \eqref{eq.1.07}, and \eqref{eq.2.54}, $\varphi_{k,\epsilon}$ satisfies
\begin{align}
\label{eq.2.55}
\epsilon\Delta\varphi_{k,\epsilon}(x)+c_{\epsilon}(x)\varphi_{k,\epsilon}(x)=g_\epsilon(x)\quad\text{for}~x\in\Omega_{k,\epsilon,\beta},
\end{align}
where
\begin{align}
\label{eq.2.56}
&c_{\epsilon}(x)=\begin{cases}\dd\frac{f(\phi_{\epsilon}(x))-f(u_{k}(\delta_k(x)/\sqrt\epsilon))}{\phi_{\epsilon}(x)-u_{k}(\delta_{k}(x)/\sqrt{\epsilon})}&\text{if}~\phi_{\epsilon}(x)\neq u_{k}(\delta_{k}(x)/\sqrt{\epsilon});\\
f'(\phi_{\epsilon}(x))&\text{if}~\phi_{\epsilon}(x)=u_{k}(\delta_{k}(x)/\sqrt{\epsilon}),
\end{cases}\\
\label{eq.2.57}
&g_{\epsilon}(x)=(d-1)H(p_{x})u_{k}'(\delta_{k}(x)/\sqrt{\epsilon})+\epsilon^{\beta}\mathcal{O}_\epsilon(1).\end{align}
Here we have used the fact that $0\leq\delta_k(x)\leq\epsilon^{\beta}$, $|\nabla\delta_k(x)|=1$, and
\begin{align}
\label{eq.2.58}
\Delta\delta_k(x)=-\sum_{i=1}^{d-1}\frac{\kappa_i}{1-\kappa_i\delta_k(x)}=-(d-1)H(p_x)+\epsilon^{\beta}\mathcal{O}_\epsilon(1)\quad\text{for}~x\in\Omega_{k,\epsilon,\beta},
\end{align}
where $\kappa_i$ are the principal curvatures at $p_x\in\partial\Omega_k$, and $H(p_x)=(d-1)^{-1}\sum_{i=1}^{d-1}\kappa_i$ is the mean curvature at $p_x\in\partial\Omega_k$ (cf. \cite[Lemma~14.17]{1977gilbarg}).
By (A1), \eqref{eq.2.12} and \Cref{proposition:2.1}, $c_\epsilon(x)\leq-64C_{2}^2<0$ for $x\in\overline\Omega_{k,\epsilon,\beta}$, where $C_{2}=m_f([\phi^{*}-C_{1},\phi^{*}+C_{1}])/8$.
Moreover, $g_\epsilon$ is uniformly bounded with respect to $\epsilon$ in $\Omega_{k,\epsilon}$ because $u_{k}'$ is bounded on $[0,\infty)$ by \Cref{proposition:B.1} in \Cref{appendix:B}. For the boundary condition of $\varphi_{k,\epsilon}$, we use \eqref{eq.1.04}, \eqref{eq.1.08}, \eqref{eq.2.03} and \eqref{eq.2.54} to get
\begin{align}
\label{eq.2.59}
\varphi_{k,\epsilon}+\gamma_k\sqrt\epsilon\partial_\nu\varphi_{k,\epsilon}=0\quad\text{on}~\partial\Omega_k.
\end{align}

Then we prove the uniform boundedness of $\varphi_{k,\epsilon}$ as follows.
\begin{proposition}[Uniform boundedness of $\varphi_{k,\epsilon}$]
\label{proposition:2.7}
There exists a constant $M>0$ independent of $\epsilon$ such that $\dd\max_{\overline\Omega_{k,\epsilon,\beta}}|\varphi_{k,\epsilon}|\leq M$ for $k=0,1,\dots,K$ and $0<\epsilon<\epsilon^{*}$.
\end{proposition}
\begin{proof}
Fix $k\in\{0,1,\dots,K\}$.
It is equivalent to proving that $\dd\max_{\overline\Omega_{k,\epsilon,\beta}}\varphi_{k,\epsilon}\leq M$ and $\dd\min_{\overline\Omega_{k,\epsilon,\beta}}\varphi_{k,\epsilon}\geq-M$ for some constant $M>0$ independent of $\epsilon$.
First, we show that $\dd\max_{\overline\Omega_{k,\epsilon,\beta}}\varphi_{k,\epsilon}\leq M$ for $0<\epsilon<\epsilon^{*}$.
Suppose by contradiction that there exists a sequence $\{\epsilon_{n}\}_{n=1}^\infty$ of positive numbers with $\dd\lim_{n\to\infty}\epsilon_{n}=0$ and $\{x_n\}_{n=1}^\infty\subset\overline\Omega_{k,\epsilon_{n},\beta}$ such that $\dd\varphi_{k,\epsilon_{n}}(x_n)=\max_{\overline\Omega_{k,\epsilon_{n},\beta}}\varphi_{k,\epsilon_{n}}\geq n$ for $n\in\N$.
Since $0<\beta<1/2$, we may, without loss of generality, assume that $8C_{1}\exp(-C_{2}\epsilon_{n}^{(2\beta-1)/2})\leq\sqrt{\epsilon_{n}}$ for all $n\in\N$, where $C_{1}$ and $C_{2}$ are given in \Cref{proposition:2.1} and \Cref{lemma:2.4}, respectively.
Note that the maximum point $x_n$ cannot lie on the boundary $\partial\Omega_k$ because from \eqref{eq.2.59}, $\partial_\nu\varphi_{k,\epsilon_{n}}(x_n)=-\varphi_{k,\epsilon_{n}}(x_n)/(\gamma_k\sqrt{\epsilon_{n}})\leq-n/(\gamma_k\sqrt{\epsilon_{n}})<0$ if the maximum point $x_n\in\partial\Omega_k$.
On the other hand, by \eqref{eq.2.12} (with \Cref{proposition:2.5}) and \eqref{eq.2.41}, we have
\begin{align}
\label{eq.2.60}
|\varphi_{k,\epsilon_{n}}(x)|\leq\frac{|\phi_{\epsilon_{n}}(x)-\phi^{*}|+|u_{k}(\delta_k(x)/\sqrt{\epsilon_{n}})-\phi^{*}|}{\sqrt{\epsilon_{n}}}\leq\frac{4C_{1}}{\sqrt{\epsilon_{n}}}\exp\left(-C_{2}\epsilon_{n}^{(2\beta-1)/2}\right)\leq\frac12
\end{align}
for $\delta_k(x)=\epsilon_{n}^{\beta}$ and $n\in\N$.
This shows that $x_n$ cannot lie on the boundary $\partial\Omega_{k,\epsilon_{n},\beta}$.
Hence $x_n\in\Omega_{k,\epsilon_{n},\beta}$ for all $n\in\N$, which implies $\nabla\varphi_{k,\epsilon_{n}}(x_n)=0$ and $\Delta\varphi_{k,\epsilon_{n}}(x_n)\leq0$ for $n\in\N$.
Thus by \eqref{eq.2.55},
\begin{align}
\label{eq.2.61}
0\leq-\epsilon_{n}\Delta\varphi_{k,\epsilon_{n}}(x_n)=c_{\epsilon_{n}}(x_n)\varphi_{k,\epsilon_{n}}(x_n)-g_{\epsilon_{n}}(x_n).
\end{align}
Recall that $g_\epsilon$ is uniformly bounded, $\varphi_{k,\epsilon_{n}}(x_n)\geq n$ for $n\in\N$, and $c_\epsilon\leq-64C_{2}^2<0$ in $\overline\Omega_{k,\epsilon}$ for $\epsilon>0$.
Letting $n\to\infty$, we obtain $\dd\lim_{n\to\infty}c_{\epsilon_{n}}(x_n)\varphi_{k,\epsilon_{n}}(x_n)=-\infty$, which contradicts \eqref{eq.2.61}.
Thus, $\dd\max_{\overline\Omega_{k,\epsilon,\beta}}\varphi_{k,\epsilon}\leq M$ for some constant $M>0$ independent of $\epsilon$.
Similarly, we may use \eqref{eq.2.12}, \eqref{eq.2.41}, \eqref{eq.2.55} and \eqref{eq.2.59} to prove $\dd\min_{\overline\Omega_{k,\epsilon,\beta}}\varphi_{k,\epsilon}\geq-M$ for $0<\epsilon<\epsilon^{*}$, where $M>0$ is independent of $\epsilon>0$.
Therefore, we complete the proof of \Cref{proposition:2.7}.
\end{proof}

To study the asymptotic expansion of $\varphi_{k,\epsilon}$, we first claim that $\Psi_p(\overline B_{\epsilon^{(2\beta+1)/2}}^+)\subseteq\overline\Omega_{k,\epsilon,\beta}$ for $p\in\partial\Omega_k$, $k\in\{0,1,\dots,K\}$ and sufficiently small $\epsilon>0$ because
\begin{align*}
\delta_k(\Psi_p(y))=|\delta_k(\Psi_p(y))-\delta_k(p)|\leq C|\Psi_p(y)-p|\leq C'|y|\leq\epsilon^{\beta}
\end{align*}
for $y\in\overline B_{\epsilon^{(2\beta+1)/2}}^+$, $p\in\partial\Omega_k$, $k\in\{0,1,\dots,K\}$ and  sufficiently small $\epsilon>0$, where $\Psi_p$ is the diffeomorphism in \eqref{eq.2.01}, $C=C(\Omega)$ and $C'=C'(p,\Omega)$ are positive constants independent of $y$.
Here we have used the facts that $\delta_k(p)=0$, $\psi_p(p')=p^d$ (cf. (D1)) and the smoothness of $\psi_p$ and $\delta_k$.
Hence there exists $\epsilon^{*}>0$ sufficiently small such that $\Psi_p(\overline B_{\epsilon^{(2\beta+1)/2}}^+)\subset\overline\Omega_{k,\epsilon,\beta}\cap\Psi_p(\overline B_b^+)$ for $0<\epsilon<\epsilon^{*}$.
As in \Cref{section:2.1}, we apply the principal coordinate system \eqref{eq.2.01} to the neighborhood $\Psi_p(\overline B_{\epsilon^{(2\beta+1)/2}}^+)$, rescale the spatial variable $y$ by $\sqrt{\epsilon}$ and define
\begin{align}
\label{eq.2.62}
\W_{k,p,\epsilon}(z)=\varphi_{k,\epsilon}(\Psi_p(\sqrt\epsilon z))\quad\text{for}~z\in\overline B_{\epsilon^{(2\beta-1)/2}}^+~\text{and}~0<\epsilon<\epsilon^{*}.
\end{align}
Note that $\W_{k,p,\epsilon}$ is well-defined due to $\Psi_p(\overline B_{\epsilon^{(2\beta+1)/2}}^+)\subset\overline\Omega_{k,\epsilon,\beta}\cap\Psi_p(\overline B_b^+)$ for $0<\epsilon<\epsilon^{*}$.

By \Cref{lemma:2.2} and equations \eqref{eq.2.55}--\eqref{eq.2.59}, $\W_{k,p,\epsilon}$ satisfies
\begin{alignat}{2}
\label{eq.2.63}
\sum_{i,j=1}^da_{ij}(z)\frac{\partial^2\W_{k,p,\epsilon}}{\partial z^{i}\partial z^j}+\sum_{j=1}^db_j(z)\frac{\partial \W_{k,p,\epsilon}}{\partial z^j}+c_\epsilon(z)\W_{k,p,\epsilon}&=g_\epsilon(z)\quad&&\text{in}~B_{\epsilon^{(2\beta-1)/2}}^+,\\
\label{eq.2.64}
\W_{k,p,\epsilon}-\gamma_k\partial_{z^{d}}\W_{k,p,\epsilon}&=0\quad&&\text{on}~\overline B_{\epsilon^{(2\beta-1)/2}}^+\cap\partial\R_{+}^{d}
\end{alignat}
for $0<\epsilon<\epsilon^{*}$, where $a_{ij}$ and $b_j$ are given in \eqref{eq.2.05}--\eqref{eq.2.06}, and
\begin{align}
\label{eq.2.65}
&c_\epsilon(z)=\begin{cases}\dd\frac{f(u_{k}(z^{d})+\sqrt{\epsilon}\W_{k,p,\epsilon}(z))-f(u_{k}(z^{d}))}{\sqrt{\epsilon}\W_{k,p,\epsilon}(z)}&\text{if}~\W_{k,p,\epsilon}(z)\neq0;\\
f'(u_{k}(z^{d}))&\text{if}~\W_{k,p,\epsilon}(z)=0,\end{cases}\\
\label{eq.2.66}
&g_\epsilon(z)=(d-1)H(p)u_{k}'(z^{d})+\epsilon^{\beta}\mathcal{O}_\epsilon(1).
\end{align}
As in the proofs of \Cref{lemma:2.3,lemma:2.4}, we may use the uniform boundedness of $\varphi_{k,\epsilon}$ (cf. \Cref{proposition:2.7}) to prove the convergence of $\W_{k,p,\epsilon}$ in $\C^{2,\alpha}(\overline B_m^+)$ as $\epsilon$ tends to zero (up to a subsequence) for $m\in\N$ and then prove the exponential-type estimate of $\W_{k,p}$.
\begin{lemma}
\label{lemma:2.8}
For any sequence $\{\epsilon_{n}\}_{n=1}^\infty$ of positive numbers with $\dd\lim_{n\to\infty}\epsilon_{n}=0$, $\alpha\in(0,1)$, and $p\in\partial\Omega_k$ ($k\in\{0,1,\dots,K\}$), there exists a subsequence $\{\epsilon_{nn}\}_{n=1}^\infty$ such that $\dd\lim_{n\to\infty}\|\W_{k,p,\epsilon_{nn}}-\W_{k,p}\|_{\C^{2,\alpha}(\overline B_m^+)}=0$ for $m\in\N$, where $\W_{k,p}\in\C^{2,\alpha}_{\text{loc}}(\overline{\R}_{+}^{d})$ satisfies
\begin{alignat}{2}
\label{eq.2.67}
\Delta \W_{k,p}+f'(u_{k})\W_{k,p}&=(d-1)H(p)u_{k}'(z^{d})\quad&&\text{in}~\R_{+}^{d},\\
\label{eq.2.68}
\W_{k,p}-\gamma_k\partial_{z^{d}}\W_{k,p}&=0\quad&&\text{on}~\partial\R_{+}^{d}.
\end{alignat}
Moreover, there exists $M>0$ such that $\W_{k,p}$ satisfies
\begin{align}
\label{eq.2.69}
|\W_{k,p}(z)-(d-1)H(p)v_{k}(z^{d})|\leq M\exp(-C_{2}z^{d})
\end{align}
for $z=(z',z^{d})\in\overline\R_{+}^{d}$ and $z^{d}\geq(d-1)/(4C_{2})$, where $v_{k}$ is the solution to \eqref{eq.1.10}--\eqref{eq.1.12} and $C_{2}$ is given in \Cref{lemma:2.4}.
\end{lemma}

From \Cref{lemma:2.8}, the solution $\W_{k,p}$ to \eqref{eq.2.67}--\eqref{eq.2.68} may, a priori depend on the sequence $\{\epsilon_{n}\}_{n=1}^\infty$.
To establish the independence from the sequence, we apply the moving plane arguments to prove $\W_{k,p}(z)=\W_{k,p}(z^{d})=(d-1)H(p)v_{k}(z^{d})$ for $z=(z',z^{d})\in\overline\R_{+}^{d}$ (see \Cref{proposition:2.9} below), where $v_{k}$ is the unique solution to \eqref{eq.1.10}--\eqref{eq.1.12}, and $v_{k}$ is independent of the choice of sequence $\{\epsilon_{n}\}_{n=1}^\infty$ and point $p\in\partial\Omega_k$.
Therefore, the results of \Cref{lemma:2.8} can be improved as
\begin{align}
\label{eq.2.70}
\lim_{\epsilon\to0^+}\|\W_{k,p,\epsilon}-(d-1)H(p)v_{k}\|_{\C^{2,\alpha}(\overline B_m^+)}=0\quad\text{for}~m\in\N~\text{and}~\alpha\in(0,1),
\end{align}
and
\begin{align}
\label{eq.2.71}
\lim_{\epsilon\to0^+}\W_{k,p,\epsilon}(z)=(d-1)H(p)v_{k}(z^{d})\quad\text{for}~z=(z',z^{d})\in\overline\R_{+}^{d},
\end{align}
where $\W_{k,p,\epsilon}$ is defined in \eqref{eq.2.62}.
Below are the details.
\begin{proposition}
\label{proposition:2.9}
For $p\in\partial\Omega_k$ and $k\in\{0,1,\dots,K\}$, the solution $\W_{k,p}$ to \eqref{eq.2.67}--\eqref{eq.2.68} satisfies
\begin{enumerate}
\item[(a)] $\W_{k,p}$ depends only on the variable $z^{d}$, i.e., $\W_{k,p}(z)=\W_{k,p}(z^{d})$ for $z=(z',z^{d})\in\overline\R_{+}^{d}$.
\item[(b)] $\W_{k,p}(z^{d})=(d-1)H(p)v_{k}(z^{d})$ for $z^{d}\in[0,\infty)$, where $v_{k}$ is the unique solution to \eqref{eq.1.10}--\eqref{eq.1.12}.
\end{enumerate}
\end{proposition}
As in \Cref{proposition:2.5}, we may use the moving plane arguments to prove \Cref{proposition:2.9}.
The proof of \Cref{proposition:2.9} is similar to that of \Cref{proposition:2.5} so we omit it here.

We may use \Cref{proposition:2.9} to prove \eqref{eq.2.70}, \eqref{eq.2.71}, and $\W_{k,p}(z)=(d-1)H(p)v_{k}(z^{d})$ for $z=(z',z^{d})\in\overline\R_{+}^{d}$, where $v_{k}$ is the solution to \eqref{eq.1.10}--\eqref{eq.1.12}.
Let $T>0$, $k\in\{0,1,\dots,K\}$, and $p\in\partial\Omega_k$.
Then by \eqref{eq.2.01}, \eqref{eq.2.62}, \eqref{eq.2.70}--\eqref{eq.2.71} and $\nabla\delta_k(p-t\sqrt{\epsilon}\nu_{p})=-\nu_{p}$, we have
\begin{align}
\label{eq.2.72}
\varphi_{k,\epsilon}(p-t\sqrt{\epsilon}\nu_{p})=\W_{k,p,\epsilon}(te_d)=\W_{k,p}(t)+o_{\epsilon}(1)=(d-1)H(p)v_{k}(t)+o_{\epsilon}(1)
\end{align}
and
\begin{align}
\label{eq.2.73}
\nabla\varphi_{k,\epsilon}(p-t\sqrt{\epsilon}\nu_{p})=-\frac{1}{\sqrt{\epsilon}}(\W_{k,p}'(t)\nu_{p}+o_{\epsilon}(1))=-\frac{1}{\sqrt{\epsilon}}((d-1)H(p)v_{k}'(t)\nu_{p}+o_{\epsilon}(1))
\end{align}
for $0\leq t\leq T$ as $\epsilon\to0^+$.
Combining \eqref{eq.2.54}, \eqref{eq.2.62}, and \eqref{eq.2.72}--\eqref{eq.2.73}, we obtain
\begin{align*}
&\phi_{\epsilon}(p-t\sqrt{\epsilon}\nu_{p})=u_{k}(t)+\sqrt{\epsilon}\Big((d-1)H(p)v_{k}(t)+o_{\epsilon}(1)\Big),\\
&\nabla\phi_{\epsilon}(p-t\sqrt{\epsilon}\nu_{p})=-\left(\frac{1}{\sqrt{\epsilon}}u_{k}'(t)+(d-1)H(p)v_{k}'(t)\right)\nu_{p}+o_{\epsilon}(1)
\end{align*}
for $0\leq t\leq T$ as $\epsilon\to0^+$, which implies \eqref{eq.1.05}--\eqref{eq.1.06}.
Therefore, we complete the proof of \Cref{theorem:1}(a).

\subsection{Proof of \texorpdfstring{\Cref{corollary:1}}{Corollary 1}}
\label{section:2.3}

We present the proof of \Cref{corollary:1}.
To prove (a), we integrate \eqref{eq.1.15} over $\overline\Omega_{k,T,\epsilon}$ and apply the coarea formula (see \cite{2015evans,2002lin}) to obtain
\begin{align}
\label{eq.2.74}
\int_{\overline{\Omega}_{k,T,\epsilon}}\!f(\phi_{\epsilon}(x))\,\mathrm{d}x=\sqrt{\epsilon}\int_{0}^{T}\!\int_{\partial\Omega_{k}}\!\{f(u_{k}(t))+\sqrt{\epsilon}[(d-1)H(p)f'(u_{k}(t))v_{k}(t)+o_{\epsilon}(1)]\}\mathcal{J}(t,p)\,\mathrm{d}S_p\,\mathrm{d}t,\end{align}
where
\begin{align}\label{eq.2.75}
\mathcal{J}(t,p)=1-t\sqrt{\epsilon}(d-1)H(p)+t\sqrt{\epsilon}o_{\epsilon}(1)\quad\text{for}~p\in\partial\Omega_{k}~\text{and}~t\geq0,
\end{align}
follows from Steiner's formula (cf. \cite{1981abbena,1977gilbarg,2004gray}).
Combining \eqref{eq.1.07}, \eqref{eq.1.10}, and \eqref{eq.2.74}--\eqref{eq.2.75}, we have
\begin{align*}\int_{\overline\Omega_{k,T,\epsilon}}\!f(\phi_{\epsilon}(x))\,\mathrm{d}x&=\sqrt{\epsilon}\int_{0}^{T}\!\int_{\partial\Omega_{k}}\!f(u_{k}(t))\,\mathrm{d}S_{p}\,\mathrm{d}t\\&\quad+\epsilon(d-1)\left(\int_{\partial\Omega_{k}}\!H(p)\,\mathrm{d}S_{p}\right)\left(\int_{0}^{T}\![f'(u_{k}(t))v_{k}(t)-tf(u_{k}(t))]\,\mathrm{d}t\right)+\epsilon o_{\epsilon}(1)\\&=-\sqrt{\epsilon}|\partial\Omega_k|\int_{0}^{T}\!u_{k}''(t)\,\mathrm{d}t\\&\quad+\epsilon(d-1)\left(\int_{\partial\Omega_{k}}\!H(p)\,\mathrm{d}S_{p}\right)\left(\int_{0}^{T}\![u_{k}'(t)-v_{k}''(t)+tu_{k}''(t)]\,\mathrm{d}t\right)+\epsilon o_{\epsilon}(1)\\&=\sqrt{\epsilon}|\partial\Omega_k|(u_{k}'(0)-u_{k}'(T))\\&\quad+\epsilon(d-1)\left(\int_{\partial\Omega_k}\!H(p)\,\mathrm{d}S_{p}\right)(Tu_{k}'(T)+v_{k}'(0)-v_{k}'(T))+\epsilon o_{\epsilon}(1),\end{align*}
which completes the proof of (a).

We now  prove (b). Note that
\[\partial\Omega_{k,T,\epsilon,\beta}=\{p-T\sqrt{\epsilon}\nu_{p}\in\Omega\,:\,p\in\partial\Omega_{k}\}\cup\{p-\epsilon^{\beta}\nu_{p}\in\Omega\,:\,p\in\partial\Omega_{k}\}\]
is the union of two disjoint parallel surfaces to the boundary $\partial\Omega_{k}$.
Integrating \eqref{eq.1.01} over $\overline\Omega_{k,T,\epsilon,\beta}$ and applying the divergence theorem, we obtain
\begin{align}
\label{eq.2.76}
\begin{aligned}
\int_{\overline\Omega_{k,T,\epsilon,\beta}}\!f(\phi_{\epsilon}(x))\,\mathrm{d}x&=-\epsilon\int_{\overline\Omega_{k,T,\epsilon,\beta}}\!\Delta\phi_{\epsilon}(x)\,\mathrm{d}x\\&=-\epsilon\int_{\partial\Omega_{k,T,\epsilon,\beta}}\partial_{\nu_x}\phi_{\epsilon}(x)\,\mathrm{d}S_x\\&=-\int_{\partial\Omega_{k}}\!(\epsilon\partial_{\nu_p}\phi_{\epsilon}(p-T\sqrt{\epsilon}\nu_{p}))\mathcal{J}(T,p)\,\mathrm{d}S_{p}\\&\quad+\int_{\partial\Omega_{k}}\!(\epsilon\partial_{\nu_p}\phi_{\epsilon}(p-\epsilon^{\beta}\nu_{p}))\mathcal{J}(\epsilon^{(2\beta-1)/2},p)\,\mathrm{d}S_p,
\end{aligned}
\end{align}
where $\nu_x$ is the unit outer normal at $x\in\partial\Omega_{k,T,\epsilon,\beta}$ with respect to $\Omega_{k,T,\epsilon,\beta}$ and
\begin{align}
\label{eq.2.77}
\mathcal{J}(T,p)=1-T\sqrt{\epsilon}(d-1)H(p)+\sqrt{\epsilon}o_{\epsilon}(1),\quad\mathcal{J}(\epsilon^{(2\beta-1)/2},p)=1-\epsilon^{\beta}(d-1)H(p)+\epsilon^{\beta}o_{\epsilon}(1).
\end{align}
Here we have used Steiner's formula in differential geometry (cf. \cite{1981abbena,2004gray}).
By \eqref{eq.1.06}, we have
\begin{align}
\label{eq.2.78}
\begin{aligned}
&\quad-\int_{\partial\Omega_{k}}\!(\epsilon\partial_{\nu_p}\phi_{\epsilon}(p-T\sqrt{\epsilon}\nu_{p}))\mathcal{J}(T,p)\,\mathrm{d}S_{p}\\&=-\int_{\partial\Omega_{k}}\!(\epsilon\partial_{\nu_p}\phi_{\epsilon}(p-T\sqrt{\epsilon}\nu_{p}))[1-T\sqrt{\epsilon}(d-1)H(p)+\sqrt{\epsilon}o_{\epsilon}(1)]\,\mathrm{d}S_{p}\\&=\sqrt{\epsilon}\int_{\partial\Omega_{k}}\!\{u_{k}'(T)+\sqrt{\epsilon}[(d-1)H(p)v_{k}'(T)+o_{\epsilon}(1)]\}[1-T\sqrt{\epsilon}(d-1)H(p)+\sqrt{\epsilon}o_{\epsilon}(1)]\,\mathrm{d}S_{p}\\&=
\sqrt{\epsilon}|\partial\Omega_k|u_{k}'(T)+\epsilon(d-1)\left(\int_{\partial\Omega_k}\!H(p)\,\mathrm{d}S_p\right)(-Tu_{k}'(T)+v_{k}'(T))+\epsilon o_{\epsilon}(1).
\end{aligned}
\end{align}
On the other hand, by \Cref{theorem:1}(b,ii), we have
\begin{align}
\label{eq.2.79}
\left|\int_{\partial\Omega_{k}}\!(\epsilon\partial_{\nu_p}\phi_{\epsilon}(p-\epsilon^{\beta}\nu_{p}))\mathcal{J}(\epsilon^{(2\beta-1)/2},p)\,\mathrm{d}S_p\right|\leq\sqrt{\epsilon}|\partial\Omega_{k}|M'\exp\left(-M\epsilon^{(2\beta-1)/2}\right).
\end{align}
Combining \eqref{eq.2.76}--\eqref{eq.2.79} with $0<\beta<1/2$, we get (b).
By an argument similar to that used in the proof of (b), one may use \eqref{eq.2.79} to establish (c).
Therefore, we complete the proof of \Cref{corollary:1}.

\section{Proof of \texorpdfstring{\Cref{theorem:2}}{Theorem 2}}
\label{section:3}

In this section, we derive the first- and second-order asymptotic expansions of the solution $\phi_{\epsilon}$ to equation \eqref{eq.1.02} with condition \eqref{eq.1.04} under the charge neutrality condition \eqref{eq.1.03}. Equation \eqref{eq.1.02} with condition \eqref{eq.1.04} can be denoted by
\begin{alignat}{2}
\label{eq.3.001}
-\epsilon\Delta\phi_{\epsilon}&=f_{\epsilon}(\phi_{\epsilon})\quad&&\text{in}~\Omega,\\
\label{eq.3.002}
\phi_{\epsilon}+\gamma_{k}\sqrt{\epsilon}\partial_{\nu}\phi_{\epsilon}&=\phi_{bd,k}\quad&&\text{on}~\partial\Omega_{k}~\text{for}~k=0,1,\dots,K,
\end{alignat}
where $f_{\epsilon}=f_{\epsilon}(\phi)$ is defined as
\begin{align}
\label{eq.3.003}
f_{\epsilon}(\phi)=\sum_{i=1}^{I}\frac{m_{i}z_{i}}{A_{i,\epsilon}}\exp(-z_{i}\phi)\quad\text{for}~\phi\in\R,\quad\text{and}\quad A_{i,\epsilon}=\int_{\Omega}\!\exp(-z_{i}\phi_{\epsilon}(y))\,\mathrm{d}y.
\end{align}
In \Cref{section:3.1}, we prove the uniform boundedness of $\phi_{\epsilon}$ and show that $f_{\epsilon}$ satisfies conditions (A1)--(A2) with the unique zero $\phi_{\epsilon}^{*}$ for $\epsilon>0$.
Thus, $f_{\epsilon}(\phi)=f_{0}(\phi)+o_{\epsilon}(1)$ for $\phi\in\R$, and \eqref{eq.3.001} can be approximated by $-\epsilon\Delta\phi_{\epsilon}=f_{0}(\phi_{\epsilon})+o_{\epsilon}(1)$ in $\Omega$, where
\begin{align*}
f_{0}(\phi)=\frac{1}{|\Omega|}\sum_{i=1}^{I}m_{i}z_{i}\exp(-z_{i}(\phi-\phi_{0}^{*}))\quad\text{for}~\phi\in\R,
\end{align*}
and $\phi_{0}^{*}$ is the pointwise limit of $\phi_{\epsilon}$ as $\epsilon$ tends to zero (see \Cref{remark:4}).
Note that $f_{0}$ is strictly decreasing on $\R$ and has a unique zero $\phi_{0}^{*}$. Hence we can use the method of \Cref{section:2.1} to prove the first-order asymptotic expansions $\phi_{\epsilon}(p-t\sqrt{\epsilon}\nu_{p})=u_{k}(t)+o_{\epsilon}(1)$ and $\nabla\phi_{\epsilon}(p-t\sqrt{\epsilon}\nu_{p})=-\epsilon^{-1/2}(u_{k}'(t)\nu_{p}+o_{\epsilon}(1))$ for $T>0$, $p\in\partial\Omega_{k}$, and $0\leq t\leq T$, where $u_{k}$ is the unique solution to \eqref{eq.1.19}--\eqref{eq.1.21} (cf. \Cref{section:3.2}).
To obtain the second-order asymptotic expansions of $\phi_{\epsilon}$ and $\nabla\phi_{\epsilon}$, it is essential to show the uniform boundedness of $|\phi_{\epsilon}^{*}-\phi_{0}^{*}|/\sqrt{\epsilon}$ in \Cref{section:3.3}.
We suppose, by contradiction, that $|\phi_{\epsilon}^{*}-\phi_{0}^{*}|/\sqrt{\epsilon}\to\infty$ as $\epsilon\to0^+$, and then derive the asymptotic expansions of the integral terms $A_{i,\epsilon}$ in \eqref{eq.3.003} and the nonlinear term $f_{\epsilon}$.
Then, as in \Cref{section:2.2,section:2.3}, we obtain the asymptotic expansions of $\phi_{\epsilon}$ and $\nabla\phi_{\epsilon}$, and hence the asymptotic expansions of total ionic charge over the regions $\overline{\Omega}_{k,T,\epsilon}$, $\overline{\Omega}_{k,T,\epsilon,\beta}$, and $\overline{\Omega}_{\epsilon,\beta}$ (see \Cref{proposition:3.12}), which contradicts the charge neutrality condition \eqref{eq.1.03}.
As a result, we show $\phi_{\epsilon}^{*}=\phi_{0}^{*}+\sqrt{\epsilon}(Q+o_{\epsilon}(1))$ (see \eqref{eq.3.076} in \Cref{remark:6}).
In \Cref{section:3.4}, we use $\phi_{\epsilon}^{*}=\phi_{0}^{*}+\sqrt{\epsilon}(Q+o_{\epsilon}(1))$ and follow a similar argument of \Cref{section:3.3} to establish the asymptotic expansions of $A_{i,\epsilon}$ and get the further asymptotic expansion of $f_{\epsilon}$: $f_{\epsilon}(\phi)=f_{0}(\phi)+\sqrt{\epsilon}(f_1(\phi)+o_{\epsilon}(1))$ for $\phi\in\R$ (cf. \eqref{eq.3.096}), where
\begin{align}
\label{eq.3.004}
f_{1}(\phi)=-Qf_{0}'(\phi)+\hat{f}_{1}(\phi)\quad\text{for}~\phi\in\R.
\end{align}
Here \begin{align*}&\hat{f}_{1}(\phi)=\frac{1}{|\Omega|}\sum_{i=1}^{I}\hat m_{i}z_{i}\exp(-z_{i}(\phi-\phi_{0}^{*}))\quad\text{for}~\phi\in\R,\\&\hat m_{i}=\frac{m_{i}}{|\Omega|}\sum_{k=0}^K|\partial\Omega_k|\int_{0}^{\infty}\![1-\exp(-z_{i}(u_{k}(s)-\phi_{0}^{*}))]\,\mathrm{d}s\quad\text{for}~i=1,\dots,I.\end{align*}
Consequently, equation \eqref{eq.3.001} becomes
\[-\epsilon\Delta\phi_{\epsilon}=f_{0}(\phi_{\epsilon})+\sqrt{\epsilon}(f_1(\phi_{\epsilon})+o_{\epsilon}(1))\quad\text{in}~\Omega,\]
and we can apply the method of \Cref{section:2.2} to prove the second-order asymptotic expansions of $\phi_{\epsilon}$ and $\nabla\phi_{\epsilon}$ (cf. \eqref{eq.1.17}--\eqref{eq.1.18}).
In \Cref{section:3.4}, using the asymptotic expansions of $\phi_{\epsilon}$ and $\nabla\phi_{\epsilon}$, we prove \Cref{corollary:2} and then apply \Cref{corollary:2} to determine the unique value of $Q$, which verifies \Cref{remark:6}.
Therefore, we complete the proof of \Cref{theorem:2}.

\subsection{Estimate of the solution \texorpdfstring{$\phi_{\epsilon}$}{ϕ.ϵ} and \texorpdfstring{$f_{\epsilon}$}{f.ϵ}}
\label{section:3.1}

Equation \eqref{eq.1.02} with the boundary condition \eqref{eq.1.04} is the Euler--Lagrange equation of the following energy functional
\[E[\phi]=\frac{\epsilon}{2}\int_\Omega\!|\nabla\phi|^2\,\mathrm{d}x-\sum_{i=1}^{I}m_{i}\ln\left(\int_\Omega\!\exp(-z_{i}\phi)\,\mathrm{d}x\right)+\sqrt{\epsilon}\sum_{k=0}^K\frac{1}{2\gamma_k}\int_{\partial\Omega_k}\!(\phi-\phi_{bd,k})^2\,\mathrm{d}S\quad\text{for}~\phi\in H^1(\Omega).\]
The existence and uniqueness of the smooth solution $\phi_{\epsilon}\in\C^\infty(\overline{\Omega})$ was proved in \cite{2014lee}.
Due to the charge neutrality condition \eqref{eq.1.03}, we prove the uniform boundedness of the solution $\phi_{\epsilon}$ as follows.
\begin{proposition}[Uniform boundedness of $\phi_{\epsilon}$]
\label{proposition:3.01}
Assume that $\phi_{bd,k}$ are not equal and \eqref{eq.1.03} holds true.
Let $\phi_{\epsilon}\in\C^{\infty}(\overline{\Omega})$ be the unique solution to equations \eqref{eq.3.001}--\eqref{eq.3.002}.
Then $\phi_{\epsilon}$ satisfies
\begin{align}
\label{eq.3.005}
\underline{\phi_{bd}}\leq\phi_{\epsilon}(x)\leq\overline{\phi_{bd}}\quad\text{for}~x\in\overline\Omega~\text{and}~\epsilon>0,
\end{align}
where $\underline{\phi_{bd}}=\min\limits_{0\leq k\leq K}\phi_{bd,k}$ and $\overline{\phi_{bd}}=\max\limits_{0\leq k\leq K}\phi_{bd,k}$.
\end{proposition}
\begin{proof}We first prove that $\phi_{\epsilon}$ attains its maximum value at a boundary point.
Suppose by contradiction that $\phi_{\epsilon}$ attains the maximum value at an interior point $x_{0}\in\Omega$. It is clear that $\phi_{\epsilon}(x_{0})\geq\phi_{\epsilon}(y)$ for $y\in\overline{\Omega}$ and $\Delta\phi_{\epsilon}(x_{0})\leq0$.
Then by \eqref{eq.1.03} with $m_{i}>0$ for $i=1,\dots,I$, \eqref{eq.3.001} and \eqref{eq.3.003},
\begin{align*}
0\leq-\epsilon\Delta\phi_{\epsilon}(x_{0})=\sum_{i=1}^{I}\frac{m_{i}z_{i}\exp(-z_{i}\phi_{\epsilon}(x_{0}))}{\int_\Omega\!\exp(-z_{i}\phi_{\epsilon}(y))\,\mathrm{d}y}=\sum_{i=1}^{I}\frac{m_{i}z_{i}}{\int_\Omega\!\exp(-z_{i}(\phi_{\epsilon}(y)-\phi_{\epsilon}(x_{0})))\,\mathrm{d}y}\leq\sum_{i=1}^{I}\frac{m_{i}z_{i}}{|\Omega|}=0,
\end{align*}
which implies $\phi_{\epsilon}(x)=\phi_{\epsilon}(x_{0})$ for $x\in\overline\Omega$. 
Hence by \eqref{eq.1.04}, $\phi_{bd,k}$ are equal but this contradicts the assumption that $\phi_{bd,k}$ are not equal.
Thus, $\phi_{\epsilon}$ must attain the maximum value at a boundary point $x_{k_{0}}\in\partial\Omega_{k_{0}}$ for some $k_{0}\in\{0,1,\dots,K\}$ and $\partial_\nu\phi_{\epsilon}(x_{k_{0}})\geq0$.
Moreover, from \eqref{eq.1.04} with $\gamma_{k_{0}}>0$, we obtain
\[\phi_{\epsilon}(x)\leq\phi_{\epsilon}(x_{k_{0}})=\phi_{bd,k_{0}}-\gamma_{k_{0}}\sqrt{\epsilon}\partial_\nu\phi_{\epsilon}(x_{k_{0}})\leq\phi_{bd,k_{0}}\leq\overline{\phi_{bd}}\quad\text{for}~x\in\overline{\Omega}.\]
Similarly, we can prove that $\phi_{\epsilon}$ attains the minimum value at a boundary point and hence $\phi_{\epsilon}\geq\underline{\phi_{bd}}$ in $\overline{\Omega}$. Therefore, we complete the proof of \Cref{proposition:3.01}.
\end{proof}

By \eqref{eq.3.003} and \Cref{proposition:3.01}, we prove $f_{\epsilon}$ satisfies conditions (A1)--(A2).
\begin{proposition}
\label{proposition:3.02}
Assume that $\phi_{bd,k}$ are not equal and \eqref{eq.1.03} holds true.
Function $f_{\epsilon}=f_{\epsilon}(\phi)$ is strictly decreasing on $\R$ and has a unique zero denoted by $\phi_{\epsilon}^{*}\in(\underline{\phi_{bd}},\overline{\phi_{bd}})$ (which depends on $\epsilon$). Moreover, $f_{\epsilon}$ satisfies
\begin{align}
\label{eq.3.006}
&|f_{\epsilon}(\phi)|\leq M\quad\text{for}~\phi\in[\underline{\phi_{bd}},\overline{\phi_{bd}}]~\text{and}~\epsilon>0,\\
\label{eq.3.007}
&f_{\epsilon}'(\phi)\leq-C_{3}^2\quad\text{for}~\phi\in\R~\text{and}~\epsilon>0,
\end{align}
where $M$ and $C_{3}$ are positive constants independent of $\epsilon>0$.
\end{proposition}
\noindent Hereafter $M>0$ denotes a generic constant independent of $\epsilon>0$.
\begin{proof}By \eqref{eq.3.005}, the integral term $A_{i,\epsilon}$ satisfies the following estimate:
\begin{align}
\label{eq.3.008}
\alpha_{1}\leq A_{i,\epsilon}=\int_{\Omega}\!\exp(-z_{i}\phi_{\epsilon}(y))\,\mathrm{d}y\leq\alpha_{2}\quad\text{for}~i=1,\dots,I,
\end{align}
where $\alpha_1$ and $\alpha_2$ are positive constants independent of $\epsilon$.
Note that $z_{i}\neq0$ for $i=1,\dots,I$.
Then from \eqref{eq.3.003}, we can use \eqref{eq.3.008} to find
\begin{align*}
f_{\epsilon}(\phi)=\sum_{i=1}^{I}\frac{m_{i}z_{i}}{A_{i,\epsilon}}\exp(-z_{i}\phi)&=\left(\sum_{z_{i}>0}+\sum_{z_{i}<0}\right)\frac{m_{i}z_{i}}{A_{i,\epsilon}}\exp(-z_{i}\phi)\\&\geq\sum_{z_{i}>0}\frac{m_{i}z_{i}}{\alpha_2}\exp(-z_{i}\phi)+\sum_{z_{i}<0}\frac{m_{i}z_{i}}{\alpha_1}\exp(-z_{i}\phi):=g_1(\phi)\quad\text{for}~\phi\in\R,\end{align*}
and
\begin{align*}
f_{\epsilon}(\phi)=\left(\sum_{z_{i}>0}+\sum_{z_{i}<0}\right)\frac{m_{i}z_{i}}{A_{i,\epsilon}}\exp(-z_{i}\phi)\leq\sum_{z_{i}>0}\frac{m_{i}z_{i}}{\alpha_1}\exp(-z_{i}\phi)+\sum_{z_{i}<0}\frac{m_{i}z_{i}}{\alpha_2}\exp(-z_{i}\phi):=g_2(\phi)\quad\text{for}~\phi\in\R,\end{align*}
which gives
\begin{align}
\label{eq.3.009}
g_1(\phi)\leq f_{\epsilon}(\phi)\leq g_2(\phi)\quad\text{for}~\phi\in\R.
\end{align}
Clearly, $g_1$ and $g_2$ are independent of $\epsilon$, strictly decreasing on $\R$ and
\begin{align}
\label{eq.3.010}
\lim_{\phi\to\pm\infty}g_i(\phi)=\mp\infty\quad\text{for}~i=1,2.
\end{align}
Hence by \eqref{eq.3.009}, $g_1(\overline{\phi_{bd}})\leq f_{\epsilon}(\phi)\leq g_2(\underline{\phi_{bd}})$ for $\phi\in[\underline{\phi_{bd}},\overline{\phi_{bd}}]$, which implies \eqref{eq.3.006}.
Moreover, by \eqref{eq.3.003} and \eqref{eq.3.008}, there exists $C_{3}>0$ independent of $\epsilon$ such that
\begin{align*}
f_{\epsilon}'(\phi)=-\sum_{i=1}^{I}\frac{m_{i}z_{i}^2}{A_{i,\epsilon}}\exp(-z_{i}\phi)\leq-\sum_{i=1}^{I}\frac{m_{i}z_{i}^2}{\alpha_1}\exp(-z_{i}\phi)\leq-C_{3}^2\quad\text{for}~\phi\in\R,
\end{align*}
which gives \eqref{eq.3.007}.
Here $C_{3}^2=\alpha_{1}^{-1}\min\left\{\sum\limits_{z_{i}<0}m_{i}z_{i}^2,\sum\limits_{z_{i}>0}m_{i}z_{i}^2\right\}$ comes from
\[\sum_{i=1}^{I}\frac{m_{i}z_{i}^2}{\alpha_1}\exp(-z_{i}\phi)\geq\sum_{z_{i}<0}\frac{m_{i}z_{i}^2}{\alpha_1}\quad\text{for}~\phi\geq0\quad\text{and}\quad\sum_{i=1}^{I}\frac{m_{i}z_{i}^2}{\alpha_1}\exp(-z_{i}\phi)\geq\sum_{z_{i}>0}\frac{m_{i}z_{i}^2}{\alpha_1}\quad\text{for}~\phi\leq0\]
because of $z_{i}z_j<0$ for some $i,j\in\{1,\dots,I\}$.
By \eqref{eq.3.007} and \eqref{eq.3.009}--\eqref{eq.3.010}, $f_{\epsilon}$ is strictly decreasing and has a unique zero $\phi_{\epsilon}^{*}\in\R$.
To complete the proof, it remains to show $\phi_{\epsilon}^{*}\in(\underline{\phi_{bd}},\overline{\phi_{bd}})$ for $\epsilon>0$.
If $\phi_{\epsilon}^{*}\geq\overline{\phi_{bd}}$ for some $\epsilon>0$, then by \eqref{eq.3.005}, $\phi_{\epsilon}^{*}\geq\phi_{\epsilon}(x)$ for $x\in\overline\Omega$.
Due to the strict decrease of $f_{\epsilon}$, we have $0=f_{\epsilon}(\phi_{\epsilon}^{*})\leq f_{\epsilon}(\phi_{\epsilon}(x))$ for $x\in\overline\Omega$.
Then by \eqref{eq.1.03} and \eqref{eq.3.003}, we obtain
\[0=\int_\Omega\!f_{\epsilon}(\phi_{\epsilon}^{*})\,\mathrm{d}x\leq\int_\Omega\!f_{\epsilon}(\phi_{\epsilon}(x))\,\mathrm{d}x=\int_\Omega\!\sum_{i=1}^{I}\frac{m_{i}z_{i}\exp(-z_{i}\phi_{\epsilon}(x))}{\int_\Omega\!\exp(-z_{i}\phi_{\epsilon}(y))\,\mathrm{d}y}\,\mathrm{d}x=\sum_{i=1}^{I}m_{i}z_{i}=0,\]
which implies $f_{\epsilon}(\phi_{\epsilon})\equiv f_{\epsilon}(\phi_{\epsilon}^{*})=0$ and hence $\phi_{\epsilon}\equiv\phi_{\epsilon}^{*}$ in $\overline\Omega$.
By \eqref{eq.1.04}, $\phi_{bd,k}$ are equal but this is impossible because $\phi_{bd,k}$ are not equal. Thus $\phi_{\epsilon}^{*}<\overline{\phi_{bd}}$ for $\epsilon>0$.
Similarly, we can prove $\phi_{\epsilon}^{*}>\underline{\phi_{bd}}$ for $\epsilon>0$.
Therefore, the proof of \Cref{proposition:3.02} is complete.
\end{proof}

By \Cref{proposition:3.01,proposition:3.02}, we prove the asymptotic limit of $\phi_{\epsilon}$ in $\Omega$ as below.
\begin{proposition}
\label{proposition:3.03}
There exists a constant $\phi_{0}^{*}\in[\underline{\phi_{bd}},\overline{\phi_{bd}}]$ such that $\dd\lim_{\epsilon\to0^+}\phi_{\epsilon}(x)=\phi_{0}^{*}$ for $x\in\Omega$ (up to a subsequence).
\end{proposition}
\begin{proof}We begin by proving
\begin{align}
\label{eq.3.011}
|\nabla\phi_{\epsilon}(x)|\leq M'/\sqrt{\epsilon}\quad\text{for}~x\in\Omega~\text{and}~0<\epsilon<(\delta(x))^2,\end{align}
where $M'>0$ is generic constant independent of $\epsilon>0$ and $\delta(x)=\dist(x,\partial\Omega)$ for $x\in\overline\Omega$.
Let $x_{0}\in\Omega$ be arbitrary and $B_{\sqrt{\epsilon}}(x_{0})\subseteq\Omega$ for $0<\epsilon<(\delta(x_{0}))^2$.
Set $y=(x-x_{0})/\sqrt{\epsilon}$ and $\tilde\phi_{\epsilon}(y)=\phi_{\epsilon}(x_{0}+\sqrt{\epsilon}y)$.
Then from \eqref{eq.3.001}, we have $-\Delta\tilde\phi_{\epsilon}=f_{\epsilon}(\tilde\phi_{\epsilon})$ in $B_1(0)$.
By the uniform boundedness of $\phi_{\epsilon}$ and $f_{\epsilon}(\phi_{\epsilon})$ (cf. \Cref{proposition:3.01,proposition:3.02}), we apply the elliptic $L^q$-estimate to obtain $\|\tilde\phi_{\epsilon}\|_{W^{2,q}(B_{1/2}(0))}\leq M'$ for $q>1$.
Then using Sobolev's compact embedding theorem, we get $\|\tilde\phi_{\epsilon}\|_{\C^{1,\alpha}(B_{1/4}(0))}\leq M'$ for $\alpha\in(0,1)$.
In particular, we have $|\nabla\tilde\phi_{\epsilon}(0)|\leq M'$, which gives $|\nabla\phi _\epsilon(x_{0})|\leq M'/\sqrt{\epsilon}$ and \eqref{eq.3.011}.

Suppose that $\Omega'$ is a smooth subdomain of $\Omega$ such that $\Omega'\subset\subset\Omega$.
We claim that
\begin{align}
\label{eq.3.012}
|\nabla\phi_{\epsilon}(x)|\leq\frac{M'}{\sqrt{\epsilon}}\exp\left(-\frac{\sqrt2C_{3}\dist(\Omega',\partial\Omega)}{16\sqrt{\epsilon}}\right)\quad\text{for}~x\in\Omega'~\text{and}~0<\epsilon<\epsilon^{*},
\end{align}
where $\epsilon^{*}>0$ is a sufficiently small constant depending on $\Omega'$ and $C_{3}>0$ is given in \eqref{eq.3.007}.
Let $x_1\in\Omega'$ be arbitrary.
Clearly, $B_{R_1}(x_1)\subset\subset\Omega$ with $R_1=\dist(\Omega',\partial\Omega)/2$.
Then we use \eqref{eq.3.001} to obtain
\[\epsilon\Delta|\nabla\phi_{\epsilon}|^2=2\epsilon\sum_{i,j=1}^d\left(\frac{\partial^2\phi_{\epsilon}}{\partial x^j\partial x^i}\right)^2+2\epsilon\sum_{j=1}^d\frac{\partial\phi_{\epsilon}}{\partial x^j}\frac{\partial}{\partial x^j}\Delta\phi_{\epsilon}\geq-2\nabla\phi_{\epsilon}\cdot\nabla(f_{\epsilon}(\phi_{\epsilon}))=-2f_{\epsilon}'(\phi_{\epsilon})|\nabla\phi_{\epsilon}|^2\quad\text{in}~B_{R_1}(x_1).\]
Along with \eqref{eq.3.005} and \eqref{eq.3.007}, we obtain $\epsilon\Delta|\nabla\phi_{\epsilon}|^2\geq2C_{3}^2|\nabla\phi_{\epsilon}|^2$ in $B_{R_1}(x_1)$.
Let $\overline\phi_{\epsilon}$ be the solution to $\epsilon\Delta\overline\phi_{\epsilon}=2C_{3}^2\overline\phi_{\epsilon}$ in $B_{R_1}(x_1)$ with the Dirichlet boundary condition $\overline\phi_{\epsilon}=\max\limits_{\partial B_{R_1}(x_1)}|\nabla\phi_{\epsilon}|^2$.
Then the standard comparison principle yields
\begin{align}
\label{eq.3.013}
|\nabla\phi_{\epsilon}(x)|^2\leq|\overline\phi_{\epsilon}(x)|\leq2\left(\max_{\partial B_{R_1}(x_1)}|\nabla\phi_{\epsilon}|^2\right)\exp\left(-\frac{\sqrt2C_{3}\dist(x,\partial B_{R_1}(x_1))}{8\sqrt{\epsilon}}\right)\end{align}
for $x\in\overline B_{R_1}(x_1)$ and $0<\epsilon<\epsilon^{*}$. By \eqref{eq.3.011} and \eqref{eq.3.013}, there exists $\epsilon^{*}>0$ (depending on $\Omega'$) such that
\[|\nabla\phi_{\epsilon}(x_1)|\leq\frac{M'}{\sqrt{\epsilon}}\exp\left(-\frac{\sqrt2C_{3}R_1}{8\sqrt{\epsilon}}\right)=\frac{M'}{\sqrt{\epsilon}}\exp\left(-\frac{\sqrt2C_{3}\dist(\Omega',\partial\Omega)}{16\sqrt{\epsilon}}\right)\]
for $0<\epsilon<\epsilon^{*}$, which gives \eqref{eq.3.012}.

To complete the proof, let $\{\epsilon_{n}\}_{n=1}^\infty$ be a sequence of positive numbers with $\dd\lim_{n\to\infty}\epsilon_{n}=0$.
Since the sequence $\{\phi_{\epsilon_{n}}(x_1)\}_{n=1}^\infty$ is bounded (cf. \Cref{proposition:3.01}), by the Bolzano--Weierstrass theorem, there exists a subsequence $\{\epsilon_{n_k}\}_{k=1}^\infty$ of $\{\epsilon_{n}\}_{n=1}^\infty$ such that $\{\phi_{\epsilon_{n_k}}(x_1)\}_{k=1}^\infty$ converges to a number denoted by $\phi_{0}^{*}$.
Let $y\in\Omega'$ arbitrarily.
Then by \eqref{eq.3.012} and \Cref{proposition:3.01}, we have $|\phi_{\epsilon_{n_k}}(y)-\phi_{\epsilon_{n_k}}(x_1)|\to0$ as $k\to\infty$, which implies $\dd\lim_{k\to\infty}\phi_{\epsilon_{n_k}}(y)=\phi_{0}^{*}$.
Due to the arbitrary choice of $\Omega'$, this finalizes the proof of \Cref{proposition:3.03}.
\end{proof}

\begin{remark}
\label{remark:4}
In \Cref{section:3.2}, we will prove that $\phi_{0}^{*}\in(\underline{\phi_{bd}},\overline{\phi_{bd}})$ is uniquely determined by the algebraic equations \eqref{eq.3.026}--\eqref{eq.3.027}, which improves upon \Cref{proposition:3.03} to yield
\begin{align}
\label{eq.3.014}
\lim_{\epsilon\to0^+}\phi_{\epsilon}(x)=\phi_{0}^{*}\quad\text{for}~x\in\Omega.
\end{align}
Moreover, we will use the first-order asymptotic expansion of $\nabla\phi_{\epsilon}$ near the boundary to improve the interior exponential-type estimate \eqref{eq.3.012} (in the proof of \Cref{proposition:3.03}) and obtain the global exponential-type estimate \eqref{eq.3.037} in \Cref{proposition:3.07}.
\end{remark}
By \eqref{eq.3.014} and the uniform boundedness of $\phi_{\epsilon}$ (cf. \Cref{proposition:3.01}), we apply Lebesgue's dominated convergence theorem to get
\begin{align}
\label{eq.3.015}
&\lim_{\epsilon\to0^+}f_{\epsilon}(\phi)=\sum_{i=1}^{I}\frac{m_{i}z_{i}\exp(-z_{i}\phi)}{\dd\lim_{\epsilon\to0^+}\int_{\Omega}\!\exp(-z_{i}\phi_{\epsilon}(y))\,\mathrm{d}y}=\frac{1}{|\Omega|}\sum_{i=1}^{I}m_{i}z_{i}\exp(-z_{i}(\phi-\phi_{0}^{*})):=f_{0}(\phi)
\end{align}
for $\phi\in\R$, and
\begin{align}
\label{eq.3.016}
\lim_{\epsilon\to0^+}f_{\epsilon}'(\phi)=-\sum_{i=1}^{I}\frac{m_{i}z_{i}^2\exp(-z_{i}\phi)}{\dd\lim_{\epsilon\to0^+}\int_{\Omega}\!\exp(-z_{i}\phi_{\epsilon}(y))\,\mathrm{d}y}=-\frac{1}{|\Omega|}\sum_{i=1}^{I}m_{i}z_{i}^2\exp(-z_{i}(\phi-\phi_{0}^{*}))=f_{0}'(\phi)
\end{align}for $\phi\in\R$.
Note that $\phi_{0}^{*}$ may depend on the choice of sequence $\{\epsilon_{n}\}_{n=1}^\infty$, so may $f_{0}$.
By \eqref{eq.1.03}, \eqref{eq.3.007}, and \eqref{eq.3.015}--\eqref{eq.3.016}, it is clear that $\phi_{0}^{*}$ is the unique zero of $f_{0}$, and $f_{0}$ satisfies
\begin{align}
\label{eq.3.017}
f_{0}'(\phi)\leq-C_{3}^{2}\quad\text{for}~\phi\in\R.
\end{align}

\subsection{First-order asymptotic expansion of \texorpdfstring{$\phi_{\epsilon}$}{ϕ.ϵ}}
\label{section:3.2}

To obtain the first-order asymptotic expansion of $\phi_{\epsilon}$ near the boundary $\partial\Omega_k$, we define
\begin{align}
\label{eq.3.018}
u_{k,p,\epsilon}(z)=\phi_{\epsilon}(\Psi_{p}(\sqrt{\epsilon}z))\quad\text{for}~z\in B_{b/\sqrt{\epsilon}}^{+}~\text{and}~p\in\partial\Omega_k,\end{align}where $\phi_{\epsilon}$ is the solution to \eqref{eq.3.001}--\eqref{eq.3.002} and $\Psi_p$ is defined in \eqref{eq.2.01}.
As in \eqref{eq.2.08}--\eqref{eq.2.09}, we substitute \eqref{eq.2.02}--\eqref{eq.2.03} into \eqref{eq.3.001}--\eqref{eq.3.002} and get
\begin{alignat}{2}
\label{eq.3.019}
\sum_{i,j=1}^da_{ij}(z)\frac{\partial^2u_{k,p,\epsilon}}{\partial z^{i}\partial z^j}+\sum_{j=1}^db_j(z)\frac{\partial u_{k,p,\epsilon}}{\partial z^{j}}+f_{\epsilon}(u_{k,p,\epsilon})&=0\quad&&\text{in}~B_{b/\sqrt{\epsilon}}^{+},\\
\label{eq.3.020}
u_{k,p,\epsilon}-\gamma_{k}\partial_{z^{d}}u_{k,p,\epsilon}&=\phi_{bd,k}\quad&&\text{on}~\overline B_{b/\sqrt{\epsilon}}^{+}\cap\partial\R_{+}^{d},
\end{alignat}
where $a_{ij}$ and $b_j$ are given in \eqref{eq.2.05}--\eqref{eq.2.06} and $f_{\epsilon}$ is defined in \eqref{eq.3.003}.
As for \Cref{lemma:2.3,lemma:2.4}, we may use the uniform boundedness of $\phi_{\epsilon}$ (cf. \Cref{proposition:3.01}) and \eqref{eq.3.015} to prove
\begin{lemma}
\label{lemma:3.04}
For any sequence $\{\epsilon_{n}\}_{n=1}^\infty$ of positive numbers with $\dd\lim_{n\to\infty}\epsilon_{n}=0$, $\alpha\in(0,1)$, and $p\in\partial\Omega_k$ ($k\in\{0,1,\dots,K\}$), there exists a subsequence $\{\epsilon_{nn}\}_{n=1}^\infty$ such that $\dd\lim_{n\to\infty}\|u_{k,p,\epsilon_{nn}}-u_{k,p}\|_{\C^{2,\alpha}(\overline B_m^+)}=0$ for $m\in\N$, where $u_{k,p}\in\C^{2,\alpha}_{\text{loc}}(\overline\R_{+}^{d})$ satisfies
\begin{alignat}{2}
\label{eq.3.021}
\Delta u_{k,p}+f_{0}(u_{k,p})&=0\quad&&\text{in}~\R_{+}^{d},\\
\label{eq.3.022}
u_{k,p}-\gamma_k\partial_{z^{d}}u_{k,p}&=\phi_{bd,k}\quad&&\text{on}~\partial\R_{+}^{d}.
\end{alignat}
Moreover, $u_{k,p}$ satisfies the following exponential-type estimate
\begin{align}
\label{eq.3.023}
|u_{k,p}(z)-\phi_{0}^{*}|\leq2(\overline{\phi_{bd}}-\underline{\phi_{bd}})\exp(-C_{3}z^{d}/8)\end{align}
for $z=(z',z^{d})\in\overline\R_{+}^{d}$ and $z^{d}\geq2(d-1)/C_{3}$,
where $\phi_{0}^{*}$ is defined in \Cref{proposition:3.03} and $C_{3}$ is the positive constant given in \eqref{eq.3.007}.
\end{lemma}

From \Cref{lemma:3.04}, the solution $u_{k,p}$ to \eqref{eq.3.021}--\eqref{eq.3.022} may, a priori, depend on the sequence $\{\epsilon_{n}\}_{n=1}^\infty$ and on the point $p\in\partial\Omega_{k}$.
To establish the independence, we first use the moving plane arguments (as in \Cref{proposition:2.5}) to prove that $u_{k,p}$ satisfies $u_{k,p}(z)=u_{k}(z^{d})$ for $z=(z',z^{d})\in\overline\R_{+}^{d}$, where $u_{k}$ is the unique solution to \eqref{eq.1.19}--\eqref{eq.1.21}, and $u_{k}$ is independent of the point $p\in\partial\Omega_k$.
However, unlike \Cref{section:2}, the solution $u_{k}$ to \eqref{eq.1.19}--\eqref{eq.1.21} may still depend on the sequence $\{\epsilon_{n}\}_{n=1}^\infty$ because $\phi_{0}^{*}$ (the unique zero of $f_{0}$) may depend on the sequence $\{\epsilon_{n}\}_{n=1}^{\infty}$ (see \Cref{proposition:3.03}).
To remove the dependence on the sequence, we prove that $\phi_{0}^{*}$ is the unique solution to the algebraic equations \eqref{eq.3.026}--\eqref{eq.3.027}, which are independent of the sequence $\{\epsilon_{n}\}_{n=1}^{\infty}$, in \Cref{proposition:3.06}.
Consequently, as in \eqref{eq.2.19}--\eqref{eq.2.20}, we may apply \Cref{proposition:3.05,proposition:3.06} to prove that $u_{k,p,\epsilon}$ (defined in \eqref{eq.3.018}) satisfies
\begin{align}
\label{eq.3.024}
\lim_{\epsilon\to0^+}\|u_{k,p,\epsilon}-u_{k}\|_{\C^{2,\alpha}(\overline{B}_{m}^{+})}=0\quad\text{for}~m\in\N~\text{and}~\alpha\in(0,1),
\end{align}
and
\begin{align}
\label{eq.3.025}
\lim_{\epsilon\to0^+}u_{k,p,\epsilon}(z)=u_{k}(z^{d})\quad\text{for}~z=(z',z^{d})\in\overline{\R}_{+}^{d},
\end{align}
which improve the results of \Cref{lemma:3.04}. Below are the details.
\begin{proposition}
\label{proposition:3.05}
For $p\in\partial\Omega_k$ and $k\in\{0,1,\dots,K\}$, the solution $u_{k,p}$ to \eqref{eq.3.021}--\eqref{eq.3.022} satisfies
\begin{enumerate}
\item[(a)] $u_{k,p}$ depends only on the variable $z^{d}$, i.e., $u_{k,p}(z)=u_{k,p}(z^{d})$ for $z=(z',z^{d})\in\overline\R_{+}^{d}$.
\item[(b)] $u_{k,p}$ is independent of $p$ and depends only on $k$, i.e., $u_{k,p}(z^{d})=u_{k}(z^{d})$ for $z^{d}\in[0,\infty)$, where $u_{k}$ is the unique solution to \eqref{eq.1.19}--\eqref{eq.1.21}.
\end{enumerate}
\end{proposition}

\noindent The proof of \Cref{proposition:3.05} is similar to that of \Cref{proposition:2.5}
because \eqref{eq.3.021}--\eqref{eq.3.022} have the same form as \eqref{eq.2.10}--\eqref{eq.2.11} with $f=f_{0}$, and $f_{0}$ satisfies (A1)--(A2) (cf. \Cref{proposition:3.03}).

For the uniqueness of $\phi_{0}^{*}$ in \Cref{remark:4}, we use \Cref{proposition:3.03,proposition:3.05} to prove
\begin{proposition}
\label{proposition:3.06}
The value of $\phi_{0}^{*}\in(\underline{\phi_{bd}},\overline{\phi_{bd}})$, defined in \Cref{proposition:3.03}, is uniquely determined by the following equations.
\begin{align}
\label{eq.3.026}
&\phi_{bd,k}-u_{k}(0)=\sgn(\phi_{bd,k}-\phi_{0}^{*})\gamma_k\sqrt{\frac2{|\Omega|}\sum_{i=1}^{I}m_{i}[\exp(-z_{i}(u_{k}(0)-\phi_{0}^{*}))-1]}\quad\text{for}~k=0,1,\dots,K,\\
\label{eq.3.027}
&\sum_{k=0}^{K}|\partial\Omega_{k}|u_{k}'(0)=\sum_{k=0}^K|\partial\Omega_{k}|\frac{\phi_{bd,k}-u_{k}(0)}{\gamma_{k}}=0,
\end{align}
where $u_{k}$ is the unique solution to \eqref{eq.1.19}--\eqref{eq.1.21}.
\end{proposition}
\begin{proof}It is clear that \eqref{eq.3.026} follows from \eqref{eq.B.10} (in \Cref{appendix:B}) with $\Phi_{bd}=\phi_{bd,k}$, $U_{0}=u_{k}(0)$, $\phi^{*}=\phi_{0}^{*}$, $\gamma=\gamma_{k}$ and $F(\phi)=F_{0}(\phi)=|\Omega|^{-1}\sum\limits_{i=1}^{I}m_{i}(1-\exp(-z_{i}(\phi-\phi_{0}^{*})))$.
On the other hand, we integrate \eqref{eq.3.001} over $\Omega$ and apply the divergence theorem to get
\begin{align}
\label{eq.3.028}
\sum_{k=0}^{K}\int_{\partial\Omega_{k}}\!\frac{\phi_{bd,k}-\phi_{\epsilon}}{\gamma_k}\,\mathrm{d}S=\sqrt{\epsilon}\int_{\partial\Omega}\!\partial_{\nu}\phi_{\epsilon}\,\mathrm{d}S=\frac{1}{\sqrt{\epsilon}}\int_{\Omega}\!\Delta\phi_{\epsilon}(x)\,\mathrm{d}x=-\frac{1}{\sqrt{\epsilon}}\int_{\Omega}\!f_{\epsilon}(\phi_{\epsilon})\,\mathrm{d}x=-\frac{1}{\sqrt{\epsilon}}\sum_{i=1}^{I}m_{i}z_{i}=0.
\end{align}
Here we have used the conditions \eqref{eq.1.03} and \eqref{eq.1.04}.
By \eqref{eq.2.04} and \eqref{eq.3.025}, $\dd\lim_{\epsilon\to0^+}\phi_{\epsilon}(p)=\lim_{\epsilon\to0^+}u_{k,p,\epsilon}(0)=u_{k}(0)$ for $p\in\partial\Omega_k$. Thus, letting $\epsilon\to0^+$, \eqref{eq.3.028} implies \eqref{eq.3.027}. Hence $\phi_{0}^{*}$ satisfies \eqref{eq.3.026}--\eqref{eq.3.027}.

We now prove $\phi_{0}^{*}\in(\underline{\phi_{bd}},\overline{\phi_{bd}})$. Recall that $\phi_{0}^{*}\in[\underline{\phi_{bd}},\overline{\phi_{bd}}]$ (cf. \Cref{proposition:3.03}).
If $\phi_{0}^{*}=\overline{\phi_{bd}}$, then $\phi_{0}^{*}\geq\phi_{bd,k}$ for $k=0,1,\dots,K$.
When $\phi_{0}^{*}=\phi_{bd,k}$, $u_{k}\equiv\phi_{0}^{*}$ on $[0,\infty)$; when $\phi_{0}^{*}>\phi_{bd,k}$, $u_{k}$ is strictly increasing on $[0,\infty)$ and $u_{k}(0)>\phi_{bd,k}$ (cf. \Cref{proposition:B.1}).
Thus, $u_{k}(0)\geq\phi_{bd,k}$ for $k=0,1,\dots,K$.
Along with \eqref{eq.3.027}, $u_{k}(0)=\phi_{bd,k}$ for $k=0,1,\dots,K$, which means $\phi_{0}^{*}=\phi_{bd,k}$ for $k=0,1,\dots,K$.
This cannot happen because $\phi_{bd,k}$ are not equal.
Hence $\phi_{0}^{*}<\overline{\phi_{bd}}$.
Similarly, we can prove $\phi_{0}^{*}\neq\underline{\phi_{bd}}$.
Consequently, $\phi_{0}^{*}\in(\underline{\phi_{bd}},\overline{\phi_{bd}})$.

To complete the proof, it remains to show that algebraic equations \eqref{eq.3.026}--\eqref{eq.3.027} admit a unique solution $\phi_{0}^{*}$.
Suppose that there exist $\{\epsilon_{n}\}_{n=1}^\infty$ and $\{\tilde\epsilon_{n}\}_{n=1}^\infty$ such that $\dd\lim_{n\to\infty}\phi_{\epsilon_{n}}^{*}=\phi_{0}^{*}$ and $\dd\lim_{n\to\infty}\phi_{\tilde\epsilon_{n}}^{*}=\tilde\phi_{0}^{*}$.
For $k=0,1,\dots,K$, let $u_{k}$ be the solutions to
\begin{align*}
&u_{k}''+f_{0}(u_{k})=0\quad\text{in}~(0,\infty),\\
&u_{k}(0)-\gamma_ku_{k}'(0)=\phi_{bd,k},\\
&\lim_{t\to\infty}u_{k}(t)=\phi_{0}^{*},
\end{align*}
and let $\tilde u_{k}$ be the solutions to
\begin{align*}
&\tilde u_{k}''+\tilde f_{0}(\tilde u_{k})=0\quad\text{in}~(0,\infty),\\
&\tilde u_{k}(0)-\gamma_k\tilde u_{k}'(0)=\phi_{bd,k},\\
&\lim_{t\to\infty}\tilde u_{k}(t)=\tilde\phi_{0}^{*}.
\end{align*}
where
\[f_{0}(\phi)=\frac{1}{|\Omega|}\sum_{i=1}^{I}m_{i}z_{i}\exp(-z_{i}(\phi-\phi_{0}^{*}))\quad\text{and}\quad\tilde f_{0}(\phi)=\frac{1}{|\Omega|}\sum_{i=1}^{I}m_{i}z_{i}\exp(-z_{i}(\phi-\tilde\phi_{0}^{*}))\quad\text{for}~\phi\in\R.\]
As $\phi_{0}^{*}=\tilde\phi_{0}^{*}$, it is clear that $f_{0}=\tilde f_{0}$ and $u_{k}=\tilde u_{k}$ for $k=0,1,\dots,K$.
Note that $(u_{0}(0),u_{1}(0),\dots,u_{K}(0),\phi_{0}^{*})$ and $(\tilde u_{0}(0),\tilde u_1(0),\dots,\tilde u_{K}(0),\tilde\phi_{0}^{*})$ satisfy \eqref{eq.3.026}--\eqref{eq.3.027}.
Consequently, the values of $u_{k}(0)$ and $\tilde u_{k}(0)$ depend on $\phi_{0}^{*}$ and $\tilde\phi_{0}^{*}$, respectively.
Now we prove
\begin{claim}
\label{claim:2}
$\phi_{0}^{*}>\tilde{\phi}_{0}^{*}$ ($\phi_{0}^{*}<\tilde\phi_{0}^{*}$, resp.) if and only if $u_{k}(0)>\tilde{u}_{k}(0)$ ($u_{k}(0)<\tilde{u}_{k}(0)$, resp.) for all $k=0,1,\dots,K$.
\end{claim}
\begin{proof}[Proof of \Cref{claim:2}]
We first suppose $\phi_{0}^{*}>\tilde{\phi}_{0}^{*}$.
Let
\[g(\phi,s)=\frac{1}{|\Omega|}\sum_{i=1}^{I}m_{i}z_{i}\exp(-z_{i}(\phi-s))\quad\text{for}~\phi,s\in\R.\]
It is obvious that $g(\phi,\phi_{0}^{*})=f_{0}(\phi)$ and $g(\phi,\tilde\phi_{0}^{*})=\tilde f_{0}(\phi)$ for $\phi\in\R$.
Since \[\partial_sg(\phi,s)=\frac{1}{|\Omega|}\sum_{i=1}^{I}m_{i}z_{i}^2\exp(-z_{i}(\phi-s))>0\quad\text{for}~s\in\R,\]the function $g$ is increasing in $s$ for $\phi\in\R$.
Hence by the assumption $\phi_{0}^{*}>\tilde\phi_{0}^{*}$, we get
\begin{align}
\label{eq.3.029}
f_{0}(\phi)>\tilde f_{0}(\phi)\quad\text{for}~\phi\in\R.
\end{align}
Suppose, to the contrary, that $u_{k_{0}}(0)\leq\tilde u_{k_{0}}(0)$ for some $k_{0}\in\{0,1,\dots,K\}$.
Let $\overline u_{k_{0}}=u_{k_{0}}-\tilde{u}_{k_{0}}$ on $[0,\infty)$.
Then it is easy to verify that $\overline{u}_{k_0}(0)\leq0$ and
\begin{align}
\label{eq.3.030}
&\overline{u}_{k_{0}}''+f_{0}(u_{k_{0}})-\tilde{f}_{0}(\tilde{u}_{k_{0}})=0\quad\text{in}~(0,\infty),\\
\label{eq.3.031}
&\overline{u}_{k_{0}}(0)-\gamma_k\overline{u}_{k_{0}}'(0)=0,\\
\label{eq.3.032}
&\dd\lim_{t\to\infty}\overline{u}_{k_{0}}(t)=\phi_{0}^{*}-\tilde{\phi}_{0}^{*}>0.
\end{align}
Since $\overline{u}_{k_{0}}(0)\leq0$, we have $\overline{u}_{k_{0}}'(0)\leq0$ by \eqref{eq.3.031}.
Moreover, since $\tilde{f}_{0}$ is strictly decreasing, we use \eqref{eq.3.029}--\eqref{eq.3.030} to obtain
\[\overline{u}_{k_{0}}''(0)=\tilde{f}_{0}(\tilde{u}_{k_0}(0))-f_{0}(u_{k_0}(0))\leq\tilde f_{0}(u_{k_0}(0))-f_{0}(u_{k_{0}}(0))<0.\]
Thus, by \eqref{eq.3.032}, there exists $t_1\in(0,\infty)$ such that $\overline u_{k_0}(t)<0$ on $(0,t_1)$, $\overline{u}_{k_{0}}(t_1)=0$ and $\overline u_{k_0}'(t_1)\geq0$.
Integrating \eqref{eq.3.030} over $[0,t_1]$, we get
\begin{align*}
\overline{u}_{k_{0}}'(t_1)-\overline{u}_{k_{0}}'(0)+\int_{0}^{t_{1}}\![f_{0}(u_{k_{0}}(s))-\tilde{f}_{0}(\tilde{u}_{k_{0}}(s))]\,\mathrm{d}s=0,
\end{align*}
which implies
\begin{align}
\label{eq.3.033}
\int_{0}^{t_{1}}\![f_{0}(u_{k_0}(s))-\tilde{f}_{0}(\tilde{u}_{k_0}(s))]\,\mathrm{d}s=\overline{u}_{k_{0}}'(0)-\overline{u}_{k_{0}}'(t_{1})\leq0.
\end{align}
Since $f_{0}$ is strictly decreasing on $\R$ (cf. \eqref{eq.3.017}) and $u_{k_{0}}(t)-\tilde u_{k_0}(t)=\overline u_{k_0}(t)<0$ on $(0,t_1)$, it follows that
\begin{align*}
f_{0}(u_{k_0}(t))-\tilde f_{0}(\tilde u_{k_0}(t))>f_{0}(\tilde u_{k_0}(t))-\tilde f_{0}(\tilde u_{k_0}(t))>0\quad\text{for}~0<t<t_1,
\end{align*}
which contradicts \eqref{eq.3.033}.
Here we have used \eqref{eq.3.029} in the last inequality.
Therefore, we arrive at $\overline u_{k}(0)>0$, which means $u_{k}(0)>\tilde u_{k}(0)$ for all $k=0,1,\dots,K$.
Similarly, we can prove $\phi_{0}^{*}<\tilde\phi_{0}^{*}$ implies $u_{k}(0)<\tilde u_{k}(0)$ and thus complete proof of \Cref{claim:2}.
\end{proof}
To complete the proof of \Cref{proposition:3.06}, we combine \Cref{claim:2} and \eqref{eq.3.027} for $u_{k}$ and $\tilde u_{k}$ to get the following contradiction:
\[0=\sum_{k=0}^K|\partial\Omega_k|\frac{\phi_{bd,k}-u_{k}(0)}{\gamma_k}\neq\sum_{k=0}^K|\partial\Omega_k|\frac{\phi_{bd,k}-\tilde u_{k}(0)}{\gamma_k}=0\quad\text{if}~\phi_{0}^{*}\neq\tilde\phi_{0}^{*},\]
which implies $\phi_{0}^{*}=\tilde\phi_{0}^{*}$ and $u_{k}(0)=\tilde u_{k}(0)$.
Therefore, we complete the proof of \Cref{proposition:3.06}.
\end{proof}

Consequently, we can use \eqref{eq.2.01}, \eqref{eq.2.04}, \eqref{eq.3.018}, and \eqref{eq.3.024}--\eqref{eq.3.025} to obtain the first-order asymptotic expansions of $\phi_{\epsilon}$ and $\nabla\phi_{\epsilon}$:
\begin{align}
\label{eq.3.034}
&\phi_{\epsilon}(p-t\sqrt{\epsilon}\nu_{p})=u_{k}(t)+o_{\epsilon}(1),\\
\label{eq.3.035}
&\nabla\phi_{\epsilon}(p-t\sqrt{\epsilon}\nu_{p})=-\frac{1}{\sqrt{\epsilon}}(u_{k}'(t)\nu_{p}+o_{\epsilon}(1))
\end{align}
for $0\leq t\leq T$ as $\epsilon\to0^{+}$, where $T>0$, $k\in\{0,1,\dots,K\}$, and $p\in\partial\Omega_k$ (as in \eqref{eq.2.39}--\eqref{eq.2.40}).

To derive the second-order asymptotic expansions of $\phi_{\epsilon}$ and $\nabla\phi_{\epsilon}$ and capture the asymptotic behavior of the term $o_{\epsilon}(1)/\sqrt{\epsilon}$ in the subsequent analysis, it is crucial to establish the following exponential-type estimates of $\phi_{\epsilon}$ and $\nabla\phi_{\epsilon}$ analogously to \Cref{proposition:2.6}:
\begin{proposition}[Exponential-type estimates of $\phi_{\epsilon}$ and $\nabla\phi_{\epsilon}$]
\label{proposition:3.07}
Under the hypothesis of \Cref{theorem:2}, let $\Omega$ satisfy the uniform interior sphere condition, i.e., there exists $R>0$ such that $B_R(p-R\nu_{p})\subseteq\Omega$ and $\partial B_R(p-R\nu_{p})\cap\partial\Omega=\{p\}$ for $p\in\partial\Omega$, where $\nu_{p}$ is the unit outer normal of $\partial\Omega$ at $p$.
Then we have
\begin{align}
\label{eq.3.036}
&|\phi_{\epsilon}(x)-\phi_{\epsilon}^{*}|\leq2(\overline{\phi_{bd}}-\underline{\phi_{bd}})\exp\left(-\frac{C_{3}\delta(x)}{8\sqrt{\epsilon}}\right),\\
\label{eq.3.037}
&|\nabla\phi_{\epsilon}|\leq\frac{M'}{\sqrt{\epsilon}}\exp\left(-\frac{M\delta(x)}{\sqrt{\epsilon}}\right)
\end{align}
for $x\in\overline\Omega$ and $0<\epsilon<\epsilon^{*}$, where $C_{3}$ (given in \eqref{eq.3.007}) is independent of $\epsilon$ and $\phi_{\epsilon}^{*}$ is defined in \Cref{proposition:3.02}.
\end{proposition}
\begin{proof}
As in the proof of \Cref{proposition:2.6}, the uniform interior sphere condition of $\Omega$ ensures that there exists $B_{R_0}(x_0)\subseteq\Omega$ with $R_0\geq R$; hence \eqref{eq.2.43} holds.
To show \eqref{eq.3.036}, we let $\phi_{\epsilon}^\pm$ be the solutions to
\begin{alignat}{2}
\label{eq.3.038}
-\epsilon\Delta\phi_{\epsilon}^\pm&=f_{\epsilon}(\phi_{\epsilon}^\pm)\quad&&\text{in}~B_{R_0}(x_0),\\
\label{eq.3.039}
\phi_{\epsilon}^\pm&=\phi_{\epsilon}^{*}\pm(\overline{\phi_{bd}}-\underline{\phi_{bd}})\quad&&\text{on}~\partial B_{R_0}(x_0),
\end{alignat}
The existence of solutions to \eqref{eq.3.038}--\eqref{eq.3.039} follows from the standard variational method because of the strict decrease of $f_{\epsilon}$ (cf. \Cref{proposition:3.02}).
Following the argument of \Cref{proposition:2.6} for \eqref{eq.2.46}, with replacing $f$ by $f_{\epsilon}$, we obtain $\phi_{\epsilon}^-\leq\phi_{\epsilon}\leq\phi_{\epsilon}^+$ in $\overline B_{R_0}(x_0)$.
Then by \eqref{eq.2.43} and \Cref{proposition:A.4} in \Cref{appendix:A} with $M=C_{3}$, $|\Phi_{bd}-\phi_{\epsilon}^{*}|=\overline{\phi_{bd}}-\underline{\phi_{bd}}$ and $B_R(0)=B_{R_0}(x_0)$, we have
\begin{align}
\label{eq.3.040}
|\phi_{\epsilon}^\pm(x)-\phi_{\epsilon}^{*}|\leq2(\overline{\phi_{bd}}-\underline{\phi_{bd}})\exp\left(-\frac{C_{3}\delta(x)}{8\sqrt{\epsilon}}\right)\quad\text{for}~x\in\overline B_{R_0}(x_0)~\text{and}~0<\epsilon<\epsilon^{*},
\end{align}
which implies \eqref{eq.3.036}.
As in \eqref{eq.2.50}, we can use \eqref{eq.3.035} to improve the estimate \eqref{eq.3.011} as $|\nabla\phi_{\epsilon}(x)|\leq M'/\sqrt{\epsilon}$ for $x\in\overline\Omega$ and $0<\epsilon<(\delta(x))^2$.
Then \eqref{eq.3.037} follows from \eqref{eq.3.007} and the standard comparison principle on \eqref{eq.2.51} (with replacing $f$ by $f_{\epsilon}$).
Therefore, we complete the proof of \Cref{proposition:3.07}.
\end{proof}

Recall that $\delta(x)=\dist(x,\partial\Omega)$.
By \Cref{proposition:3.07}, it is clear that \Cref{theorem:2}(b) holds true.

\subsection{Uniform boundedness of \texorpdfstring{$|\phi_{\epsilon}^{*}-\phi_{0}^{*}|/\sqrt{\epsilon}$}{|ϕ.ϵ*-ϕ.0*|/√ϵ}}
\label{section:3.3}

From \eqref{eq.3.014}--\eqref{eq.3.015} and \eqref{eq.3.036}, the first-order asymptotic expansions of $\phi_{\epsilon}^{*}$, $A_{i,\epsilon}$, and $f_{\epsilon}$ are represented~by
\begin{alignat*}{2}
\phi_{\epsilon}^{*}&=\phi_{0}^{*}+o_{\epsilon}(1),\\
A_{i,\epsilon}&=|\Omega|\exp(-z_{i}\phi_{0}^{*})+o_{\epsilon}(1)&&\quad\text{for}~i=1,\dots,I,\\
f_{\epsilon}(\phi)&=f_{0}(\phi)+o_{\epsilon}(1)&&\quad\text{for}~\phi\in\R,
\end{alignat*}
where $A_{i,\epsilon}=\int_{\Omega}\!\exp(-z_{i}\phi_{\epsilon}(y))\,\mathrm{d}y$, $f_{\epsilon}(\phi)=\sum_{i=1}^{I}(m_{i}z_{i}/A_{i,\epsilon})\exp(-z_{i}\phi)$, and $f_{0}(\phi)$ is given in \eqref{eq.1.28}.
To obtain the second-order asymptotic expansion of $\phi_{\epsilon}$, we require the second-order asymptotic expansions of $f_{\epsilon}$ and $A_{i,\epsilon}$.
Thus, it is necessary to show the uniform boundedness of $(\phi_{\epsilon}^{*}-\phi_{0}^{*})/\sqrt{\epsilon}$.

Suppose by contradiction that there exists a sequence $\{\epsilon_{n}\}_{n=1}^\infty$ such that $\dd\lim_{n\to\infty}|\phi_{\epsilon_{n}}^{*}-\phi_{0}^{*}|/\sqrt{\epsilon_{n}}=\infty$ and $\dd\lim_{n\to\infty}\epsilon_{n}=0$.
For notational convenience, we henceforth drop the subscript $n$ in this section so the assumption becomes
\begin{align}
\label{eq.3.041}
\lim_{\epsilon\to0^+}\frac{|\phi_{\epsilon}^{*}-\phi_{0}^{*}|}{\sqrt{\epsilon}}=\infty.\end{align}
A direct computation gives
\begin{align}
\label{eq.3.042}
\begin{aligned}
f_{1,\epsilon}(\phi)&:=\frac{f_{\epsilon}(\phi)-f_{0}(\phi)}{\phi_{\epsilon}^{*}-\phi_{0}^{*}}\\&=\frac{1}{\phi_{\epsilon}^{*}-\phi_{0}^{*}}\sum_{i=1}^{I}m_{i}z_{i}\left(\frac{\exp(-z_{i}\phi_{0}^{*})}{A_{i,\epsilon}}-\frac{1}{|\Omega|}\right)\exp(-z_{i}(\phi-\phi_{0}^{*}))\\
&=-\sum_{i=1}^{I}\frac{m_{i}z_{i}B_{i,\epsilon}}{|\Omega|[|\Omega|\exp(-z_{i}\phi_{0}^{*})+(\phi_{\epsilon}^{*}-\phi_{0}^{*})B_{i,\epsilon}]}\exp(-z_{i}(\phi-\phi_{0}^{*}))\quad\text{for}~\phi\in\R,
\end{aligned}
\end{align}
where
\begin{align}
\label{eq.3.043}
B_{i,\epsilon}=\frac{A_{i,\epsilon}-|\Omega|\exp(-z_{i}\phi_{0}^{*})}{\phi_{\epsilon}^{*}-\phi_{0}^{*}}=\int_\Omega\!\frac{\exp(-z_{i}\phi_{\epsilon}(y))-\exp(-z_{i}\phi_{0}^{*})}{\phi_{\epsilon}^{*}-\phi_{0}^{*}}\,\mathrm{d}y\end{align}
for $i=1,\dots,I$ and $\epsilon>0$.
Note that $B_{i,\epsilon}$ can be expressed by
\begin{align}
\label{eq.3.044}
B_{i,\epsilon}=J_{i,1}+\sum_{k=0}^K J_{i,k,2}\quad\text{for}~i=1,\dots,I,
\end{align}
where
\begin{align}
\label{eq.3.045}
&J_{i,1}=\int_{\overline\Omega_{\epsilon,\beta}}\!\frac{\exp(-z_{i}\phi_{\epsilon}(y))-\exp(-z_{i}\phi_{0}^{*})}{\phi_{\epsilon}^{*}-\phi_{0}^{*}}\,\mathrm{d}y,\\
\label{eq.3.046}
&J_{i,k,2}=\int_{\Omega_{k,\epsilon,\beta}}\!\frac{\exp(-z_{i}\phi_{\epsilon}(y))-\exp(-z_{i}\phi_{0}^{*})}{\phi_{\epsilon}^{*}-\phi_{0}^{*}}\,\mathrm{d}y\quad\text{for}~k=0,1,\dots,K.
\end{align}
Here $\overline\Omega_{\epsilon,\beta}=\{x\in\Omega\,:\,\delta(x)\geq\epsilon^\beta\}$ and $\Omega_{k,\epsilon,\beta}=\{x\in\Omega\,:\delta_k(x)<\epsilon^\beta\}$ for $0<\beta<1/2$ and $k=0,1,\dots,K$. Clearly, $\Omega_{k,\epsilon,\beta}$ are disjoint for sufficiently small $\epsilon>0$.
Then we may apply \eqref{eq.3.036}, \eqref{eq.3.041}, and \eqref{eq.3.043} to get
\begin{lemma}
\label{lemma:3.08}
$B_{i,\epsilon}=-z_{i}\exp(-z_{i}\phi_{0}^{*})|\Omega|+o_{\epsilon}(1)$ for $i=1,\dots,I$.
\end{lemma}
\noindent Henceforth $\epsilon^{*}>0$ is a sufficiently small constant, and $M$ and $M'$ are generic constants independent of $\epsilon$.
\begin{proof}
Fix $i\in\{1,\dots,I\}$.
In view of \eqref{eq.3.044}, it suffices to determine the limits of $J_{i,1}$ and $J_{i,k,2}$.
We begin with an analysis of $J_{i,1}$ when $z_{i}>0$.
By \eqref{eq.3.036}, we have
\[\exp\left[-z_{i}\phi_{\epsilon}^{*}-z_{i}M'\exp\left(-\frac{M\delta(x)}{\sqrt{\epsilon}}\right)\right]\leq\exp(-z_{i}\phi_{\epsilon}(x))\leq\exp\left[-z_{i}\phi_{\epsilon}^{*}+z_{i}M'\exp\left(-\frac{M\delta(x)}{\sqrt{\epsilon}}\right)\right]\]
for $x\in\overline\Omega$ and $0<\epsilon<\epsilon^{*}$, which implies
\begin{align}
\label{eq.3.047}
\begin{aligned}
&\exp(-z_{i}\phi_{0}^{*})\left\{\exp[-z_{i}M'\exp(-M\delta(x)/\sqrt{\epsilon})]\exp[-z_{i}(\phi_{\epsilon}^{*}-\phi_{0}^{*})]-1\right\}\\
\leq&\exp(-z_{i}\phi_{\epsilon}(x))-\exp(-z_{i}\phi_{0}^{*})\\
\leq&\exp(-z_{i}\phi_{0}^{*})\left\{\exp[z_{i}M'\exp(-M\delta(x)/\sqrt{\epsilon})]\exp[-z_{i}(\phi_{\epsilon}^{*}-\phi_{0}^{*})]-1\right\}
\end{aligned}
\end{align}
for $x\in\overline\Omega$ and $0<\epsilon<\epsilon^{*}$.
For $x\in\overline\Omega_{\epsilon,\beta}$ (i.e., $\delta(x)\geq\epsilon^{\beta}$) and $0<\epsilon<\epsilon^{*}$, \eqref{eq.3.047} becomes
\begin{align}
\label{eq.3.048}
\begin{aligned}
&\exp(-z_{i}\phi_{0}^{*})\left\{\exp[-z_{i}M'\exp(-M\epsilon^{(2\beta-1)/2})]\exp[-z_{i}(\phi_{\epsilon}^{*}-\phi_{0}^{*})]-1\right\}\\
\leq&\exp(-z_{i}\phi_{\epsilon}(x))-\exp(-z_{i}\phi_{0}^{*})\\
\leq&\exp(-z_{i}\phi_{0}^{*})\left\{\exp[z_{i}M'\exp(-M\epsilon^{(2\beta-1)/2})]\exp[-z_{i}(\phi_{\epsilon}^{*}-\phi_{0}^{*})]-1\right\}.
\end{aligned}
\end{align}
Now we can combine \eqref{eq.3.045} and \eqref{eq.3.048} to get
\begin{align*}
&\exp(-z_{i}\phi_{0}^{*})|\overline\Omega_{\epsilon,\beta}|\frac{\exp[-z_{i}M'\exp(-M\epsilon^{(2\beta-1)/2})]\exp[-z_{i}(\phi_{\epsilon}^{*}-\phi_{0}^{*})]-1}{\phi_{\epsilon}^{*}-\phi_{0}^{*}}\\\leq& J_{i,1}\leq\exp(-z_{i}\phi_{0}^{*})|\overline\Omega_{\epsilon,\beta}|\frac{\exp[z_{i}M'\exp(-M\epsilon^{(2\beta-1)/2})]\exp[-z_{i}(\phi_{\epsilon}^{*}-\phi_{0}^{*})]-1}{\phi_{\epsilon}^{*}-\phi_{0}^{*}}.
\end{align*}
Along with the fact that
\begin{align*}&\quad\lim_{\epsilon\to0^+}\frac{\exp[z_{i}M'\exp(-M\epsilon^{(2\beta-1)/2})]\exp[-z_{i}(\phi_{\epsilon}^{*}-\phi_{0}^{*})]-1}{\phi_{\epsilon}^{*}-\phi_{0}^{*}}\\&=\lim_{\epsilon\to0^+}\frac{\exp[z_{i}M'\exp(-M\epsilon^{(2\beta-1)/2})-z_{i}(\phi_{\epsilon}^{*}-\phi_{0}^{*})]-1}{\phi_{\epsilon}^*-\phi_{0}^{*}}=-z_{i},\end{align*}
which follows from \eqref{eq.3.041} and $0<\beta<1/2$, we obtain
\begin{align}
\label{eq.3.049}
J_{i,1}=-z_{i}\exp(-z_{i}\phi_{0}^{*})|\Omega|+o_{\epsilon}(1)\quad\text{for}~i=1,\dots,I.\end{align}
For the case of $z_{i}<0$, we can follow a similar argument to get \eqref{eq.3.049}.

Now it remains to estimate $J_{i,k,2}$.
Fix $k\in\{0,1,\dots,K\}$.
For $y\in\overline\Omega_{k,\epsilon,\beta}$ and sufficiently small $\epsilon>0$, there exist unique $p\in\partial\Omega_k$ and $0\leq s\leq\epsilon^{(2\beta-1)/2}$ such that $y=p-s\sqrt{\epsilon}\nu_{p}$.
By the coarea formula (cf. \cite{2015evans,2002lin}), \eqref{eq.3.046} becomes
\begin{align*}
J_{i,k,2}&=\int_{\Omega_{k,\epsilon,\beta}}\!\frac{\exp(-z_{i}\phi_{\epsilon}(y))-\exp(-z_{i}\phi_{0}^{*})}{\phi_{\epsilon}^{*}-\phi_{0}^{*}}\,\mathrm{d}y\\&
=\frac{\sqrt{\epsilon}}{\phi_{\epsilon}^{*}-\phi_{0}^{*}}\int_0^{\epsilon^{(2\beta-1)/2}}\!\!\!\!\!\int_{\partial\Omega_k}\![\exp(-z_{i}\phi_{\epsilon}(p-s\sqrt{\epsilon}\nu_{p}))-\exp(-z_{i}\phi_{0}^{*})]\mathcal{J}(s,p)\,\mathrm{d}S_{p}\,\mathrm{d}s
\end{align*}
for $i=1,\dots,I$ and $k=0,1,\dots,K$.
Along with \eqref{eq.2.75}, we have
\begin{align}
\label{eq.3.050}
|J_{i,k,2}|\leq\frac{2\sqrt{\epsilon}}{|\phi_{\epsilon}^{*}-\phi_{0}^{*}|}\int_0^{\epsilon^{(2\beta-1)/2}}\!\!\!\!\!\int_{\partial\Omega_k}\!|\exp(-z_{i}\phi_{\epsilon}(p-s\sqrt{\epsilon}\nu_{p}))-\exp(-z_{i}\phi_{0}^{*})|\,\mathrm{d}S_p\,\mathrm{d}s.
\end{align}
From \Cref{proposition:3.01,proposition:3.03}, $|\phi_{\epsilon}-\phi_{0}^{*}|$ is uniformly bounded in $\overline\Omega$ and $\epsilon>0$, which implies that there exists $M>1$ such that $|\exp[-z_{i}(\phi_{\epsilon}(x)-\phi_{0}^{*})]-1|\leq M|\phi_{\epsilon}(x)-\phi_{0}^{*}|$ for $x\in\overline{\Omega}$ and $0<\epsilon<\epsilon^{*}$.
Then by \eqref{eq.3.036}, we get
\begin{align}
\label{eq.3.051}
\begin{aligned}|\exp(-z_{i}\phi_{\epsilon}(x))-\exp(-z_{i}\phi_{0}^{*})|&=\exp(-z_{i}\phi_{0}^{*})|\exp[-z_{i}(\phi_{\epsilon}(x)-\phi_{0}^{*})]-1|\\&\leq M'\exp\left(-\frac{M\delta(x)}{\sqrt{\epsilon}}\right)+M'|\phi_{\epsilon}^{*}-\phi_{0}^{*}|\quad\text{for}~x\in\overline{\Omega}~\text{and}~0<\epsilon<\epsilon^{*}.\end{aligned}\end{align}
Combining \eqref{eq.3.041}, \eqref{eq.3.050} and \eqref{eq.3.051}, we arrive at
\begin{align}
\label{eq.3.052}
\begin{aligned}
|J_{i,k,2}|&\leq\frac{2\sqrt{\epsilon}}{|\phi_{\epsilon}^{*}-\phi_{0}^{*}|}\int_0^{\epsilon^{(2\beta-1)/2}}\!\!\!\!\!\int_{\partial\Omega_k}\![M'\exp(-Ms)+M'|\phi_{\epsilon}^{*}-\phi_{0}^{*}|]\,\mathrm{d}S_{p}\,\mathrm{d}s\\&=\frac{M'\sqrt{\epsilon}}{|\phi_{\epsilon}^{*}-\phi_{0}^{*}|}(1-\exp(M\epsilon^{(2\beta-1)/2}))+M'\epsilon^{\beta}\to0\quad\text{as}~\epsilon\to0^+,\end{aligned}
\end{align}
for $i=1,\dots,I$ and $k=0,1,\dots,K$.
Thus, combining \eqref{eq.3.044}, \eqref{eq.3.049}, and \eqref{eq.3.052}, we  complete the proof of \Cref{lemma:3.08}.
\end{proof}

Applying \Cref{lemma:3.08} to \eqref{eq.3.042}, we obtain
\begin{align*}f_{1,\epsilon}(\phi)&=-\sum_{i=1}^{I}\frac{m_{i}z_{i}B_{i,\epsilon}}{|\Omega|[|\Omega|\exp(-z_{i}\phi_{0}^{*})+(\phi_{\epsilon}^{*}-\phi_{0}^{*})B_{i,\epsilon}]}\exp(-z_{i}(\phi-\phi_{0}^{*}))\\&=\frac{1}{|\Omega|}\sum_{i=1}^{I}m_{i}z_{i}^2\exp(-z_{i}(\phi-\phi_{0}^{*}))+o_{\epsilon}(1)=-f_{0}'(\phi)+o_{\epsilon}(1),\end{align*}
which implies
\begin{align}
\label{eq.3.053}
f_{\epsilon}(\phi)=f_{0}(\phi)-(\phi_{\epsilon}^{*}-\phi_{0}^{*})f_{0}'(\phi)+(\phi_{\epsilon}^{*}-\phi_{0}^{*})o_{\epsilon}(1)\quad\text{for}~\phi\in\R.
\end{align}

Regarding \eqref{eq.2.54}, we now define
\begin{align}
\label{eq.3.054}
\varphi_{k,\epsilon}=\frac{\phi_{\epsilon}(x)-u_{k}(\delta_k(x)/\sqrt{\epsilon})}{\phi_{\epsilon}^{*}-\phi_{0}^{*}}\quad\text{for}~x\in\overline\Omega_{k,\epsilon,\beta}~\text{and}~k=0,1,\dots,K,
\end{align}
where $u_{k}$ is the unique solution to \eqref{eq.1.19}--\eqref{eq.1.21} and $\delta_k(x)=\dist(x,\partial\Omega_k)$ for $k=0,1,\dots,K$.
Then by \eqref{eq.1.02}, \eqref{eq.1.19}, \eqref{eq.2.58}, and \eqref{eq.3.053}--\eqref{eq.3.054}, $\varphi_{k,\epsilon}$ satisfies
\begin{align}
\label{eq.3.055}
\epsilon\Delta\varphi_{k,\epsilon}+c_{\epsilon}(x)\varphi_{k,\epsilon}=f_{0}'(u_{k}(\delta_{k}(x)/\sqrt{\epsilon}))+o_{\epsilon}(1)\quad\text{for}~x\in\overline\Omega_{k,\epsilon,\beta},
\end{align}
where
\[c_\epsilon(x)=\begin{cases}\dd\frac{f_{\epsilon}(\phi_{\epsilon}(x))-f_{\epsilon}(u_{k}(\delta_k(x)/\sqrt{\epsilon}))}{\phi_{\epsilon}(x)-u_{k}(\delta_k(x)/\sqrt{\epsilon})}&\text{if}~\phi_{\epsilon}(x)\neq u_{k}(\delta_k(x)/\sqrt{\epsilon});\\
f_{\epsilon}'(\phi_{\epsilon}(x))&\text{if}~\phi_{\epsilon}(x)=u_{k}(\delta_k(x)/\sqrt{\epsilon}).\end{cases}\]
For the boundary condition of $\varphi_{k,\epsilon}$, we use \eqref{eq.1.04}, \eqref{eq.1.19}, \eqref{eq.2.03} and \eqref{eq.3.054} to get
\begin{align}
\label{eq.3.056}
\varphi_{k,\epsilon}+\gamma_k\sqrt{\epsilon}\partial_{\nu}\varphi_{k,\epsilon}=0\quad\text{on}~\partial\Omega_k.
\end{align}

Following \Cref{proposition:2.7}, we prove the uniform boundedness of $\varphi_{k,\epsilon}$:
\begin{proposition}
\label{proposition:3.09}
Suppose that \eqref{eq.3.041} holds true.
Then there exists a constant $M>0$ independent of $\epsilon$ such that $\dd\max_{\overline\Omega_{k,\epsilon,\beta}}|\varphi_{k,\epsilon}|\leq M$ for $k=0,1,\dots,K$ and $0<\epsilon<\epsilon^{*}$.
\end{proposition}
\begin{proof}
Fix $k\in\{0,1,\dots,K\}$.
It is equivalent to showing that $\dd\max_{\overline\Omega_{k,\epsilon,\beta}}\varphi_{k,\epsilon}\leq M$ and $\dd\min_{\overline\Omega_{k,\epsilon,\beta}}\varphi_{k,\epsilon}\geq-M$ for some constant $M>0$ independent of $\epsilon$.
We first prove that $\dd\max_{\overline\Omega_{k,\epsilon,\beta}}\varphi_{k,\epsilon}\leq M$ for $0<\epsilon<\epsilon^{*}$.
Suppose by contradiction that there exists a sequence $\{\epsilon_{n}\}_{n=1}^\infty$ of positive numbers with $\dd\lim_{n\to\infty}\epsilon_{n}=0$ and $\{x_n\}_{n=1}^\infty\subset\overline\Omega_{k,\epsilon_{n},\beta}$ such that $\dd\varphi_{k,\epsilon_{n}}(x_n)=\max_{\overline\Omega_{k,\epsilon_{n},\beta}}\varphi_{k,\epsilon_{n}}\geq n$ for $n\in\N$.
Since $0<\beta<1/2$, we may, without loss of generality, assume that $8(\overline{\phi_{bd}}-\underline{\phi_{bd}})\exp(-C_{3}\epsilon_{n}^{(2\beta-1)/2}/8)\leq|\phi_{\epsilon_{n}}^{*}-\phi_{0}^{*}|$ for all $n\in\N$, , respectively.
Note that the maximum point $x_n$ cannot lie on the boundary $\partial\Omega_k$ because from \eqref{eq.3.056}, $\partial_\nu\varphi_{k,\epsilon_{n}}(x_n)=-\varphi_{k,\epsilon_{n}}(x_n)/(\gamma_k\sqrt{\epsilon_{n}})\leq-n/(\gamma_k\sqrt{\epsilon_{n}})<0$ if the maximum point $x_n\in\partial\Omega_k$.
On the other hand, by \eqref{eq.3.023} (with \Cref{proposition:3.05}) and \eqref{eq.3.036}, we have
\begin{align*}
|\varphi_{k,\epsilon_{n}}(x)|&\leq\frac{|\phi_{\epsilon_{n}}(x)-\phi_{\epsilon_{n}}^{*}|+|u_{k}(\delta_k(x)/\sqrt{\epsilon_{n}})-\phi_{0}^{*}|+|\phi_{\epsilon_{n}}^{*}-\phi_{0}^{*}|}{|\phi_{\epsilon_{n}}^{*}-\phi_{0}^{*}|}\\&\leq1+\frac{4(\overline{\phi_{bd}}-\underline{\phi_{bd}})}{|\phi_{\epsilon_{n}}^{*}-\phi_{0}^{*}|}\exp\left(-\frac{C_{3}}8\epsilon_{n}^{(2\beta-1)/2}\right)\leq\frac{3}{2}\quad\text{for}~\delta_k(x)=\epsilon_{n}^{\beta}~\text{and}~n\in\N.
\end{align*}
This shows that $x_n$ cannot lie on the boundary $\partial\Omega_{k,\epsilon_{n},\beta}$.
Hence $x_n\in\Omega_{k,\epsilon_{n},\beta}$ for all $n\in\N$, which implies $\nabla\varphi_{k,\epsilon_{n}}(x_n)=0$ and $\Delta\varphi_{k,\epsilon_{n}}(x_n)\leq0$ for $n\in\N$.
Thus by \eqref{eq.3.055},
\begin{align}
\label{eq.3.057}
0\leq-\epsilon_{n}\Delta\varphi_{k,\epsilon_{n}}(x_n)=c_{\epsilon_{n}}(x_n)\varphi_{k,\epsilon_{n}}(x_n)-f_{0}'(u_{k}(\delta_k(x)/\sqrt{\epsilon_{n}}))+o_{\epsilon_{n}}(1).
\end{align}
Recall that $f_{0}'(u_{k}(t))$ is uniformly bounded for $t\in[0,\infty)$, $\varphi_{k,\epsilon_{n}}(x_n)\geq n$ for $n\in\N$, and $c_\epsilon\leq-C_{3}^2<0$ in $\overline\Omega_{k,\epsilon}$ for $\epsilon>0$.
Letting $n\to\infty$, we obtain $\dd\lim_{n\to\infty}c_{\epsilon_{n}}(x_n)\varphi_{k,\epsilon_{n}}(x_n)=-\infty$, which contradicts \eqref{eq.3.057}.
Thus, $\dd\max_{\overline\Omega_{k,\epsilon,\beta}}\varphi_{k,\epsilon}\leq M$ for some constant $M>0$ independent of $\epsilon$.
Similarly, we may use \eqref{eq.3.023}, \eqref{eq.3.036}, \eqref{eq.3.055}, and \eqref{eq.3.056} to prove $\dd\min_{\overline\Omega_{k,\epsilon,\beta}}\varphi_{k,\epsilon}\geq-M$ for $0<\epsilon<\epsilon^{*}$, where $M>0$ is independent of $\epsilon>0$.
Therefore, we complete the proof of \Cref{proposition:3.09}.
\end{proof}

To derive a contradiction from \eqref{eq.3.041}, we study the asymptotic expansion of $\varphi_{k,\epsilon}$ in $\Omega_{k,\epsilon,\beta}$ and define
\begin{align}
\label{eq.3.058}
\W_{k,p,\epsilon}(z)=\varphi_{k,\epsilon}(\Psi_{p}(\sqrt{\epsilon}z))\quad\text{for}~z\in\overline{B}_{\epsilon^{(2\beta-1)/2}}^{+}~\text{and}~0<\epsilon<\epsilon^{*}
\end{align}
(cf. \eqref{eq.2.62}), where $\Psi_{p}$ is given in \eqref{eq.2.01}.
Similar to \eqref{eq.2.63}--\eqref{eq.2.66}, $\W_{k,p,\epsilon}$ satisfies
\begin{alignat*}{2}
\sum_{i,j=1}^da_{ij}(z)\frac{\partial^2\W_{k,p,\epsilon}}{\partial z^{i}\partial z^j}+\sum_{j=1}^db_j(z)\frac{\partial\W_{k,p,\epsilon}}{\partial z^j}+c_\epsilon(z)\W_{k,p,\epsilon}&=f_{0}(u_{k}(z^{d}))+o_{\epsilon}(1)\quad&&\text{in}~B_{\epsilon^{(2\beta-1)/2}}^+,\\\W_{k,p,\epsilon}-\gamma_k\partial_{z^{d}}\W_{k,p,\epsilon}&=0\quad&&\text{on}~\overline B_{\epsilon^{(2\beta-1)/2}}^+\cap\partial\R_{+}^{d}
\end{alignat*}
for $0<\epsilon<\epsilon^{*}$, where $a_{ij}$ and $b_j$ are given in \eqref{eq.2.05}--\eqref{eq.2.06}, and
\begin{align*}
c_\epsilon(z)=\begin{cases}\dd\frac{f_{\epsilon}(u_{k}(z^{d})+(\phi_{\epsilon}^{*}-\phi_{0}^{*})\W_{k,p,\epsilon}(z))-f_{\epsilon}(u_{k}(z^{d})}{(\phi_{\epsilon}^{*}-\phi_{0}^{*})\W_{k,p,\epsilon}(z)}&\text{if}~\W_{k,p,\epsilon}(z)\neq0;\\f_{\epsilon}'(u_{k}(z^{d}))&\text{if}~\W_{k,p,\epsilon}(z)=0.\end{cases}
\end{align*}
As in \Cref{lemma:2.8} of \Cref{section:2.2}, we use the uniform boundedness of $\varphi_{k,\epsilon}$ (cf. \Cref{proposition:3.09}) to prove
\begin{lemma}
\label{lemma:3.10}
Suppose that \eqref{eq.3.041} holds true.
For any sequence $\{\epsilon_{n}\}_{n=1}^\infty$ of positive numbers with $\dd\lim_{n\to\infty}\epsilon_{n}=0$, $\alpha\in(0,1)$, and $p\in\partial\Omega_k$ ($k\in\{0,1,\dots,K\}$), there exists a subsequence $\{\epsilon_{nn}\}_{n=1}^\infty$ such that $\dd\lim_{n\to\infty}\|\W_{k,p,\epsilon_{nn}}-\W_{k,p}\|_{\C^{2,\alpha}(\overline B_m^+)}=0$ for $m\in\N$, where $\W_{k,p}\in\C_{\text{loc}}^{2,\alpha}(\overline\R_{+}^{d})$ satisfies
\begin{alignat}{2}
\label{eq.3.059}
\Delta\W_{k,p}+f_{0}'(u_{k})\W_{k,p}&=f_{0}'(u_{k})&&\quad\text{in}~\R_{+}^{d},\\
\label{eq.3.060}
\W_{k,p}-\gamma_k\partial_{z^{d}}\W_{k,p}&=0&&\quad\text{on}~\partial\R_{+}^{d}.
\end{alignat}
Moreover, there exists $M>0$ such that $\W_{k,p}$ satisfies the following estimate
\begin{align}
\label{eq.3.061}
|\W_{k,p}(z)-\theta_{k}(z^{d})|\leq M\exp(-C_{3}z^{d}/8)\quad\text{for}~z=(z',z^{d})\in\overline\R_{+}^{d}~\text{and}~z^{d}\geq2(d-1)/C_{3},
\end{align}
where $C_{3}$ is given in \eqref{eq.3.007} and $\theta_{k}=\theta_{k}(t)$ is the solution to ordinary differential equation
\begin{align}
\label{eq.3.062}
&\theta_{k}''+f_{0}'(u_{k})\theta_{k}=f_{0}'(u_{k})\quad\text{in}~(0,\infty),\\
\label{eq.3.063}
&\theta_{k}(0)-\gamma_k\theta_{k}'(0)=0,\\
\label{eq.3.064}
&\lim_{t\to\infty}\theta_{k}(t)=1.
\end{align}
\end{lemma}
\begin{remark}
\label{remark:5}
By the standard linear ODE theory, the unique solution $\theta_{k}$ to \eqref{eq.3.062}--\eqref{eq.3.064} can be expressed explicitly as
\begin{align}
\label{eq.3.065}
\theta_{k}(t)=1-\frac{u_{k}'(t)}{u_{k}'(0)+\gamma_{k}f_{0}(u_{k}(0))}\quad\text{for}~t\geq0.
\end{align}
By (A1)--(A2), \Cref{proposition:B.1}, together with \eqref{eq.3.065}, it follows that
\begin{align}
\label{eq.3.066}
\theta_{k}'(0)=\frac{f_{0}(u_{k}(0))}{u_{k}'(0)+\gamma_{k}f_{0}(u_{k}(0))}>0.
\end{align}
\end{remark}

From \Cref{lemma:3.10}, the solution $\W_{k,p}$ to \eqref{eq.3.059}--\eqref{eq.3.060} may, a priori, depend on the sequence $\{\epsilon_{n}\}_{n=1}^\infty$ and on the point.
To establish the independence from the sequence and the point, we apply the moving plane arguments, as in \Cref{proposition:2.5,proposition:2.9}, to prove that $\W_{k,p}=\W_{k,p}(z^{d})=\theta_{k}(z^{d})$ for $z=(z',z^{d})\in\overline{\R}_{+}^{d}$, where $\theta_{k}$ is the unique solution to \eqref{eq.3.062}--\eqref{eq.3.064}.
Here $\theta_{k}$ is independent of the sequence $\{\epsilon_{n}\}_{n=1}^\infty$ and the point $p\in\partial\Omega_k$.
Consequently, the results of \Cref{lemma:3.10} can be improved as
\begin{align}
\label{eq.3.067}
\lim_{\epsilon\to0^+}\|\W_{k,p}-\theta_{k}\|_{\C^{2,\alpha}(\overline B_m^+)}=0\quad\text{for}~m\in\N~\text{and}~\alpha\in(0,1),
\end{align}
and
\begin{align}
\label{eq.3.068}
\lim_{\epsilon\to0^+}\W_{k,p,\epsilon}(z)=\theta_{k}(z^{d})\quad\text{for}~z=(z',z^{d})\in\overline\R_{+}^{d},
\end{align}
where $\W_{k,p,\epsilon}$ is defined in \eqref{eq.3.058}.
The details are provided below.
\begin{proposition}
\label{proposition:3.11}
Suppose that \eqref{eq.3.041} holds true.
For $p\in\partial\Omega_k$ and $k\in\{0,1,\dots,K\}$, the solution $\W_{k,p}$ to \eqref{eq.3.059}--\eqref{eq.3.060} satisfies
\begin{enumerate}
\item[(a)] $\W_{k,p}$ depends only on the variable $z^{d}$, i.e., $\W_{k,p}(z)=\W_{k,p}(z^{d})$ for $z=(z',z^{d})\in\overline\R_{+}^{d}$.
\item[(b)] $\W_{k,p}$ is independent of $p$ and depends only on $k$, i.e., $\W_{k,p}(z^{d})=\theta_{k}(z^{d})$ for $z^{d}\in[0,\infty)$, where $\theta_{k}$ is the unique solution to \eqref{eq.3.062}--\eqref{eq.3.064}.
\end{enumerate}
\end{proposition}
\begin{proof}
Following \Cref{proposition:2.9}, we replace $f$ with $f_{0}$ and apply the moving plane arguments to obtain (a).
By (a), \eqref{eq.3.061}, and \eqref{eq.3.064}, $\W_{k,p}=\W_{k,p}(z^{d})$ satisfies
\begin{align*}
&\W_{k,p}''+f_{0}'(u_{k})\W_{k,p}=f_{0}'(u_{k})\quad\text{in}~(0,\infty),\\
&\W_{k,p}-\gamma_k\W_{k,p}'(0)=0,\\
&\lim_{t\to\infty}\W_{k,p}(t)=1.
\end{align*}
By the uniqueness of the solution to \eqref{eq.3.062}--\eqref{eq.3.064}, it follows that $\W_{k,p}\equiv\theta_{k}$, which gives (b).
This completes the proof of \Cref{proposition:3.11}.
\end{proof}
We may use \Cref{proposition:3.11} to prove \eqref{eq.3.067}, \eqref{eq.3.068}, and $\W_{k,p}(z)=\theta_{k}(z^{d})$ for $z=(z',z^{d})\in\overline{\R}_{+}^{d}$.
Let $T>0$, $k\in\{0,1,\dots,K\}$, and $p\in\partial\Omega_k$.
Then, under the assumption \eqref{eq.3.041}, we can use \eqref{eq.2.01}, \eqref{eq.3.058}, \eqref{eq.3.067}--\eqref{eq.3.068}, and $\nabla\delta_k(p-t\sqrt{\epsilon}\nu_{p})=-\nu_{p}$ to get
\[\varphi_{k,\epsilon}(p-t\sqrt{\epsilon}\nu_{p})=\theta_{k}(t)+o_{\epsilon}(1),\qquad\nabla\varphi_{k,\epsilon}(p-t\sqrt{\epsilon}\nu_{p})=-\frac{1}{\sqrt{\epsilon}}[\theta_{k}'(t)\nu_{p}+o_{\epsilon}(1)]\]
for $0\leq t\leq T$ as $\epsilon\to0^+$.
Together with \eqref{eq.3.054} and \eqref{eq.3.058}, we arrive at
\begin{align}
\label{eq.3.069}
&\phi_{\epsilon}(p-t\sqrt{\epsilon}\nu_{p})=u_{k}(t)+(\phi_{\epsilon}^{*}-\phi_{0}^{*})\theta_{k}(t)+(\phi_{\epsilon}^{*}-\phi_{0}^{*})o_{\epsilon}(1),\\
\label{eq.3.070}
&\nabla\phi_{\epsilon}(p-t\sqrt{\epsilon}\nu_{p})=-\frac{1}{\sqrt{\epsilon}}\{[u_{k}'(t)+(\phi_{\epsilon}^{*}-\phi_{0}^{*})\theta_{k}'(t)]\nu_{p}+(\phi_{\epsilon}^{*}-\phi_{0}^{*})o_{\epsilon}(1)\}
\end{align}
for $0\leq t\leq T$ as $\epsilon\to0^+$.

By \eqref{eq.3.053} and \eqref{eq.3.069}, we have
\begin{align}
\label{eq.3.071}
\begin{aligned}
f_{\epsilon}(\phi_{\epsilon}(p-t\sqrt{\epsilon}\nu_{p}))&=f_{0}(\phi_{\epsilon}(p-t\sqrt{\epsilon}\nu_{p}))-(\phi_{\epsilon}^{*}-\phi_{0}^{*})f_{0}'(\phi_{\epsilon}(p-t\sqrt{\epsilon}\nu_{p}))+(\phi_{\epsilon}^{*}-\phi_{0}^{*})o_{\epsilon}(1)\\&=
f_{0}(u_{k}(t))+(\phi_{\epsilon}^{*}-\phi_{0}^{*})f_{0}'(u_{k}(t))[\theta_{k}(t)-1]+(\phi_{\epsilon}^{*}-\phi_{0}^{*})o_{\epsilon}(1)
\end{aligned}
\end{align}
for $0\leq t\leq T$ as $\epsilon\to0^+$.
Combining \eqref{eq.3.070}--\eqref{eq.3.071}, we then follow an argument similar to that in \Cref{section:2.3} to obtain
\begin{proposition}
\label{proposition:3.12}
Suppose that \eqref{eq.3.041} holds true.
Then we have
\begin{itemize}
\item[(a)] $\int_{\overline\Omega_{k,T,\epsilon}}\!f_{\epsilon}(\phi_{\epsilon}(x))\,\mathrm{d}x=\sqrt{\epsilon}|\partial\Omega_{k}|(u_{k}'(0)-u_{k}'(T))+\sqrt{\epsilon}(\phi_{\epsilon}^{*}-\phi_{0}^{*})|\partial\Omega_{k}|(\theta_{k}'(0)-\theta_{k}'(T))+\sqrt{\epsilon}(\phi_{\epsilon}^{*}-\phi_{0}^{*})o_{\epsilon}(1)$,
\item[(b)] $\int_{\overline\Omega_{k,T,\epsilon,\beta}}\!f_{\epsilon}(\phi_{\epsilon}(x))\,\mathrm{d}x=\sqrt{\epsilon}|\partial\Omega_{k}|u_{k}'(T)+\sqrt{\epsilon}(\phi_{\epsilon}^{*}-\phi_{0}^{*})|\partial\Omega_{k}|\theta_{k}'(T)+\sqrt{\epsilon}(\phi_{\epsilon}^{*}-\phi_{0}^{*})\epsilon o_{\epsilon}(1)$,
\item[(c)] $\left|\int_{\overline\Omega_{\epsilon,\beta}}\!f_{\epsilon}(\phi_{\epsilon}(x))\,\mathrm{d}x\right|\leq\sqrt{\epsilon}M'\exp\left(-M\epsilon^{(2\beta-1)/2}\right)$
\end{itemize}
for $0<\epsilon<\epsilon^{*}$ and $0<\beta<1/2$, where $|\partial\Omega_k|$ is the surface area of $\partial\Omega_k$.
\end{proposition}

\begin{proof}
Following \Cref{corollary:1}, we begin with (a).
Integrating \eqref{eq.3.071} over $\overline{\Omega}_{k,T,\epsilon}$ and applying the coarea formula (cf. \cite{2015evans,2002lin}) with \eqref{eq.2.75}, we get
\begin{align*}\int_{\overline{\Omega}_{k,T,\epsilon}}\!f_{\epsilon}(\phi_{\epsilon}(x))\,\mathrm{d}x&=\sqrt{\epsilon}\int_0^T\!\int_{\partial\Omega_{k}}\!f_{0}(u_{k}(t))\,\mathrm{d}S_{p}\,\mathrm{d}t\\&\quad+\sqrt{\epsilon}(\phi_{\epsilon}^{*}-\phi_{0}^{*})\int_{0}^{T}\!\int_{\partial\Omega_{k}}\!f_{0}'(u_{k}(t))[\theta_{k}(t)-1]\,\mathrm{d}S_{p}\,\mathrm{d}t+\sqrt{\epsilon}(\phi_{\epsilon}^{*}-\phi_{0}^{*})o_{\epsilon}(1).\end{align*}
Together with \eqref{eq.1.19} and \eqref{eq.3.062}, we obtain
\begin{align*}
\int_{\overline\Omega_{k,T,\epsilon}}\!f_{\epsilon}(\phi_{\epsilon}(x))\,\mathrm{d}x&=\sqrt{\epsilon}|\partial\Omega_{k}|\int_{0}^{T}\!(-u_{k}''(t))\,\mathrm{d}t+\sqrt{\epsilon}(\phi_{\epsilon}^{*}-\phi_{0}^{*})|\partial\Omega_{k}|\int_{0}^{T}\!(-\theta_{k}''(t))\,\mathrm{d}t+\sqrt{\epsilon}(\phi_{\epsilon}^{*}-\phi_{0}^{*})o_{\epsilon}(1)\\
&=\sqrt{\epsilon}|\partial\Omega_{k}|(u_{k}'(0)-u_{k}'(T))+\sqrt{\epsilon}(\phi_{\epsilon}^{*}-\phi_{0}^{*})|\partial\Omega_{k}|(\theta_{k}'(0)-\theta_{k}'(T))+\sqrt{\epsilon}(\phi_{\epsilon}^{*}-\phi_{0}^{*})o_{\epsilon}(1),
\end{align*}
which gives (a).

We now prove (b).
As in \eqref{eq.2.76}, we integrate \eqref{eq.1.02} over $\overline\Omega_{k,T,\epsilon,\beta}$ and then apply the divergence theorem to obtain
\begin{align}
\label{eq.3.072}
\begin{aligned}
\int_{\overline\Omega_{k,T,\epsilon,\beta}}\!f_{\epsilon}(\phi_{\epsilon}(x))\,\mathrm{d}x&=-\epsilon\int_{\overline\Omega_{k,T,\epsilon,\beta}}\!\Delta\phi_{\epsilon}(x)\,\mathrm{d}x\\&=-\epsilon\int_{\partial\Omega_{k,T,\epsilon,\beta}}\partial_{\nu_{x}}\phi_{\epsilon}(x)\,\mathrm{d}S_x\\&=-\int_{\partial\Omega_{k}}\!(\epsilon\partial_{\nu_p}\phi_{\epsilon}(p-T\sqrt{\epsilon}\nu_{p}))\mathcal{J}(T,p)\,\mathrm{d}S_p\\&\quad+\int_{\partial\Omega_{k}}\!(\epsilon\partial_{\nu_{p}}\phi_{\epsilon}(p-\epsilon^{\beta}\nu_{p}))\mathcal{J}(\epsilon^{(2\beta-1)/2},p)\,\mathrm{d}S_{p},
\end{aligned}
\end{align}
where $\nu_{x}$ is the unit outer normal at $x\in\partial\Omega_{k,T,\epsilon,\beta}$ with respect to $\Omega_{k,T,\epsilon,\beta}$, and $\mathcal{J}(T,p)$ and $\mathcal{J}(\epsilon^{(2\beta-1)/2},p)$ are given in \eqref{eq.2.77}.
Then by \eqref{eq.3.041} and \eqref{eq.3.070}, we have
\begin{align}
\label{eq.3.073}
\begin{aligned}
&\quad-\int_{\partial\Omega_{k}}\!(\epsilon\partial_{\nu_p}\phi_{\epsilon}(p-T\sqrt{\epsilon}\nu_{p}))\mathcal{J}(T,p)\,\mathrm{d}S_p\\&=-\int_{\partial\Omega_{k}}\!(\epsilon\partial_{\nu_p}\phi_{\epsilon}(p-T\sqrt{\epsilon}\nu_{p}))[1-T\sqrt{\epsilon}(d-1)H(p)+\sqrt{\epsilon}o_{\epsilon}(1)]\,\mathrm{d}S_{p}\\
&=\int_{\partial\Omega_{k}}\![\sqrt{\epsilon}u_{k}'(T)+\sqrt{\epsilon}(\phi_{\epsilon}^{*}-\phi_{0}^{*})\theta_{k}'(T)][1-T\sqrt{\epsilon}(d-1)H(p)+\sqrt{\epsilon}o_{\epsilon}(1)]\,\mathrm{d}S_{p}\\
&=\sqrt{\epsilon}|\partial\Omega_{k}|u_{k}'(T)+\sqrt{\epsilon}(\phi_{\epsilon}^{*}-\phi_{0}^{*})|\partial\Omega_{k}|\theta_{k}'(T)+\sqrt{\epsilon}(\phi_{\epsilon}^{*}-\phi_{0}^{*})o_{\epsilon}(1).
\end{aligned}
\end{align}
On the other hand, by \eqref{eq.3.037} in \Cref{proposition:3.07}, we have
\begin{align}
\label{eq.3.074}
\left|\int_{\partial\Omega_{k}}\!(\epsilon\partial_{\nu_{p}}\phi_{\epsilon}(p-\epsilon^{\beta}\nu_{p}))\mathcal{J}(\epsilon^{(2\beta-1)/2},p)\,\mathrm{d}S_{p}\right|\leq\sqrt{\epsilon}|\partial\Omega_{k}|M'\exp\left(-M\epsilon^{(2\beta-1)/2}\right).\end{align}
Combining \eqref{eq.3.072}--\eqref{eq.3.074} with $0<\beta<1/2$, we get (b).
Following the argument for (b), one may use \eqref{eq.3.074} to prove (c).
Therefore, we complete the proof of \Cref{proposition:3.12}.
\end{proof}

We are now in a position to derive a contradiction from \eqref{eq.3.041}.
By \eqref{eq.1.02}--\eqref{eq.1.03}, we observe that
\[\int_{\Omega}\!f_{\epsilon}(\phi_{\epsilon}(x))\,\mathrm{d}x=\sum_{i=1}^{I}\frac{m_{i}z_{i}\exp(-z_{i}\phi_{\epsilon}(x))}{\int_{\Omega}\!\exp(-z_{i}\phi_{\epsilon}(y))\,\mathrm{d}y}\,\mathrm{d}x=\sum_{i=1}^{I}m_{i}z_{i}=0,\]
which implies
\begin{align}
\label{eq.3.075}
\sum_{k=0}^{K}\!\int_{\overline{\Omega}_{k,T,\epsilon}\cup\overline{\Omega}_{k,T,\epsilon,\beta}}\!f_{\epsilon}(\phi_{\epsilon}(x))\,\mathrm{d}x=\int_{\overline{\Omega}_{\epsilon,\beta}}\!f_{\epsilon}(\phi_{\epsilon}(x))\,\mathrm{d}x.\end{align}
Along with \eqref{eq.3.027}, \eqref{eq.3.041}, and \Cref{proposition:3.12}, we have
\[\sqrt{\epsilon}(\phi_{\epsilon}^{*}-\phi_{0}^{*})\sum_{k=0}^{K}|\partial\Omega_{k}|\theta_{k}'(0)=\sqrt{\epsilon}(\phi_{\epsilon}^{*}-\phi_{0}^{*})o_{\epsilon}(1),\]
which gives
\[\sum_{k=0}^{K}|\partial\Omega_{k}|\theta_{k}'(0)=0.\]
However, since $\theta_{k}'(0)>0$ for all $k=0,1,\dots,K$ (see \eqref{eq.3.066} in \Cref{remark:5}), we get a contradiction, which shows \eqref{eq.3.041} cannot hold true.
Consequently, we establish the uniform boundedness of $|\phi_{\epsilon}^{*}-\phi_{0}^{*}|/\sqrt{\epsilon}$.

\begin{remark}
\label{remark:6}
Due to the uniform boundedness of $|\phi_{\epsilon}^{*}-\phi_{0}^{*}|/\sqrt{\epsilon}$, there exists a sequence $\{\epsilon_{n}\}_{n=1}^{\infty}$ with $\dd\lim_{n\to\infty}\epsilon_{n}=0$ and a constant $Q\in\R$ such that $(\phi_{\epsilon_{n}}^{*}-\phi_{0}^{*})/\sqrt{\epsilon_{n}}\to Q$ as $n\to\infty$.
In \Cref{section:3.5}, we will prove that $Q$ is uniquely determined by \eqref{eq.3.114}, which implies
\begin{align}
\label{eq.3.076}
\phi_{\epsilon}^{*}=\phi_{0}^{*}+\sqrt{\epsilon}(Q+o_{\epsilon}(1))\quad\text{as}~\epsilon\to0^+.
\end{align}
\end{remark}

\subsection{Second-order asymptotic expansion of \texorpdfstring{$\phi_{\epsilon}$}{ϕ.ϵ} in \texorpdfstring{$\Omega_{k,\epsilon}$}{Ω.k,ϵ}}
\label{section:3.4}

In this section, we will use uniform boundedness of $|\phi_{\epsilon}^{*}-\phi_{0}^{*}|/\sqrt{\epsilon}$ (cf. \eqref{eq.3.076}) to derive the second-order asymptotic expansions of $\phi_{\epsilon}$ and $\nabla\phi_{\epsilon}$.
As in \Cref{section:3.3}, we first study the second-order asymptotic expansions of $f_{\epsilon}$ and $A_{i,\epsilon}$.
A direct computation of $f_{\epsilon}$ shows
\begin{align}
\label{eq.3.077}
f_{1,\epsilon}(\phi):=\frac{f_{\epsilon}(\phi)-f_{0}(\phi)}{\sqrt{\epsilon}}=-\sum_{i=1}^{I}\frac{m_{i}z_{i}B_{i,\epsilon}}{|\Omega|[|\Omega|\exp(-z_{i}\phi_{0}^{*})+\sqrt{\epsilon}B_{i,\epsilon}]}\exp(-z_{i}(\phi-\phi_{0}^{*}))\quad\text{for}~\phi\in\R,
\end{align}
where
\begin{align}
\label{eq.3.078}
B_{i,\epsilon}=\frac{A_{i,\epsilon}-|\Omega|\exp(-z_{i}\phi_{0}^{*})}{\sqrt{\epsilon}}=\int_\Omega\!\frac{\exp(-z_{i}\phi_{\epsilon}(y))-\exp(-z_{i}\phi_{0}^{*})}{\sqrt{\epsilon}}\,\mathrm{d}y
\end{align}
for $i=1,\dots,I$ and $\epsilon>0$.
Note that $B_{i,\epsilon}$ can be expressed by
\begin{align}
\label{eq.3.079}
B_{i,\epsilon}=J_{i,1}+\sum_{k=0}^KJ_{i,k,2}\quad\text{for}~i=1,\dots,I,
\end{align}
where
\begin{align}
\label{eq.3.080}
&J_{i,1}=\int_{\overline{\Omega}_{\epsilon,\beta}}\!\frac{\exp(-z_{i}\phi_{\epsilon}(y))-\exp(-z_{i}\phi_{0}^{*})}{\sqrt{\epsilon}}\,\mathrm{d}y,\\
\label{eq.3.081}
&J_{i,k,2}=\int_{\Omega_{k,\epsilon,\beta}}\!\frac{\exp(-z_{i}\phi_{\epsilon}(y))-\exp(-z_{i}\phi_{0}^{*})}{\sqrt{\epsilon}}\,\mathrm{d}y\quad\text{for}~k=0,1,\dots,K.
\end{align}
Here $\overline\Omega_{\epsilon,\beta}=\{x\in\Omega\,:\,\delta(x)\geq\epsilon^\beta\}$ and $\Omega_{k,\epsilon,\beta}=\{x\in\Omega\,:\delta_k(x)<\epsilon^\beta\}$ for $0<\beta<1/2$ and $k=0,1,\dots,K$. Clearly, $\Omega_{k,\epsilon,\beta}$ are disjoint for sufficiently small $\epsilon>0$.
Using \eqref{eq.3.036}, \eqref{eq.3.078}, and \eqref{eq.3.076}, we obtain the following result, similar to Lemma 3.2 in \Cref{section:3.3}:
\begin{lemma}
\label{lemma:3.13}
$B_{i,\epsilon}=\exp(-z_{i}\phi_{0}^{*})[-z_{i}Q|\Omega|-(\hat m_{i}/m_{i})|\Omega|]+o_{\epsilon}(1)$ for $i=1,\dots,I$, where $u_{k}$ is the unique solution to \eqref{eq.1.19}--\eqref{eq.1.21}, and $\hat m_{i}=m_{i}|\Omega|^{-1}\sum_{k=0}^K|\partial\Omega_k|\int_0^\infty\![\exp(-z_{i}(u_{k}(s)-\phi_{0}^{*}))-1]\,\mathrm{d}s$ for $i=1,\dots,I$.
\end{lemma}

\begin{remark}
\label{remark:7}
By \eqref{eq.3.078} and \Cref{lemma:3.13}, we obtain
\[A_{i,\epsilon}=|\Omega|\exp(-z_{i}\phi_{0}^{*})+\sqrt{\epsilon}\exp(-z_{i}\phi_{0}^{*})\left(-z_{i}Q|\Omega|-\frac{\hat m_{i}}{m_{i}}|\Omega|+o_{\epsilon}(1)\right)\quad\text{for}~i=1,\dots,I.\]
This allows us to compute the expansion of $c_{i,\epsilon}^{\mathrm b}=m_{i}/A_{i,\epsilon}$:
\begin{align*}
c_{i,\epsilon}^{\mathrm b}&=\frac{m_{i}}{A_{i,\epsilon}}=\frac{m_{i}}{|\Omega|\exp(-z_{i}\phi_{0}^{*})}\left(1+\sqrt{\epsilon}\left[-z_{i}Q-\frac{\hat m_{i}}{m_{i}}+o_{\epsilon}(1)\right]\right)^{-1}
\\&=\frac{m_{i}}{|\Omega|\exp(-z_{i}\phi_{0}^{*})}\left(1+\sqrt{\epsilon}\left[z_{i}Q+\frac{\hat m_{i}}{m_{i}}+o_{\epsilon}(1)\right]\right)
\\&=c_i^{\mathrm b}+\sqrt{\epsilon}\left(\frac{m_{i}z_{i}Q\exp(z_{i}\phi_{0}^{*})+\hat m_{i}\exp(z_{i}\phi_{0}^{*})}{|\Omega|}+o_{\epsilon}(1)\right)\quad\text{for}~i=1,\dots,I,\end{align*}
where $c_i^{\mathrm b}=m_{i}|\Omega|^{-1}\exp(z_{i}\phi_{0}^{*})$ for $i=1,\dots,I$.
\end{remark}
\begin{proof}[Proof of \Cref{lemma:3.13}]
To prove \Cref{lemma:3.13}, we analyze $J_{i,1}$ and $J_{i,k,2}$ as in \Cref{section:3.3}.
First, we compute the limit of $J_{i,1}$ using \eqref{eq.3.036} and \eqref{eq.3.076}.
\begin{claim}
\label{claim:3}
$J_{i,1}=-z_{i}Q\exp(-z_{i}\phi_{0}^{*})|\Omega|+o_{\epsilon}(1)$ for $i=1,\dots,I$.
\end{claim}
\noindent Henceforth $\epsilon^{*}>0$ is a sufficiently small constant, and $M,M'>0$ are generic constants independent of $\epsilon>0$.
\begin{proof}[Proof of \Cref{claim:3}]
Fix $i\in\{1,\dots,I\}$.
We first consider the case of $z_{i}>0$.
Then by \eqref{eq.3.036}, we may follow the approach of \eqref{eq.3.047}--\eqref{eq.3.048} in \Cref{lemma:3.08} (with $\sqrt{\epsilon}$ scaling instead of $\phi_{\epsilon}^*-\phi_{0}^{*}$) and apply \eqref{eq.3.076} and \eqref{eq.3.080} to obtain
\begin{align*}
&\quad\exp(-z_{i}\phi_{0}^{*})\left|\Omega-\bigcup_{k=0}^K\Omega_{k,\epsilon,\beta}\right|\frac{\exp(-z_{i}Q\sqrt{\epsilon})\exp[-z_{i}M'\exp(-M\epsilon^{(2\beta-1)/2})+\sqrt{\epsilon}o_{\epsilon}(1)]-1}{\sqrt{\epsilon}}\\
&\leq
J_{i,1}\leq\exp(-z_{i}\phi_{0}^{*})\left|\Omega-\bigcup_{k=0}^K\Omega_{k,\epsilon,\beta}\right|\frac{\exp(-z_{i}Q\sqrt{\epsilon})\exp[z_{i}M'\exp(-M\epsilon^{(2\beta-1)/2})+\sqrt{\epsilon}o_{\epsilon}(1)]-1}{\sqrt{\epsilon}}.
\end{align*}
Along with the fact that
\begin{align*}&\quad\lim_{\epsilon\to0^+}\frac{\exp(-z_{i}Q\sqrt\epsilon)\exp[\pm z_{i}M'\exp(-M\epsilon^{(2\beta-1)/2})+\sqrt{\epsilon}o_{\epsilon}(1)]-1}{\sqrt\epsilon}\\&=\lim_{\epsilon\to0^+}\frac{\exp[-z_{i}Q\sqrt\epsilon\pm z_{i}M'\exp(-M\epsilon^{(2\beta-1)/2})+\sqrt{\epsilon}o_{\epsilon}(1)]-1}{\sqrt{\epsilon}}=-z_{i}Q,\end{align*}
we obtain
\begin{align*}
J_{i,1}=-z_{i}Q\exp(-z_{i}\phi_{0}^{*})|\Omega|+o_{\epsilon}(1)\quad\text{for}~i=1,\dots,I.
\end{align*}
For the case of $z_{i}<0$, we can follow a similar argument and complete the proof of \Cref{claim:3}.
\end{proof}

To obtain the boundedness of $J_{i,k,2}$ (defined in \eqref{eq.3.081}), we employ \eqref{eq.3.036}, \eqref{eq.3.076}, the principal coordinate system (cf. \cite{1977gilbarg}), and the coarea formula (cf. \cite{2015evans,2002lin}), adapting techniques analogous to those in \Cref{section:3.3}. 
Below are the details.
\begin{claim}
\label{claim:4}
There exists $M>0$ independent of $\epsilon$ such that $|J_{i,k,2}|\leq M$ for $i=1,\dots,I$ and $k=0,1,\dots,K$ and $0<\epsilon<\epsilon^{*}$, where $\epsilon^{*}>0$ is sufficiently small.
\end{claim}
\begin{proof}[Proof of \Cref{claim:4}]
Similar to the analysis in \Cref{section:3.3}, we consider $y\in\overline\Omega_{k,\epsilon,\beta}=\{x\in\overline\Omega:\delta_k(x)\leq\epsilon^{\beta}\}$. For sufficiently small $\epsilon>0$, the point $y$ can be expressed as $y=p-s\sqrt{\epsilon}\nu_{p}$, where $p\in\partial\Omega_k$, $\nu_{p}$ is the unit outer normal, and $0\leq s\leq\epsilon^{(2\beta-1)/2}$.
Then we apply the corea formula and \eqref{eq.2.75} to get
\begin{align}
\label{eq.3.082}
\begin{aligned}
J_{i,k,2}&=\int_{\partial\Omega_k}\!\int_0^{\epsilon^{(2\beta-1)/2}}\![\exp(-z_{i}\phi_{\epsilon}(p-s\sqrt{\epsilon}\nu_{p}))-\exp(-z_{i}\phi_{0}^{*})]\{1-s\sqrt\epsilon[(d-1)H(p)+o_{\epsilon}(1)]\}\,\mathrm{d}s\,\mathrm{d}S_p\\&:=J_{i,k,3}-J_{i,k,4}\quad\text{for}~i=1,\dots,I~\text{and}~k=0,1,\dots,K,
\end{aligned}
\end{align}
where
\begin{align}
\label{eq.3.083}
&J_{i,k,3}=\int_{\partial\Omega_k}\!\int_0^{\epsilon^{(2\beta-1)/2}}\!(\exp(-z_{i}\phi_{\epsilon}(p-s\sqrt{\epsilon}\nu_{p}))-\exp(-z_{i}\phi_{0}^{*}))\,\mathrm{d}s\,\mathrm{d}S_p,\\
\label{eq.3.084}
&J_{i,k,4}=\sqrt\epsilon\int_{\partial\Omega_k}\!\int_0^{\epsilon^{(2\beta-1)/2}}\!s[\exp(-z_{i}\phi_{\epsilon}(p-s\sqrt{\epsilon}\nu_{p}))-\exp(-z_{i}\phi_{0}^{*})][(d-1)H(p)+o_{\epsilon}(1)]\,\mathrm{d}s\,\mathrm{d}S_p.
\end{align}

By \Cref{proposition:3.01,proposition:3.03}, $|\phi_{\epsilon}(x)-\phi_{0}^{*}|$ is uniformly bounded for $x\in\overline\Omega$ and $\epsilon>0$, which implies that there exists $M>1$ such that $|\exp(-z_{i}(\phi_{\epsilon}(x)-\phi_{0}^{*}))-1|\leq M|\phi_{\epsilon}(x)-\phi_{0}^{*}|$ for $x\in\overline\Omega$ and $0<\epsilon<\epsilon^{*}$.
Then by \eqref{eq.3.036}, we have
\begin{align}
\label{eq.3.085}
\left|\exp(-z_{i}\phi_{\epsilon}(x))-\exp(-z_{i}\phi_{0}^{*})\right|\leq M'\exp\left(-M\delta(x)/\sqrt\epsilon\right)+M'|\phi_{\epsilon}^{*}-\phi_{0}^{*}|
\end{align}
for $x\in\overline\Omega$ and $0<\epsilon<\epsilon^{*}$.
Using \eqref{eq.3.076} and \eqref{eq.3.085} with $x=p-s\sqrt{\epsilon}\nu_{p}$, $J_{i,k,3}$ can be estimated by
\begin{align}
\label{eq.3.086}
\begin{aligned}
|J_{i,k,3}|&=\left|\int_{\partial\Omega_k}\!\int_0^{\epsilon^{(2\beta-1)/2}}\!\left[\exp(-z_{i}\phi_{\epsilon}(p-s\sqrt{\epsilon}\nu_{p}))-\exp(-z_{i}\phi_{0}^{*})\right]\,\mathrm{d}s\,\mathrm{d}S_p\right|\\&\leq M'\int_0^{\epsilon^{(2\beta-1)/2}}\![\exp(-Ms)+|\phi_{\epsilon}^{*}-\phi_{0}^{*}|]\,\mathrm{d}s\\&
=M'(1-\exp(-M\epsilon^{(2\beta-1)/2}))+M'\epsilon^{(2\beta-1)/2}|\phi_{\epsilon}^{*}-\phi_{0}^{*}|\leq M\quad\text{for}~k=0,1,\dots,K.
\end{aligned}
\end{align}
On the other hand, due to the smoothness of $\partial\Omega_k$, there exists a positive constant $M$ independent of $\epsilon$ such that $|(d-1)H(p)+o_{\epsilon}(1)|\leq M$ for $0<\epsilon<\epsilon^{*}$.
Using \eqref{eq.3.076} and \eqref{eq.3.085} with $x=p-s\sqrt{\epsilon}\nu_{p}$, $J_{i,k,4}$ can be estimated by
\begin{align}
\label{eq.3.087}
\begin{aligned}
|J_{i,k,4}|&=\sqrt\epsilon\left|\int_{\partial\Omega_k}\!\int_0^{\epsilon^{(2\beta-1)/2}}\!s[\exp(-z_{i}\phi_{\epsilon}(p-s\sqrt{\epsilon}\nu_{p}))-\exp(-z_{i}\phi_{0}^{*})][(d-1)H(p)+o_{\epsilon}(1)]\,\mathrm{d}s\,\mathrm{d}S_p\right|\\
&\leq M'\sqrt\epsilon\int_0^{\epsilon^{(2\beta-1)/2}}\!s\left[\exp(-Ms)+|\phi_{\epsilon}^{*}-\phi_{0}^{*}|\right]\,\mathrm{d}s=\sqrt\epsilon\mathcal{O}_\epsilon(1)\quad\text{for}~k=0,1,\dots,K.
\end{aligned}
\end{align}
Combining \eqref{eq.3.082}--\eqref{eq.3.087}, we arrive at $|J_{i,k,2}|\leq|J_{i,k,3}|+|J_{i,k,4}|\leq M$ for $0<\epsilon<\epsilon^{*}$, which completes the proof of \Cref{claim:4}.
\end{proof}

To complete the proof of \Cref{lemma:3.13}, we define
\begin{align}
\label{eq.3.088}
\varphi_{k,\epsilon}(x)=\frac{\phi_{\epsilon}(x)-u_{k}(\delta_k(x)/\sqrt\epsilon)}{\sqrt\epsilon}\quad\text{for}~x\in\overline\Omega_{k,\epsilon,\beta}~\text{and}~k=0,1,\dots,K,
\end{align}
where $u_{k}$ is the unique solution to \eqref{eq.1.19}--\eqref{eq.1.21}, $\delta_k(x)=\dist(x,\partial\Omega_k)$ and $\overline\Omega_{k,\epsilon,\beta}=\{x\in\overline\Omega:\delta_k(x)\leq\epsilon^{\beta}\}$ for $k=0,1,\dots,K$.
Note that \eqref{eq.3.088} has the same form as \eqref{eq.2.54}.
Following the approach of \Cref{proposition:2.7} in \Cref{section:3.3}, we claim
\begin{claim}
\label{claim:5}
There exists a constant $M>0$ independent of $\epsilon$ such that $\dd\max_{\overline{\Omega}_{k,\epsilon,\beta}}|\varphi_{k,\epsilon}|\leq M$ for $k=0,1,\dots,K$ and $0<\epsilon<\epsilon^{*}$.
\end{claim}
\begin{proof}[Proof of \Cref{claim:5}]
Following the method of \Cref{section:2.2}, we derive the partial differential equation for $\varphi_{k,\epsilon}$ in $\Omega_{k,\epsilon,\beta}$.
As in \eqref{eq.2.55}, we use \eqref{eq.1.02}, \eqref{eq.1.19}, \eqref{eq.2.58} and \eqref{eq.3.088} to get
\begin{align*}
\epsilon\Delta\varphi_{k,\epsilon}(x)&=-\frac{f_{\epsilon}(\phi_{\epsilon}(x))-f_{0}(u_{k}(\delta_k(x)/\sqrt\epsilon))}{\sqrt\epsilon}+[(d-1)H(p_x)+\epsilon^{\beta}\mathcal{O}_\epsilon(1)]u_{k}'(\delta_k(x)/\sqrt\epsilon)\\
&=-\frac{f_{\epsilon}(\phi_{\epsilon}(x))-f_{\epsilon}(u_{k}(\delta_k(x)/\sqrt\epsilon))}{\sqrt\epsilon}-\frac{f_{\epsilon}(u_{k}(\delta_k(x)/\sqrt\epsilon))-f_{0}(u_{k}(\delta_k(x)/\sqrt\epsilon))}{\sqrt\epsilon}\\
&\quad+[(d-1)H(p_x)+\epsilon^{\beta}\mathcal{O}_\epsilon(1)]u_{k}'(\delta_k(x)/\sqrt\epsilon)\\
&=-c_\epsilon(x)\varphi_{k,\epsilon}+g_\epsilon(x)\quad\text{for}~x\in\Omega_{k,\epsilon,\beta},
\end{align*}
where
\begin{align}
\label{eq.3.089}
&c_\epsilon(x)=\begin{cases}\dd\frac{f_{\epsilon}(\phi_{\epsilon}(x))-f_{\epsilon}(u_{k}(\delta_k(x)/\sqrt\epsilon))}{\phi_{\epsilon}(x)-u_{k}(\delta_k(x)/\sqrt\epsilon)}&\text{if}~\phi_{\epsilon}(x)\neq u_{k}(\delta_k(x)/\sqrt\epsilon);\\
f_{\epsilon}'(\phi_{\epsilon}(x))&\text{if}~\phi_{\epsilon}(x)=u_{k}(\delta_k(x)/\sqrt\epsilon),\end{cases}\\
\label{eq.3.090}
&g_\epsilon(x)=(d-1)H(p_x)u_{k}'(\delta_k(x)/\sqrt\epsilon)-f_{1,\epsilon}(u_{k}(\delta_k(x)/\sqrt\epsilon))+\epsilon^{\beta}\mathcal{O}_\epsilon(1).
\end{align}
Here we have used the definition of $f_{1,\epsilon}$ (cf. \eqref{eq.3.077}) and the fact that $u_{k}'$ is bounded on $[0,\infty)$ (cf. \Cref{proposition:B.1}).
Note that functions $c_\epsilon$ and $g_\epsilon$ are different from \eqref{eq.2.56}--\eqref{eq.2.57} because $f_{1,\epsilon}$ is a nonzero function.
Along with \eqref{eq.1.04}, \eqref{eq.1.20}, \eqref{eq.2.03} and \eqref{eq.3.020}, we obtain
\begin{alignat}{2}
\label{eq.3.091}
\epsilon\Delta\varphi_{k,\epsilon}+c_\epsilon(x)\varphi_{k,\epsilon}&=g_\epsilon(x)\quad&&\text{for}~x\in\Omega_{k,\epsilon},\\
\label{eq.3.092}
\varphi_{k,\epsilon}+\gamma_k\sqrt\epsilon\partial_\nu\varphi_{k,\epsilon}&=0\quad&&\text{on}~\partial\Omega_k\quad\text{for}~k=0,1,\dots,K.
\end{alignat}
Note that $c_\epsilon(x)<0$ for $x\in\Omega_{k,\epsilon}$ because $f_{\epsilon}'(\phi)\leq-C_{3}^2$ for $\phi\in\R$ and $\epsilon>0$ (cf \eqref{eq.3.007}).
By \eqref{eq.3.077}, \eqref{eq.3.079}, \Cref{claim:3,claim:4}, $f_{1,\epsilon}$ is uniformly bounded on the interval $[\underline{\phi_{bd}},\overline{\phi_{bd}}]$ with respect to $\epsilon$ and then $g_\epsilon$ is uniformly bounded with respect to $\epsilon$.
By \eqref{eq.3.036}, \eqref{eq.3.076} and \eqref{eq.3.023} with \Cref{proposition:3.05}, we find the following inequality:
\begin{align*}
|\varphi_{k,\epsilon}(x)|\leq\frac{|\phi_{\epsilon}(x)-\phi_{0}^{*}|+|u_{k}(\delta_k(x)/\sqrt\epsilon)-\phi_{0}^{*}|}{\sqrt\epsilon}\leq\frac{M'}{\sqrt\epsilon}\exp\left(-M\epsilon^{(2\beta-1)/2}\right)+\frac{|\phi_{\epsilon}^{*}-\phi_{0}^{*}|}{\sqrt\epsilon}\leq 2|Q|+\frac12
\end{align*}
for $\delta_k(x)=\epsilon^{1/8}$ as $\epsilon$ sufficiently small, which has the same form as \eqref{eq.2.60} in \Cref{proposition:2.7}.
Hence as in \Cref{proposition:2.7}, we conclude that $\varphi_{k,\epsilon}$ is uniformly bounded in $\overline\Omega_{k,\epsilon}$, i.e., there exists a positive constant $M$ independent of $\epsilon$ such that $\dd\max_{\overline\Omega_{k,\epsilon}}|\varphi_{k,\epsilon}|\leq M$ for $k=0,1,\dots,K$ and sufficiently small $\epsilon>0$.
Therefore, this ends the proof of \Cref{claim:5}.
\end{proof}

By \eqref{eq.3.088} and \Cref{claim:5}, we have
\begin{align}
\label{eq.3.093}
\exp(-z_{i}\phi_{\epsilon}(p-s\sqrt{\epsilon}\nu_{p}))=\exp(-z_{i}u_{k}(s))+\sqrt\epsilon\mathcal{O}_\epsilon(1)\quad\text{for}~0\leq s\leq\epsilon^{(2\beta-1)/2}.
\end{align}
Thus, by \eqref{eq.3.083} and \eqref{eq.3.093}, we get
\begin{align*}
J_{i,k,3}&=\int_{\partial\Omega_k}\!\int_0^{\epsilon^{(2\beta-1)/2}}\![\exp(-z_{i}\phi_{\epsilon}(p-s\sqrt{\epsilon}\nu_{p}))-\exp(-z_{i}\phi_{0}^{*})]\,\mathrm{d}s\,\mathrm{d}S_p
\\&=\int_{\partial\Omega_k}\!\int_0^{\epsilon^{(2\beta-1)/2}}\![\exp(-z_{i}u_{k}(s))-\exp(-z_{i}\phi_{0}^{*})+\sqrt\epsilon\mathcal{O}_\epsilon(1)]\,\mathrm{d}s\,\mathrm{d}S_p
\\&=\exp(-z_{i}\phi_{0}^{*})\int_{\partial\Omega_k}\!\int_0^{\epsilon^{(2\beta-1)/2}}\![\exp(-z_{i}(u_{k}(s)-\phi_{0}^{*}))-1]\,\mathrm{d}s\,\mathrm{d}S_p+\epsilon^{\beta}\mathcal{O}_\epsilon(1)
\\&=\exp(-z_{i}\phi_{0}^{*})|\partial\Omega_k|\int_0^\infty\![\exp(-z_{i}(u_{k}(s)-\phi_{0}^{*}))-1]\,\mathrm{d}s+o_{\epsilon}(1)\quad\text{for}~i=1,\dots,I,
\end{align*}
which implies
\begin{align}
\label{eq.3.094}
J_{i,k,3}=\exp(-z_{i}\phi_{0}^{*})|\partial\Omega_k|\int_0^\infty\![\exp(-z_{i}(u_{k}(s)-\phi_{0}^{*}))-1]\,\mathrm{d}s+o_{\epsilon}(1)\quad\text{for}~i=1,\dots,I.
\end{align}
Here we have used the fact that $\int_0^\infty\![\exp(-z_{i}(u_{k}(s)-\phi_{0}^{*}))-1]\,\mathrm{d}s$ is convergent because of \eqref{eq.3.023} and \Cref{proposition:3.05}.
Therefore, we can use \eqref{eq.3.079}, \eqref{eq.3.082}--\eqref{eq.3.084}, \eqref{eq.3.087}, \eqref{eq.3.094} and \Cref{claim:3} to get
\begin{align*}
B_{i,\epsilon}&=J_{i,1}+\sum_{k=0}^KJ_{i,k,2}\\
&=(-z_{i}Q\exp(-z_{i}\phi_{0}^{*})+o_{\epsilon}(1))+\sum_{k=0}^K(J_{i,k,3}+J_{i,k,4})\\
&=\exp(-z_{i}\phi_{0}^{*})\left[-z_{i}Q|\Omega|+\sum_{k=0}^K|\partial\Omega_k|\int_0^\infty\![\exp(-z_{i}(u_{k}(s)-\phi_{0}^{*}))-1]\,\mathrm{d}s\right]+o_{\epsilon}(1)\quad\text{for}~i=1,\dots,I,\end{align*}
and complete the proof of \Cref{lemma:3.13}.
\end{proof}

Applying \Cref{lemma:3.13} to \eqref{eq.3.077}, we obtain
\begin{align*}f_{1,\epsilon}(\phi)&=-\sum_{i=1}^{I}\frac{m_{i}z_{i}B_{i,\epsilon}}{|\Omega|[|\Omega|\exp(-z_{i}\phi_{0}^{*})+\sqrt\epsilon B_{i,\epsilon}]}\exp(-z_{i}(\phi-\phi_{0}^{*}))\\&=-\sum_{i=1}^{I}\frac{m_{i}z_{i}}{|\Omega|^{2}}\left[-z_{i}Q|\Omega|+\sum_{k=0}^K|\partial\Omega_k|\int_0^\infty\![\exp(-z_{i}(u_{k}(s)-\phi_{0}^{*}))-1]\,\mathrm{d}s\right]\exp(-z_{i}(\phi-\phi_{0}^{*}))+o_{\epsilon}(1)\\&:=f_1(\phi)+o_{\epsilon}(1),\end{align*}
which implies
\begin{align}
\label{eq.3.095}
f_{1,\epsilon}(\phi)=f_1(\phi)+o_{\epsilon}(1)\quad\text{for}~\phi\in\R.
\end{align}
Note that $f_1$ is given in \eqref{eq.3.004}.
Hence the nonlinear term $f_{\epsilon}$ of the CCPB equation \eqref{eq.3.001} can be expressed as
\begin{align}
\label{eq.3.096}
f_{\epsilon}(\phi)=f_{0}(\phi)+\sqrt\epsilon(f_1(\phi)+o_{\epsilon}(1))\quad\text{for}~\phi\in\R,
\end{align}
where $f_{0}(\phi)=|\Omega|^{-1}\sum_{i=1}^{I}m_{i}z_{i}\exp(-z_{i}(\phi-\phi_{0}^{*}))$ (defined in \eqref{eq.3.015}).
Consequently, the CCPB equation \eqref{eq.3.001} can be rewritten as
\begin{align}
\label{eq.3.097}
-\epsilon\Delta\phi_{\epsilon}=f_{0}(\phi_{\epsilon})+\sqrt\epsilon(f_1(\phi_{\epsilon})+o_{\epsilon}(1))\quad\text{in}~\Omega.
\end{align}
Moreover, by \eqref{eq.3.088}--\eqref{eq.3.092} and \eqref{eq.3.096}--\eqref{eq.3.097}, we have
\begin{alignat}{2}
\label{eq.3.098}
\epsilon\Delta\varphi_{k,\epsilon}+c_\epsilon(x)\varphi_{k,\epsilon}&=g_\epsilon(x)\quad&&\text{for}~x\in\Omega_{k,\epsilon,\beta},\\
\label{eq.3.099}
\varphi_{k,\epsilon}+\gamma_k\sqrt\epsilon\partial_\nu\varphi_{k,\epsilon}&=0\quad&&\text{on}~\partial\Omega_k\quad\text{for}~k=0,1,\dots,K,\end{alignat}
where\begin{align}
\label{eq.3.100}
&\begin{aligned}
c_\epsilon(x)&=\frac{f_{0}(\phi_{\epsilon}(x))-f_{0}(u_{k}(\delta_k(x)/\sqrt\epsilon))}{\phi_{\epsilon}(x)-u_{k}(\delta_k(x)/\sqrt\epsilon)}\\&\quad+\sqrt\epsilon\frac{f_1(\phi_{\epsilon}(x))-f_1(u_{k}(\delta_k(x)/\sqrt\epsilon))+o_{\epsilon}(1)}{\phi_{\epsilon}(x)-u_{k}(\delta_k(x)/\sqrt\epsilon)}\quad\text{if}~\phi_{\epsilon}(x)\neq u_{k}(\delta_k(x)/\sqrt\epsilon),
\end{aligned}\\
\label{eq.3.101}
&c_\epsilon(x)=f_{0}'(\phi_{\epsilon}(x))+\sqrt\epsilon(f_1'(\phi_{\epsilon}(x))+o_{\epsilon}(1))\quad\text{if}~\phi_{\epsilon}(x)=u_{k}(\delta_k(x)/\sqrt\epsilon),
\end{align}
and
\begin{align}
\label{eq.3.102}
&g_\epsilon(x)=(d-1)H(p)u_{k}'(\delta_k(x)/\sqrt\epsilon)-f_1(u_{k}(\delta_k(x)/\sqrt\epsilon))+o_{\epsilon}(1).
\end{align}
Here the nonzero function $f_1$, defined in \eqref{eq.3.004}, distinguishes \eqref{eq.3.100}--\eqref{eq.3.102} from \eqref{eq.2.56}--\eqref{eq.2.57}.

To get the asymptotic expansion of $\varphi_{k,\epsilon}$ in $\Omega_{k,\epsilon,\beta}$, we apply the methods of \Cref{lemma:2.8} and \Cref{proposition:2.9} (cf. \Cref{section:2.2}) to \eqref{eq.3.098}--\eqref{eq.3.099} and define
\begin{align}
\label{eq.3.103}
\W_{k,p,\epsilon}(z)=\varphi_{k,\epsilon}(\Psi_{p}(\sqrt{\epsilon}z))\quad\text{for}~z\in\overline B_{\epsilon^{(2\beta-1)/2}}^+~\text{and}~0<\epsilon<\epsilon^{*}
\end{align}
(cf.~\eqref{eq.2.62}) where $\Psi_p$ is given in \eqref{eq.2.01}.
Analogously to \eqref{eq.2.63}--\eqref{eq.2.66}, $\W_{k,p,\epsilon}$ satisfies
\begin{alignat*}{2}
\sum_{i,j=1}^{d}a_{ij}(z)\frac{\partial^2\W_{k,p,\epsilon}}{\partial z^{i}z^{j}}+\sum_{j=1}^{d}b_j(z)\frac{\partial \W_{k,p,\epsilon}}{\partial z^{j}}+c_\epsilon(z)\W_{k,p,\epsilon}&=g_{\epsilon}(z)\quad&&\text{in}~B_{\epsilon^{(2\beta-1)/2}}^+,\\
\W_{k,p,\epsilon}-\gamma_k\partial_{z^{d}}\W_{k,p,\epsilon}&=0\quad&&\text{on}~\overline{B}_{\epsilon^{(2\beta-1)/2}}^{+}\cap\partial\R_{+}^{d},
\end{alignat*}
for $0<\epsilon<\epsilon^{*}$, where $a_{ij}$ and $b_j$ are given in \eqref{eq.2.05}--\eqref{eq.2.06}, and
\begin{align*}
&\begin{aligned}
c_\epsilon(z)&=\frac{f_{0}(u_{k}(z^{d})+\sqrt\epsilon\W_{k,p,\epsilon}(z))-f_{0}(u_{k}(z^{d}))}{\sqrt\epsilon\W_{k,p,\epsilon}(z)}\\&\quad+\frac{f_1(u_{k}(z^{d})+\sqrt\epsilon\W_{k,p,\epsilon}(z))-f_1(u_{k}(z^{d}))+o_{\epsilon}(1)}{\W_{k,p,\epsilon}(z)}\quad\text{if}~\W_{k,p,\epsilon}(z)\neq0,
\end{aligned}\\
&c_\epsilon(z)=f_{0}'(u_{k}(z^{d}))+\sqrt\epsilon(f_1'(u_{k}(z^{d}))+o_{\epsilon}(1))\quad\text{if}~\W_{k,p,\epsilon}(z)=0
\end{align*}
and
\begin{align*}
g_\epsilon(z)=(d-1)H(p)u_{k}'(z^{d})-f_1(u_{k}(z^{d}))+o_{\epsilon}(1).
\end{align*}
As in \Cref{lemma:2.8}, we use the uniform boundedness of $\varphi_{k,\epsilon}$ (cf. \Cref{claim:5}) to prove
\begin{lemma}
\label{lemma:3.14}
For any sequence $\{\epsilon_{n}\}_{n=1}^\infty$ of positive numbers with $\dd\lim_{n\to\infty}\epsilon_{n}=0$, $\alpha\in(0,1)$, and $p\in\partial\Omega_k$ ($k\in\{0,1,\dots,K\}$), there exists a subsequence $\{\epsilon_{nn}\}_{n=1}^\infty$ such that $\dd\lim_{n\to\infty}\|\W_{k,p,\epsilon_{nn}}-\W_{k,p}\|_{\C^{2,\alpha}(\overline B_m^+)}=0$ for $m\in\N$, where $\W_{k,p}\in\C^{2,\alpha}_{\text{loc}}(\overline\R_{+}^{d})$ satisfies
\begin{alignat}{2}
\label{eq.3.104}
\Delta\W_{k,p}+f_{0}'(u_{k})\W_{k,p}&=(d-1)H(p)u_{k}'-f_1(u_{k})\quad&&\text{in}~\R_{+}^{d},\\
\label{eq.3.105}
\W_{k,p}-\gamma_k\partial_{z^{d}}\W_{k,p}&=0\quad&&\text{on}~\partial\R_{+}^{d}.
\end{alignat}
Moreover, there exists $M>0$ such that $\W_{k,p}$ satisfies the following estimate
\begin{align}
\label{eq.3.106}
|\W_{k,p}(z)-(d-1)H(p)v_{k}(z^{d})-w_{k}(z^{d})|\leq M\exp(-C_{3}z^{d}/8)
\end{align}
for $z=(z',z^{d})\in\overline\R_{+}^{d}$ and $z^{d}\geq2(d-1)/C_{3}$, where $v_{k}$ and $w_{k}$ are solutions to \eqref{eq.1.22}--\eqref{eq.1.24} and \eqref{eq.1.25}--\eqref{eq.1.27}, respectively.
In addition, the constant $C_{3}$ is given in \eqref{eq.3.007}.
\end{lemma}

From \Cref{lemma:3.14}, the solution $\W_{k,p}$ to \eqref{eq.3.104}--\eqref{eq.3.105} may, a priori, depend on the sequence $\{\epsilon_{n}\}_{n=1}^\infty$.
To establish the indepedence from the sequence, we employ the moving plane arguments, as in \Cref{proposition:2.9} to prove that $\W_{k,p}(z)=\W_{k,p}(z^{d})=(d-1)H(p)v_{k}(z^{d})+w_{k}(z^{d})$ for $z=(z',z^{d})\in\R_{+}^{d}$, where $v_{k}$ and $w_{k}$ are the unique solutions to \eqref{eq.1.22}--\eqref{eq.1.24} and \eqref{eq.1.25}--\eqref{eq.1.27}. Here $v_{k}$ and $w_{k}$ are independent of the choice of sequence $\{\epsilon_{n}\}_{n=1}^\infty$ and point $p\in\partial\Omega_k$.
Consequently, we can improve the results of \Cref{lemma:3.14} as
\begin{align}
\label{eq.3.107}
\lim_{\epsilon\to0^+}\|\W_{k,p,\epsilon}-(d-1)H(p)v_{k}-w_{k}\|_{\C^{2,\alpha}(\overline B_m^+)}=0\quad\text{for}~m\in\N~\text{and}~\alpha\in(0,1),
\end{align}
and
\begin{align}
\label{eq.3.108}
\lim_{\epsilon\to0^+}\W_{k,p,\epsilon}(z)=(d-1)H(p)v_{k}(z^{d})+w_{k}(z^{d})\quad\text{for}~z=(z',z^{d})\in\overline\R_{+}^{d},
\end{align}
where $\W_{k,p,\epsilon}$ is defined in \eqref{eq.3.103}.
The details are stated as follows.
\begin{proposition}
\label{proposition:3.15}
For $p\in\partial\Omega_k$ and $k\in\{0,1,\dots,K\}$, the solution $\W_{k,p}$ to \eqref{eq.3.104}--\eqref{eq.3.105} satisfies
\begin{enumerate}
\item[(a)] $\W_{k,p}$ depends only on the variable $z^{d}$, i.e., $\W_{k,p}(z)=\W_{k,p}(z^{d})$ for $z=(z',z^{d})\in\overline\R_{+}^{d}$.
\item[(b)] $\W_{k,p}(z^{d})=(d-1)H(p)v_{k}(z^{d})+w_{k}(z^{d})$ for $z^{d}\in[0,\infty)$, where $v_{k}$ and $w_{k}$ are the unique solutions to \eqref{eq.1.22}--\eqref{eq.1.24} and \eqref{eq.1.25}--\eqref{eq.1.27}, respectively.
\end{enumerate}
\end{proposition}
\begin{proof}
Following \Cref{proposition:2.9}, we replace $f$ with $f_{0}$ and apply the moving plane arguments to obtain (a).
By (a), \eqref{eq.1.24}, \eqref{eq.1.27}, and \eqref{eq.3.106}, $\W_{k,p}=\W_{k,p}(z^{d})$ satisfies
\begin{align*}
&\W_{k,p}''+f_{0}'(u_{k})\W_{k,p}=(d-1)H(p)u_{k}'-f_1(u_{k})\quad\text{in}~(0,\infty),\\
&\W_{k,p}(0)-\gamma_k\W_{k,p}'(0)=0,\\
&\lim_{z^{d}\to\infty}\W_{k,p}(z^{d})=Q.
\end{align*}
By the uniqueness of solutions to \eqref{eq.1.22}--\eqref{eq.1.24} and \eqref{eq.1.25}--\eqref{eq.1.27} (cf. \eqref{eq.C.04} and~\eqref{eq.D.09}), we can conclude that $\W_{k,p}\equiv(d-1)H(p)v_{k}+w_{k}$, which gives (b).
This completes the proof of \Cref{proposition:3.15}.
\end{proof}

We may use \Cref{proposition:3.15} to prove \eqref{eq.3.107}, \eqref{eq.3.108}, and $\W_{k,p}(z)=(d-1)H(p)v_{k}(z^{d})+w_{k}(z^{d})$ for $z=(z',z^{d})\in\overline{\R}_{+}^{d}$.
Let $T>0$, $k\in\{0,1,\dots,K\}$, and $p\in\partial\Omega_k$.
Analogously to \eqref{eq.2.72}--\eqref{eq.2.73}, we use \eqref{eq.2.01}, \eqref{eq.3.103}, \eqref{eq.3.107}--\eqref{eq.3.108}, and $\nabla\delta_k(p-t\sqrt{\epsilon}\nu_{p})=-\nu_{p}$ to get
\begin{align}
\label{eq.3.109}
&\varphi_{k,\epsilon}(p-t\sqrt{\epsilon}\nu_{p})=(d-1)H(p)v_{k}(t)+w_{k}(t)+o_{\epsilon}(1),\\
\label{eq.3.110}
&\nabla\varphi_{k,p}(p-t\sqrt{\epsilon}\nu_{p})=-\frac{1}{\sqrt\epsilon}\{[(d-1)H(p)v_{k}'(t)+w_{k}'(t)]\nu_{p}+o_{\epsilon}(1)]
\end{align}
for $0\leq t\leq T$ as $\epsilon\to0^+$.
Combining \eqref{eq.3.088}, \eqref{eq.3.103} and \eqref{eq.3.109}--\eqref{eq.3.110}, we arrive at
\begin{align*}
&\phi_{\epsilon}(p-t\sqrt{\epsilon}\nu_{p})=u_{k}(t)+\sqrt\epsilon\left((d-1)H(p)v_{k}(t)+w_{k}(t)+o_{\epsilon}(1)\right),\\
&\nabla\phi_{\epsilon}(p-t\sqrt{\epsilon}\nu_{p})=-\left(\frac{1}{\sqrt\epsilon}u_{k}'(t)+(d-1)H(p)v_{k}'(t)+w_{k}'(t)\right)\nu_{p}+o_{\epsilon}(1)
\end{align*}
for $0\leq t\leq T$ as $\epsilon\to0^+$, which implies \eqref{eq.1.17}--\eqref{eq.1.18}.
Therefore, we complete the proof of \Cref{theorem:2}(a).

\subsection{Proof of \texorpdfstring{\Cref{corollary:2}}{Corollary 2}}
\label{section:3.5}

We present the proof of \Cref{corollary:2}, following the structure of \Cref{corollary:1}. We first prove (a). Integrating \eqref{eq.1.35} over $\overline\Omega_{k,T,\epsilon}$ and applying the coarea formula (cf. \cite{2015evans,2002lin}) with \eqref{eq.2.75}, we get
\begin{align*}\int_{\overline\Omega_{k,T,\epsilon}}\!f_{\epsilon}(\phi_{\epsilon}(x))\,\mathrm{d}x&=\sqrt\epsilon\int_{\partial\Omega_k}\!\int_0^T\!f_{0}(u_{k}(t))\,\mathrm{d}t\,\mathrm{d}S_p\\&\quad+\epsilon\int_{\partial\Omega_k}\!\int_0^T\!\{(d-1)H(p)[f_{0}'(u_{k}(t))v_{k}(t)-tf_{0}(u_{k}(t))]+f_{0}'(u_{k}(t))w_{k}(t)+f_1(u_{k}(t))\}\,\mathrm{d}t\,\mathrm{d}S_p\\&\quad+\epsilon o_{\epsilon}(1).\end{align*}
Together with \eqref{eq.1.19}, \eqref{eq.1.22}, and \eqref{eq.1.25}, we obtain
\begin{align*}
\int_{\overline\Omega_{k,T,\epsilon}}\!f_{\epsilon}(\phi_{\epsilon}(x))\,\mathrm{d}x&=\sqrt\epsilon|\partial\Omega_k|\int_0^T\!(-u_{k}''(t))\,\mathrm{d}t\\&\quad+\epsilon\left[(d-1)\left(\int_{\partial\Omega_k}\!H(p)\,\mathrm{d}S_p\right)\left(\int_0^T\![u_{k}'(t)-v_{k}''(t)+tu_{k}''(t)]\,\mathrm{d}t\right)+|\partial\Omega_k|\int_0^T\!(-w_{k}''(t))\,\mathrm{d}t\right]\\&\quad+\epsilon o_{\epsilon}(1)\\&=
\sqrt\epsilon|\partial\Omega_k|(u_{k}'(0)-u_{k}'(T))\\&\quad+\epsilon\left[(d-1)\left(\int_{\partial\Omega_k}\!H(p)\,\mathrm{d}S_p\right)(Tu_{k}'(T)+v_{k}'(0)-v_{k}'(T))+|\partial\Omega_k|(w_{k}'(0)-w_{k}'(T))\right]\\&\quad+\epsilon o_{\epsilon}(1),
\end{align*}
which gives (a).

We now prove (b) as follows.
As in \eqref{eq.2.76}, we integrate \eqref{eq.1.02} over $\overline\Omega_{k,T,\epsilon,\beta}$ and apply the divergence theorem to obtain
\begin{align}
\label{eq.3.111}
\begin{aligned}
\int_{\overline\Omega_{k,T,\epsilon,\beta}}\!f_{\epsilon}(\phi_{\epsilon}(x))\,\mathrm{d}x&=-\epsilon\int_{\overline\Omega_{k,T,\epsilon,\beta}}\!\Delta\phi_{\epsilon}(x)\,\mathrm{d}x\\&=-\epsilon\int_{\partial\Omega_{k,T,\epsilon,\beta}}\partial_{\nu_x}\phi_{\epsilon}(x)\,\mathrm{d}S_x\\&=-\int_{\partial\Omega_{k}}\!(\epsilon\partial_{\nu_{p}}\phi_{\epsilon}(p-T\sqrt{\epsilon}\nu_{p}))\mathcal{J}(T,p)\,\mathrm{d}S_{p}\\&\quad+\int_{\partial\Omega_{k}}\!(\epsilon\partial_{\nu_{p}}\phi_{\epsilon}(p-\epsilon^\beta\nu_{p}))\mathcal{J}(\epsilon^{(2\beta-1)/2},p)\,\mathrm{d}S_{p},
\end{aligned}
\end{align}
where $\nu_x$ is the unit outer normal at $x\in\partial\Omega_{k,T,\epsilon,\beta}$ with respect to $\Omega_{k,T,\epsilon,\beta}$, and $\mathcal{J}(T,p)$ and $\mathcal{J}(\epsilon^{(2\beta-1)/2},p)$ are given in \eqref{eq.2.77}.
Then by \eqref{eq.1.18}, we have
\begin{align}
\label{eq.3.112}
\begin{aligned}
&\quad-\int_{\partial\Omega_{k}}\!(\epsilon\partial_{\nu_{p}}\phi_{\epsilon}(p-T\sqrt{\epsilon}\nu_{p}))\mathcal{J}(T,p)\,\mathrm dS_{p}\\&=-\int_{\partial\Omega_{k}}\!(\epsilon\partial_{\nu_{p}}\phi_{\epsilon}(p-T\sqrt{\epsilon}\nu_{p}))[1-T\sqrt\epsilon(d-1)H(p)+\sqrt\epsilon o_{\epsilon}(1)]\,\mathrm{d}S_{p}\\&=\sqrt{\epsilon}\int_{\partial\Omega_{k}}\{u_{k}'(T)+\sqrt{\epsilon}(d-1)H(p)v_{k}'(T)+w_{k}'(T)+o_{\epsilon}(1)]\}[1-T\sqrt{\epsilon}(d-1)H(p)+\sqrt{\epsilon}o_{\epsilon}(1)]\,\mathrm dS_{p}\\&=
\sqrt\epsilon|\partial\Omega_k|u_{k}'(T)+\epsilon\left[(d-1)\left(\int_{\partial\Omega_k}\!H(p)\,\mathrm{d}S_p\right)(-Tu_{k}'(T)+v_{k}'(T))+|\partial\Omega_k|w_{k}'(T)\right]+\epsilon o_{\epsilon}(1).
\end{aligned}
\end{align}
On the other hand, by \Cref{theorem:2}(b,ii), we have
\begin{align}
\label{eq.3.113}
\left|\int_{\partial\Omega_{k}}\!(\epsilon\partial_{\nu_{p}}\phi_{\epsilon}(p-\epsilon^\beta\nu_{p}))\mathcal{J}(\epsilon^{(2\beta-1)/2},p)\,\mathrm{d}S_{p}\right|\leq\sqrt{\epsilon}|\partial\Omega_{k}|M'\exp\left(-M\epsilon^{(2\beta-1)/2}\right).
\end{align}
Combining \eqref{eq.3.111}--\eqref{eq.3.113} with $0<\beta<1/2$, we get (b).
Following the argument as in (b), one may use \eqref{eq.3.113} to prove (c).
Therefore, we complete the proof of \Cref{corollary:2}.

As aforementioned in \Cref{remark:6}, to complete the proof of \eqref{eq.3.076}, it is necessary to show that $Q$ is uniquely determined by
\begin{align}
\label{eq.3.114}
Q=\frac{\dd\sum_{k=0}^{K}\frac{|\partial\Omega_{k}|\hat{F}_{1}(u_{k}(0))+(d-1)\left(\int_{\partial\Omega_{k}}\!H(p)\,\mathrm{d}S_{p}\right)\left(\int_{0}^{\infty}\!u_{k}'^2(s)\,\mathrm{d}s\right)}{u_{k}'(0)+\gamma_{k}f_{0}(u_{k}(0))}}{\dd\sum_{k=0}^{K}\frac{|\partial\Omega_{k}|f_{0}(u_{k}(0))}{u_{k}'(0)+\gamma_{k}f_{0}(u_{k}(0))}}.
\end{align}
Using \eqref{eq.3.027}, \eqref{eq.3.075}, and \Cref{corollary:2}, we find
\[\epsilon\left[(d-1)\sum_{k=0}^{K}\left(\int_{\partial\Omega_{k}}\!H(p)\,\mathrm{d}S_{p}\right)v_{k}'(0)+\sum_{k=0}^{K}|\partial\Omega_{k}|w_{k}'(0)\right]=\epsilon o_{\epsilon}(1),\]
which implies
\begin{equation}
\label{eq.3.115}
(d-1)\sum_{k=0}^{K}\left(\int_{\partial\Omega_{k}}\!H(p)\,\mathrm{d}S_{p}\right)v_{k}'(0)+\sum_{k=0}^{K}|\partial\Omega_{k}|w_{k}'(0)=0.
\end{equation}
By \eqref{eq.1.23}, \eqref{eq.1.26}, \eqref{eq.C.05} (in \Cref{appendix:C}), and \eqref{eq.D.10} (in \Cref{appendix:D}), a straightforward calculation gives \eqref{eq.3.114}.
Therefore, we complete the proof of \eqref{eq.3.076} in \Cref{remark:6}.

\section*{Final Remark}

When the Robin boundary condition \eqref{eq.1.04} is replaced by the Dirichlet boundary condition $\phi_{\epsilon}=\phi_{bd,k}$ on $\partial\Omega_k$ for $k=0,1,\dots,K$ (i.e., $\gamma_k=0$ for $k=0,1,\dots,K$), the conclusions of \Cref{theorem:1,theorem:2} hold true, particularly the asymptotic expansions \eqref{eq.1.05}--\eqref{eq.1.06} and \eqref{eq.1.17}--\eqref{eq.1.18}, remain valid.
In \eqref{eq.1.05}--\eqref{eq.1.06} and \eqref{eq.1.17}--\eqref{eq.1.18}, the solutions $u_{k}$, $v_{k}$ and $w_{k}$ to ODEs satisfy the Dirichlet boundary conditions $u_{k}(0)=\phi_{bd,k}$ and $v_{k}(0)=w_{k}(0)=0$, which can be directly regarded as the Robin boundary conditions \eqref{eq.1.08}, \eqref{eq.1.11}, \eqref{eq.1.20}, \eqref{eq.1.23} and \eqref{eq.1.26} with $\gamma_k=0$.
By applying the maximum principle, the $L^q$-theory and Schauder estimate, the proofs of \Cref{theorem:1,theorem:2} remain valid without other modifications for the Dirichlet boundary condition except \Cref{proposition:3.06}.
When $\gamma_k=0$ for all $k=0,1,\dots,K$, equations \eqref{eq.3.026}--\eqref{eq.3.027} are replaced by
\[\sum_{k=0}^K|\partial\Omega_k|u_{k}'(0)=\sum_{k=0}^K|\partial\Omega_k|\sgn(\phi_{0}^{*}-\phi_{bd,k})\sqrt{\frac{2}{|\Omega|}{\sum_{i=1}^{I}m_{i}[\exp(-z_{i}(\phi_{bd,k}-\phi_{0}^{*}))-1]}}=0,\]
which determines the value of $\phi_{0}^{*}$ (given in \eqref{eq.3.014}).
The properties of ODEs remains true whenever $\gamma_{k}\geq0$ (see \Cref{appendix:B,appendix:C,appendix:D}).

\section*{Conclusion}

In this work, we present a rigorous analysis of boundary layer solutions to Poisson--Boltzmann (PB) type equations in general smooth domains.
We extend the classical boundary layer theory to multiply connected domains via the moving plane arguments, and address the challenge posed by the nonlocal nonlinearity in charge-conserving PB equations using novel integral estimates.
In the near-boundary region (where the distance to the boundary is of order $\sqrt\epsilon\mathcal{O}(1)$), we construct asymptotic expansions based on solutions to associated ordinary differential equations, capturing the effect of mean curvature through second-order approximations.
In the far-field region, we establish exponential decay estimates, which reveal the localized nature of boundary layer phenomena.
Our analysis quantifies how domain geometry, particularly mean curvature, influences electrostatic characteristics such as electric potential, electric field, total ionic charge and total ionic charge density.
These results highlight the explicit dependence of electrostatic screening on geometric features and provide analytical tools for applications in electrochemistry (e.g., porous electrodes), biophysics (e.g. ion channels), and microfluidics.

\appendix

\makeatletter\gdef\thesection{\@Alph\c@section}\makeatother
\counterwithout{figure}{section} 
\section{Exponential-type estimate of radial solution}
\label{appendix:A}

Here we study the radial solutions $\Phi_\epsilon=\Phi_\epsilon(r)$ to equations $\eqref{eq.1.01}$ and \eqref{eq.3.001} in the ball $B_R(0)$ with the Dirichlet boundary condition, respectively.
Equation \eqref{eq.1.01} with the Dirichlet boundary condition can be denoted as
\begin{align}
\label{eq.A.01}
&-\epsilon(r^{d-1}\Phi_{\epsilon}'(r))'=r^{d-1}f(\Phi_{\epsilon}(r))\quad\text{for}~r\in(0,R),\\
\label{eq.A.02}
&\Phi_{\epsilon}'(0)=0,\\
\label{eq.A.03}
&\Phi_{\epsilon}(R)=\Phi_{bd},
\end{align}
where $f\in\C^1(\R)$ is a strictly decreasing function satisfying (A1)--(A2), and $\Phi_{bd}$ is a constant independent of $\epsilon$.
Note that if $\Phi_{bd}=\phi^{*}$ ($\phi^{*}$ is the unique zero of $f$), then by the uniqueness of \eqref{eq.A.01}--\eqref{eq.A.03}, $\Phi_\epsilon\equiv\phi^{*}$ is a trivial solution.
Hence we assume $\Phi_{bd}\neq\phi^{*}$. By (A1)--(A2), it is easy to prove

\begin{proposition}
\label{proposition:A.1}
Assume that $f$ satisfies (A1)--(A2).
Let $\Phi_{\epsilon}$ be the solution to \eqref{eq.A.01}--\eqref{eq.A.03} for $\epsilon>0$.
\begin{enumerate}
\item[(a)] If $\Phi_{bd}>\phi^{*}$, then $\Phi_{\epsilon}(0)>\phi^{*}$ and $\Phi_{\epsilon}$ is strictly increasing;
\item[(b)] if $\Phi_{bd}<\phi^{*}$, then $\Phi_{\epsilon}(0)<\phi^{*}$ and $\Phi_{\epsilon}$ is strictly decreasing.
\end{enumerate}
\end{proposition}

\begin{proposition}\label{proposition:A.2}
Assume that $f$ satisfies (A1)--(A2).
Let $\Phi_\epsilon$ be the solution to \eqref{eq.A.01}--\eqref{eq.A.03} for $\epsilon>0$. Then
\begin{align}
\label{eq.A.04}
|\Phi_{\epsilon}(r)-\phi^{*}|\leq2|\Phi_{bd}-\phi^{*}|\exp\left(-\frac{m_f(R-r)}{8\sqrt\epsilon}\right)\end{align}
for $r\in[0,R]$ and $0<\epsilon<(m_fR)^2/[8(d-1)^2]$, where $m_f=m_f([\min\{\phi^{*},\Phi_{bd}\},\max\{\phi^{*},\Phi_{bd}\}])$ is given in (A1).
\end{proposition}
\begin{proof}
Let $\overline\Phi_{\epsilon}:=\Phi_{\epsilon}-\phi^{*}$ on $[0,R]$ for $\epsilon>0$.
Then from \eqref{eq.A.01}--\eqref{eq.A.03}, $\overline{\Phi}_{\epsilon}$ satisfies
\begin{align}
\label{eq.A.05}
&-\epsilon(r^{d-1}\overline{\Phi}_{\epsilon}'(r))'=r^{d-1}f(\phi^{*}+\overline{\Phi}_{\epsilon}(r))\quad\text{for}~r\in(0,R),\\
\label{eq.A.06}
&\overline{\Phi}_{\epsilon}'(0)=0,\\
\label{eq.A.07}
&\overline{\Phi}_{\epsilon}(R)=\Phi_{bd}-\phi^{*}.
\end{align}
By (A1)--(A2), we have
\begin{align}
\label{eq.A.08}
\begin{aligned}
f(\phi^{*}+\overline\Phi_\epsilon(r))\overline\Phi_\epsilon(r)&=\left(\int_0^1\!\frac{\mathrm{d}}{\mathrm{d}s}f(\phi^{*}+s\overline\Phi_\epsilon(r))\,\mathrm{d}s\right)\overline\Phi_\epsilon(r)\\&=\left(\int_0^1\!f'(\phi^{*}+s\overline\Phi_\epsilon(r))\,\mathrm{d}s\right)\overline\Phi_\epsilon^2(r)\leq-m_f^2\overline\Phi_\epsilon^2(r)
\end{aligned}
\end{align}
for $r\in(0,R)$.
Here we have used the fact that $\overline\Phi_\epsilon=\Phi_{\epsilon}-\phi^{*}$ is uniformly bounded (cf. \Cref{proposition:2.1} also works for the Dirichlet boundary condition \eqref{eq.A.03}).
By \eqref{eq.A.05} and \eqref{eq.A.08}, we obtain
\begin{equation}
\label{eq.A.09}
\begin{aligned}
\epsilon(\overline{\Phi}_{\epsilon}^2(r))''&=2\epsilon\overline{\Phi}_{\epsilon}'^2(r)+2\epsilon\overline{\Phi}_{\epsilon}''(r)\overline{\Phi}_{\epsilon}(r)\\
&=2\epsilon\overline{\Phi}_{\epsilon}'^2(r)-\frac{2\epsilon(d-1)}{r}\overline{\Phi}_{\epsilon}'(r)\overline{\Phi}_{\epsilon}(r)-2f(\phi^{*}+\overline{\Phi}_{\epsilon}(r))\overline{\Phi}_{\epsilon}(r)\\
&\geq2\epsilon\overline{\Phi}_{\epsilon}'^2(r)-\frac{2\epsilon(d-1)}{r}\overline{\Phi}_{\epsilon}'(r)\overline{\Phi}_{\epsilon}(r)+2m_f^2\overline{\Phi}_{\epsilon}^2(r)\quad\text{for}~r\in(0,R).
\end{aligned}
\end{equation}
On the other hand, by Young's inequality, we have
\begin{align}
\label{eq.A.10}
\begin{aligned}\frac{2\epsilon(d-1)}{r}\overline{\Phi}_{\epsilon}'(r)\overline{\Phi}_{\epsilon}(r)
&\leq2\epsilon\overline{\Phi}_{\epsilon}'^2(r)+\frac{\epsilon(d-1)^2}{2r^2}\overline{\Phi}_{\epsilon}^{2}(r)
\\&\leq2\epsilon\overline{\Phi}_{\epsilon}'^2(r)+\frac{8\epsilon(d-1)^2}{R^2}\overline{\Phi}_{\epsilon}^{2}(r)\quad\text{for}~r\in[R/4,R).
\end{aligned}
\end{align}
Combining \eqref{eq.A.09} and \eqref{eq.A.10}, we get
\begin{align}
\label{eq.A.11}
\epsilon(\overline{\Phi}_{\epsilon}^2(r))''\geq\left(2m_f^2-\frac{8\epsilon(d-1)^2}{R^2}\right)\overline{\Phi}_{\epsilon}^{2}(r)\geq m_f^2\overline{\Phi}_{\epsilon}^{2}(r)\quad\text{for}~r\in[R/4,R).
\end{align}
Here we have used that $0<\epsilon<(m_fR)^2/[8(d-1)^2]$.
Applying the comparison theorem to \eqref{eq.A.11}, it yields
\begin{align}
\label{eq.A.12}
\overline{\Phi}_{\epsilon}^{2}(r)\leq\max\{\overline{\Phi}_{\epsilon}^{2}(R/4),\overline{\Phi}_{\epsilon}^{2}(R)\}\left[\exp\left(-\frac{m_f(R-r)}{\sqrt{\epsilon}}\right)+\exp\left(-\frac{m_f(r-R/4)}{\sqrt{\epsilon}}\right)\right]
\end{align}
for $r\in[R/4,R]$.
By \Cref{proposition:A.1}, $(\overline{\Phi}_{\epsilon}^2(r))'=2\overline{\Phi}_{\epsilon}'(r)\overline{\Phi}_{\epsilon}(r)>0$ on $(0,R)$, and $\overline{\Phi}^2$ is strictly increasing on $[0,R]$.
Consequently, we have $\overline{\Phi}_{\epsilon}^2(R/4)<\overline\Phi_{\epsilon}^2(R)=(\Phi_{bd}-\phi^{*})^2$.
From \eqref{eq.A.12}, we get
\[|\overline{\Phi}_{\epsilon}(r)|\leq|\Phi_{bd}-\phi^{*}|\left[\exp\left(-\frac{m_f(R-r)}{2\sqrt\epsilon}\right)+\exp\left(-\frac{m_fR}{8\sqrt\epsilon}\right)\right]\quad\text{for}~r\in[3R/4,R],\]
which implies
\begin{align}
\label{eq.A.13}
|\overline{\Phi}_{\epsilon}(r)|\leq2|\Phi_{bd}-\phi^{*}|\exp\left(-\frac{m_f(R-r)}{2\sqrt{\epsilon}}\right)\quad\text{for}~r\in[3R/4,R].
\end{align}
Since $\overline{\Phi}_{\epsilon}^2$ is strictly increasing on $[0,R]$, $|\overline{\Phi}|$ is also strictly increasing on $[0,R]$.
Then by \eqref{eq.A.13}, we have
\begin{align}
\label{eq.A.14}
|\overline{\Phi}_{\epsilon}(r)|\leq|\overline{\Phi}_{\epsilon}(3R/4)|\leq2|\Phi_{bd}-\phi^{*}|\exp\left(-\frac{m_fR}{8\sqrt\epsilon}\right)\leq2|\Phi_{bd}-\phi^{*}|\exp\left(-\frac{m_f(R-r)}{8\sqrt\epsilon}\right)
\end{align}
for $r\in[0,3R/4]$.
Combining \eqref{eq.A.13} and \eqref{eq.A.14}, we arrive at \eqref{eq.A.04} and complete the proof of \Cref{proposition:A.2}.
\end{proof}

Now we consider the radial solution $\Phi_\epsilon=\Phi_\epsilon(r)$ to equation \eqref{eq.3.001} with the Dirichlet boundary condition, which satisfies
\begin{align}
\label{eq.A.15}
&-\epsilon(r^{d-1}\Phi_\epsilon'(r))'=r^{d-1}f_{\epsilon}(\Phi_\epsilon(r))\quad\text{for}~r\in(0,R),\\
\label{eq.A.16}
&\Phi_\epsilon'(0)=0,\\
\label{eq.A.17}
&\Phi_\epsilon(R)=\Phi_{bd},
\end{align}
where $f_{\epsilon}$ has a unique zero $\phi_{\epsilon}^{*}\in\R$ and satisfies $f_{\epsilon}'(\phi)\leq-M^2<0$ for $\phi\in\R$. Note that $f_{\epsilon}$ and its zero $\phi_{\epsilon}^{*}$ may depend on $\epsilon$ but $M$ is independent of $\epsilon$.

\newpage

Analogous to \Cref{proposition:A.1,proposition:A.2} (with $\phi^{*}$ replaced by $\phi_{\epsilon}^{*}$), we can prove
\begin{proposition}\label{proposition:A.3}
Assume that $f_{\epsilon}$ has a unique zero $\phi_{\epsilon}^{*}\in\R$ and satisfies $f_{\epsilon}'(\phi)\leq-M^2<0$ for $\phi\in\R$ and $\epsilon>0$, where $M$ is independent of $\epsilon$.
Let $\Phi_\epsilon$ be the solution to \eqref{eq.A.15}--\eqref{eq.A.17} for $\epsilon>0$.
\begin{enumerate}
\item[(a)] If $\Phi_{bd}>\phi_{\epsilon}^{*}$ for $\epsilon>0$, then $\Phi_\epsilon(0)>\phi_{\epsilon}^{*}$ and $\Phi_\epsilon$ is strictly increasing for $\epsilon>0$;
\item[(b)] if $\Phi_{bd}<\phi_{\epsilon}^{*}$ for $\epsilon>0$, then $\Phi_\epsilon(0)<\phi_{\epsilon}^{*}$ and $\Phi_\epsilon$ is strictly decreasing for $\epsilon>0$.
\end{enumerate}
\end{proposition}
\begin{proposition}\label{proposition:A.4}
Assume that $f_{\epsilon}$ has a unique zero $\phi_{\epsilon}^{*}\in\R$ and satisfies $f_{\epsilon}'(\phi)\leq-M^2<0$ for $\phi\in\R$ and $\epsilon>0$, where $M$ is independent of $\epsilon$.
Let $\Phi_\epsilon$ be the solution to \eqref{eq.A.15}--\eqref{eq.A.17} for $\epsilon>0$. Then
\begin{align}
\label{eq.A.18}
|\phi_{\epsilon}(r)-\phi_{\epsilon}^{*}|\leq2|\Phi_{bd}-\phi_{\epsilon}^{*}|\exp\left(-\frac{M(R-r)}{8\sqrt\epsilon}\right)\end{align}
for $r\in[0,R]$ and $0<\epsilon<(MR)^2/[8(d-1)^2]$.
\end{proposition}

\section{Properties of the solution to \texorpdfstring{\eqref{eq.1.07}}{(1.7)}--\texorpdfstring{\eqref{eq.1.09}}{(1.9)} and \texorpdfstring{\eqref{eq.1.19}}{(1.19)}--\texorpdfstring{\eqref{eq.1.21}}{(1.21)}}
\label{appendix:B}

In this appendix, we establish the existence, uniqueness, qualitative properties, and asymptotic behaviors of solutions to \eqref{eq.1.07}--\eqref{eq.1.09} and \eqref{eq.1.19}--\eqref{eq.1.21}.
For the sake of simplicity, equations \eqref{eq.1.07}--\eqref{eq.1.09} and \eqref{eq.1.19}--\eqref{eq.1.21} can be represented as follows.
\begin{align}
\label{eq.B.01}
&U''+f(U)=0\quad\text{in}~(0,\infty),\\
\label{eq.B.02}
&U(0)-\gamma\,U'(0)=\Phi_{bd},\\
\label{eq.B.03}
&\lim_{t\to\infty}U(t)=\phi^{*},
\end{align}
where $\gamma\geq0$ and $\Phi_{bd}\neq\phi^{*}$ are fixed constants. (When $\Phi_{bd}=\phi^{*}$, the solution $U\equiv\phi^{*}$ is trivial.)
Here the function $f\in\C^\infty(\R)$ satisfies assumptions (A1) and (A2).
It is clear that the existence and uniqueness of solution to equations \eqref{eq.1.07}--\eqref{eq.1.09} and \eqref{eq.1.19}--\eqref{eq.1.21} for $k=0,1,\dots,K$ follow from that of solution to equations \eqref{eq.B.01}--\eqref{eq.B.03}.
We may approach \eqref{eq.B.01}--\eqref{eq.B.03} by the following initial-value problem
\begin{align}
\label{eq.B.04}
&\tilde U''+f(\tilde U)=0\quad\text{in}~(0,\infty),\\
\label{eq.B.05}
&\tilde U(0)=U_0,\\
\label{eq.B.06}
&\tilde U'(0)=U_0',
\end{align}
where $U_0$ (between $\phi^{*}$ and $\Phi_{bd}$) and $U_0'$ are constants to be determined such that $\tilde U\equiv U$ on $[0,\infty)$.
The existence and uniqueness of \eqref{eq.B.04}--\eqref{eq.B.06} follow from the standard ODE theory.
From \eqref{eq.B.02} and \eqref{eq.B.05}--\eqref{eq.B.06}, $U_0$ and $U_0'$ must satisfy
\begin{align}
\label{eq.B.07}
U_0-\gamma U_0'=\Phi_{bd}.
\end{align}
Multiplying \eqref{eq.B.04} by $\tilde U'$ and then integrating it over $[0,t]$, we can use \eqref{eq.B.05}--\eqref{eq.B.06} to get
\begin{align}
\label{eq.B.08}
\tilde U'^2(t)+2F(\tilde U(t))=U_0'^2+2F(U_0)\quad\text{for}~t\geq0,
\end{align}
where $F(\phi)=\int_{\phi^{*}}^{\phi}\!f(s)\,\mathrm{d}s$ for $\phi\in\R$.
Note that $F'(\phi^{*})=f(\phi^{*})=0$ and $F''(\phi)=f'(\phi)<0$ for $\phi\in\R$ (by (A1)--(A2)) so
\begin{align}
\label{eq.B.09}
0=F(\phi^{*})>F(\phi)\quad\text{for}~\phi\neq\phi^{*}.
\end{align}
Suppose that $\dd\lim_{t\to\infty}\tilde U'(t)=0$.
Then we use \eqref{eq.B.03} and \eqref{eq.B.07}--\eqref{eq.B.08} to find that $U_0$ (between $\phi^{*}$ and $\Phi_{bd}$) must satisfy
\begin{align}
\label{eq.B.10}
&\Phi_{bd}-U_0=\sgn(\Phi_{bd}-\phi^{*})\gamma\sqrt{-2F(U_0)}.
\end{align}
To show that there exists a unique $U_0$, between $\phi^{*}$ and $\Phi_{bd}$, satisfying \eqref{eq.B.10}, we define a function $g:\R\to\R$ by
\[g(s)=\Phi_{bd}-s-\sgn(\Phi_{bd}-\phi^{*})\gamma\sqrt{-2F(s)}\quad\text{for}~s\in\R,\]
which is well-defined because of \eqref{eq.B.09}.
Since $g(\phi^{*})g(\Phi_{bd})<0$ and $g'(s)<0$ for $s$ between $\phi^{*}$ and $\Phi_{bd}$, function $g$ has a unique zero $U_0$ between $\phi^{*}$ and $\Phi_{bd}$, which means that equation \eqref{eq.B.10} has a unique solution.
Here one may use the facts that (i) if $\Phi_{bd}>\phi^{*}$, then $f(s)<0$ for $s\in(\phi^{*},\Phi_{bd}]$ and (ii) if $\Phi_{bd}<\phi^{*}$, then $f(s)>0$ for $s\in[\Phi_{bd},\phi^{*})$ due to (A1)--(A2).
Therefore, the initial-value problem \eqref{eq.B.04}--\eqref{eq.B.06} can be reformulated as
\begin{align}
\label{eq.B.11}
&\tilde U''+f(\tilde U)=0\quad\text{in}~(0,\infty),\\
\label{eq.B.12}
&\tilde U(0)=U_0,\\
\label{eq.B.13}
&\tilde U'(0)=\sgn(\phi^{*}-\Phi_{bd})\sqrt{-2F(U_0)},
\end{align}
where $U_0$ (between $\phi^{*}$ and $\Phi_{bd}$) is uniquely determined  by \eqref{eq.B.10}.

\begin{figure}[!htb]\centering\includegraphics[scale=0.7]{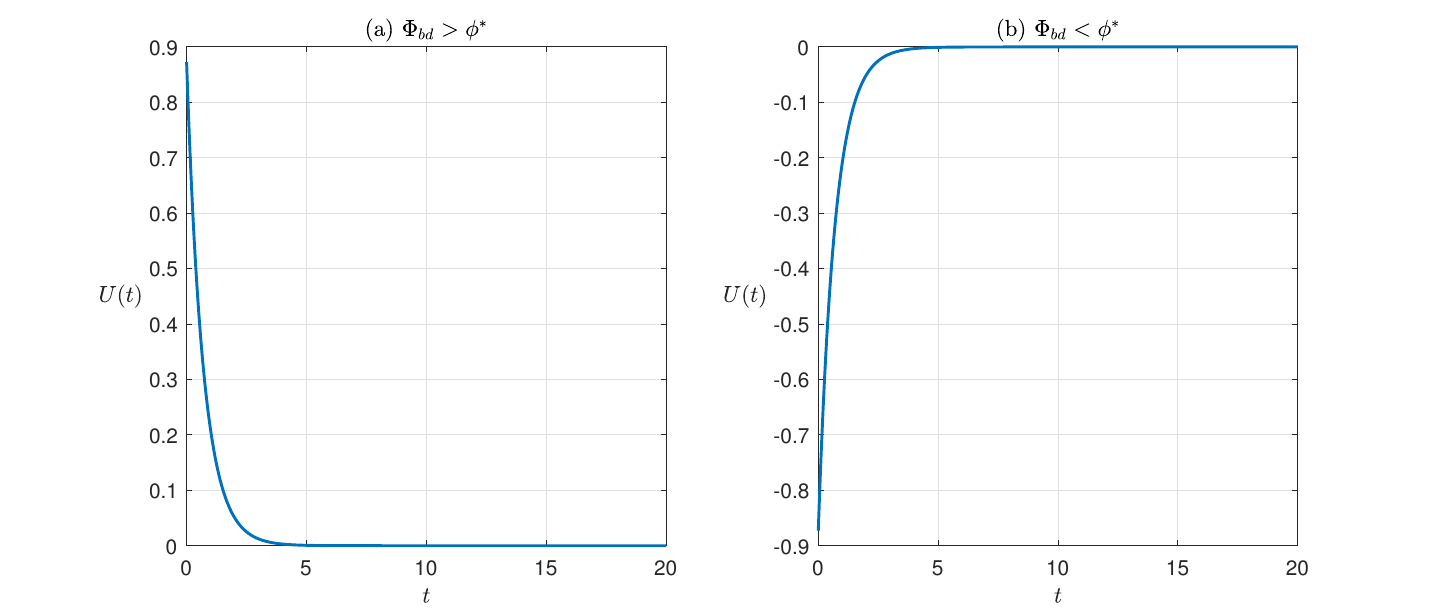}
\caption{We sketch the numerical profile for the solution $U$ to \eqref{eq.B.01}--\eqref{eq.B.03} for the case (a) $\Phi_{bd}=1$ and the case (b) $\Phi_{bd}=-1$, which is consistent with \Cref{proposition:B.1}.
Here $f(\phi)=\exp(\phi)-\exp(-\phi)$ and $\gamma=0.1$.}
\label{fig:U}\end{figure}

\newpage

To prove $\tilde U\equiv U$ on $[0,\infty)$, we need the following proposition.
\begin{proposition}
\label{proposition:B.1}
Let $\tilde U$ be the solution to \eqref{eq.B.11}--\eqref{eq.B.13}.
\begin{enumerate}
\item[(a)] If $\Phi_{bd}>\phi^{*}$, then $\Phi_{bd}\geq\tilde U(t)>\phi^{*}$ and $\tilde U'(t)<0$ for $t\geq0$;
\item[(b)] if $\Phi_{bd}<\phi^{*}$, then $\Phi_{bd}\leq\tilde U(t)<\phi^{*}$ and $\tilde U'(t)>0$ for $t\geq0$.
\end{enumerate}
\end{proposition}
\begin{proof}
We firstly claim that $\tilde U'(t)\neq0$ for $t\geq0$.
Suppose by contradiction that there exists $t_1\in(0,\infty)$ such that $\tilde U'(t_1)=0$.
Multiplying \eqref{eq.B.11} by $\tilde U$ and then integrating it over $[0,t]$, we obtain
\begin{align}
\label{eq.B.14}
\tilde U'^2(t)=\tilde U'^2(0)+2(F(\tilde U(0))-F(\tilde U(t)))=-2F(\tilde U(t))\quad\text{for}~t\geq0.
\end{align}
Here we have used \eqref{eq.B.12}--\eqref{eq.B.13}.
By \eqref{eq.B.14} and $\tilde U'(t_1)=0$, we get $\tilde U(t_1)=\phi^{*}$, which implies $\tilde U\equiv\phi^{*}$ on $[t_1,\infty)$.
Due to the unique continuation, $U\equiv\phi^{*}$ on $[0,\infty)$, which contradicts to \eqref{eq.B.12}.
Hence $\tilde U'(t)\neq0$ for $t\geq0$.
By \eqref{eq.B.09} and \eqref{eq.B.14}, it is clear that $\tilde U(t)\neq\phi^{*}$ for $t\geq0$ because $\tilde U'(t)\neq0$ for $t\geq0$.

Now we prove (a) and (b).
For the case of (a), we assume that $\Phi_{bd}>\phi^{*}$.
From \eqref{eq.B.10} and \eqref{eq.B.12}--\eqref{eq.B.13}, we have $\Phi_{bd}\geq\tilde U(0)>\phi^{*}$ and $\tilde U'(0)<0$.
Since $\tilde U(t)\neq\phi^{*}$ and $\tilde U'(t)\neq0$ for $t\geq0$, which implies $\Phi_{bd}\geq\tilde U(t)>\phi^{*}$ and $\tilde U'(t)<0$ for $t\geq0$.
Thus, the proof of (a) is complete.
The proof of (b) is similar to (a).
Therefore, we complete the proof of \Cref{proposition:B.1}.\end{proof}

Now we are in the position to show that $\tilde U$ solves equations \eqref{eq.B.01}--\eqref{eq.B.03} and the uniqueness of \eqref{eq.B.01}--\eqref{eq.B.03}, which implies $\tilde U\equiv U$ in $[0,\infty)$.
By \eqref{eq.B.10} and \eqref{eq.B.12}--\eqref{eq.B.13}, we have
\[\tilde U(0)-\gamma\tilde U'(0)=U_0\gamma\sgn(\phi^{*}-\Phi_{bd})\sqrt{-2F(U_0)}=\Phi_{bd},\]
which means $\tilde U$ satisfies \eqref{eq.B.02}.
By \Cref{proposition:B.1}, $\tilde U$ is bounded and strictly monotonic on $[0,\infty)$, which implies that there exists a constant $C$ such that\[\lim_{t\to\infty}\tilde U(t)=C.\]
Suppose by contradiction that $C\neq\phi^{*}$.
By \eqref{eq.B.09} and \eqref{eq.B.12}--\eqref{eq.B.14}, we get
\begin{align}
\label{eq.B.15}
\lim_{t\to\infty}\tilde U'^2(t)=\tilde U'^2(0)+2(F(U_0)-F(C))=-2F(C)>0.
\end{align}
Since $\tilde U'(t)\neq0$ for $t\geq0$, we can use \eqref{eq.B.15} to get $\dd\lim_{t\to\infty}|\tilde U(t)|=\infty$, which leads to a contradiction.
Thus, $\dd\lim_{t\to\infty}\tilde U(t)=\phi^{*}$, i.e., $\tilde U$ satisfies \eqref{eq.B.03}.
Hence $\tilde U$ solves equations \eqref{eq.B.01}--\eqref{eq.B.03}.
To prove $\tilde U\equiv U$ on $[0,\infty)$, it suffices to show the uniqueness of solution to \eqref{eq.B.01}--\eqref{eq.B.03}.
Suppose that there exist two solutions $U_1$ and $U_2$ to equations \eqref{eq.B.01}--\eqref{eq.B.03}.
Let $\overline{U}=U_1-U_2$.
Then $\overline{U}$ satisfies
\begin{align}
\label{eq.B.16}
&\overline{U}''+c(t)\overline{U}=0\quad\text{in}~(0,\infty),\\
\label{eq.B.17}
&\overline{U}(0)-\gamma\,\overline{U}'(0)=0,\\
\label{eq.B.18}
&\lim_{t\to\infty}\overline{U}(t)=0,
\end{align}
where
\[c(t)=\begin{cases}\dd\frac{f(U_1(t))-f(U_2(t))}{U_1(t)-U_2(t)}&\text{if}~U_1(t)\neq U_2(t);\\
f'(U_1(t))&\text{if}~U_1(t)=U_2(t).
\end{cases}\]
Note that $c(t)<0$ for $t\geq0$ by (A1).
Thus, by the standard maximum principle, $\overline{U}\equiv0$ on $[0,\infty)$, and we obtain the uniqueness of solution to \eqref{eq.B.01}--\eqref{eq.B.03}.
Therefore, $\tilde U\equiv U$ on $[0,\infty)$ and hence $U$ satisfies \Cref{proposition:B.1} and \eqref{eq.B.14}.

The following result will be used in \Cref{appendix:C} to prove the existence of solution to \eqref{eq.1.10}--\eqref{eq.1.12} and \eqref{eq.1.22}--\eqref{eq.1.24}.
\begin{proposition}\label{proposition:B.2}
Let $U$ be the solution to \eqref{eq.B.01}--\eqref{eq.B.03}. Then we have
\begin{itemize}
\item[(a)] $|U'(t)|\leq|U'(0)|\exp(-m_ft)$ for $t\geq0$, where $m_f=m_f([\min\{\phi^{*},\Phi_{bd}\},\max\{\phi^{*},\Phi_{bd}\}])$.
\item[(b)] $\int_0^\infty\!U'^2(t)\,\mathrm{d}t=\sgn(\phi^{*}-\Phi_{bd})\int_{U_0}^{\phi^{*}}\!\sqrt{-2F(s)}\,\mathrm{d}s<\infty$.
\item[(c)] $\int_0^t\!(U'(s))^{-2}\,\mathrm{d}s\geq t^2(\int_0^\infty\!U'^2(s)\,\mathrm{d}s)^{-1}$ for $t\geq0$.
\item[(d)] $\lim_{t\to\infty}[f(U(t))/U'(t)]=-\sqrt{-f'(\phi^{*})}$.
\item[(e)] $\int_0^\infty\!(U'(s_1))^{-2}\int_{s_1}^\infty\!U'^2(s_2)\,\mathrm{d}s_2\,\mathrm{d}s_1=\infty$.
\end{itemize}
\end{proposition}
\noindent The proof of \Cref{proposition:B.2} is easy by \Cref{proposition:B.1} and calculus so we omit it here.

\section{Properties of the solution to \texorpdfstring{\eqref{eq.1.10}}{(1.10)}--\texorpdfstring{\eqref{eq.1.12}}{(1.12)} and \texorpdfstring{\eqref{eq.1.22}}{(1.22)}--\texorpdfstring{\eqref{eq.1.24}}{(1.24)}}
\label{appendix:C}

In this appendix, we establish the existence, uniqueness, and qualitative properties for solutions to \eqref{eq.1.10}--\eqref{eq.1.12} and \eqref{eq.1.22}--\eqref{eq.1.24}.
For the sake of simplicity, equations \eqref{eq.1.10}--\eqref{eq.1.12} and \eqref{eq.1.22}--\eqref{eq.1.24} can be represented as follows.
\begin{align}
\label{eq.C.01}
&V''+f'(U)V=U'\quad\text{in}~(0,\infty),\\
\label{eq.C.02}
&V(0)-\gamma\,V'(0)=0,\\
\label{eq.C.03}
&\lim_{t\to\infty}V(t)=0,
\end{align}
where $U$ is the solution to equations \eqref{eq.B.01}--\eqref{eq.B.03}, $f\in\C^\infty(\R)$ and $\gamma\geq0$ are as given in \eqref{eq.B.01}--\eqref{eq.B.02}.
By standard ODE theory, the solution $V$ to equations \eqref{eq.C.01}--\eqref{eq.C.03} can be represented as
\begin{align}
\label{eq.C.04}
V(t)=\frac{V_0}{U'(0)}U'(t)-U'(t)\int_0^t\frac{1}{U'^2(s_1)}\int_{s_1}^{\infty}\!U'^2(s_2)\,\mathrm{d}s_2\,\mathrm{d}s_1\quad\text{for}~t\geq0,
\end{align}
where $V_0$ is given by
\begin{align}
\label{eq.C.05}
V_0=-\frac{\gamma}{U'(0)+\gamma f(U_0)}\int_0^\infty\!U'^2(t)\,\mathrm{d}t.
\end{align}
Note that $V_0$ is well-defined due to \Cref{proposition:B.2}(b) and $U'(0)+\gamma f(U_0)\neq0$.
Clearly, \eqref{eq.C.04} with \eqref{eq.C.05} satisfies \eqref{eq.C.01}--\eqref{eq.C.02}.
It suffices to verify that \eqref{eq.C.04} satisfies \eqref{eq.C.03}.
By \eqref{eq.B.01}, \eqref{eq.B.03}, (A1)--(A2), \Cref{proposition:B.2} and L'H\^opital's rule, we obtain $\lim\limits_{t\to\infty}V(t)=\lim\limits_{t\to\infty}(U'(t)/f'(U(t)))=0$.
Thus, \eqref{eq.C.04} defines a solution to \eqref{eq.C.01}--\eqref{eq.C.03}.

Now we prove the uniqueness of the solution to \eqref{eq.C.01}--\eqref{eq.C.03}.
Suppose that there exist two solutions $V$ and $\tilde V$ to \eqref{eq.C.01}--\eqref{eq.C.03}.
Let $\overline V=V-\tilde V$.
Then $\overline V$ satisfies
\begin{align}
\label{eq.C.06}
&\overline V''+f'(U)\overline V=0\quad\text{in}~(0,\infty),\\
\label{eq.C.07}
&\overline V(0)-\gamma\overline V'(0)=0,\\
\label{eq.C.08}
&\lim_{t\to\infty}\overline V(t)=0.
\end{align}
Hence by (A1) and the standard maximum principle, $\overline V\equiv0$ on $[0,\infty)$, and the solution to \eqref{eq.C.01}--\eqref{eq.C.03} must be given by \eqref{eq.C.04}.

\begin{proposition}
\label{proposition:C.1}
Let $V$ be the unique solution to \eqref{eq.C.01}--\eqref{eq.C.03}.
\begin{itemize}
\item[(a)] If $\Phi_{bd}>\phi^{*}$, then $V$ is nonnegative on $[0,\infty)$, and there exists a unique $t^{*}\in(0,\infty)$ such that $V$ is strictly increasing on $[0,t^{*})$ and strictly decreasing on $(t^{*},\infty)$. Moreover, if $\gamma>0$, then $V$ is positive on $[0,\infty)$.
\item[(b)] If $\Phi_{bd}<\phi^{*}$, then $V$ is nonpositive on $[0,\infty)$, and there exists a unique $t^{*}\in(0,\infty)$ such that $V$ is strictly decreasing on $[0,t^{*})$ and strictly increasing on $(t^{*},\infty)$. Moreover, if $\gamma>0$, then $V$ is negative on $[0,\infty)$.
\end{itemize}
\end{proposition}
\begin{proof}
We first note that
\begin{claim}
\label{claim:6}
$\dd\lim_{t\to\infty}V'(t)=0$.
\end{claim}
\begin{proof}[Proof of \Cref{claim:6}]
By \eqref{eq.B.01} and \eqref{eq.C.04}, we have
\[V'(t)=-\frac{V_0}{U'(0)}f(U(t))+f(U(t))\int_0^t\!\frac1{U'^2(s_1)}\int_{s_1}^{\infty}\!U'^2(s_2)\,\mathrm ds_2\,\mathrm ds_1-\frac1{U'(t)}\int_t^\infty\!U'^2(s)\,\mathrm ds.\]
By (A2), \eqref{eq.B.01} and \Cref{proposition:B.2}(d), we apply the L'H\^opital's rule to obtain
\begin{align*}
\lim_{t\to\infty}V'(t)&=\lim_{t\to\infty}\left[\frac{\dd\int_0^t\!\frac1{U'^2(s_1)}\int_{s_1}^{\infty}\!U'^2(s_2)\,\mathrm ds_2\,\mathrm ds_1}{[f(U(t)]^{-1}}-\frac{\dd\int_t^\infty\!U'^2(s)\,\mathrm ds}{U'(t)}\right]=\lim_{t\to\infty}\left[\frac{\dd\frac1{U'^2(t)}\int_t^\infty\!U'^2(s)\,\mathrm ds}{[f(U(t))]^{-2}f'(U(t))U'(t)}+\frac{U'^2(t)}{U''(t)}\right]\\
&=\frac1{f'(\phi^*)}\lim_{t\to\infty}\frac{\dd[f(U(t))]^2\int_t^\infty\!U'^2(s)\,\mathrm ds}{U'^3(t)}-\lim_{t\to\infty}\frac{U'^2(t)}{f(U(t))}=-\lim_{t\to\infty}\frac{\dd\int_t^\infty\!U'^2(s)\,\mathrm ds}{U'(t)}=\lim_{t\to\infty}\frac{U'^2(t)}{U''(t)}\\
&=-\lim_{t\to\infty}\frac{U'^2(t)}{f(U(t))}=0,
\end{align*}which completes the proof of \Cref{claim:6}.
\end{proof}
We now assume $\Phi_{bd}>\phi^*$ and prove (a).
By \Cref{proposition:B.1}(a), we have $U'(0)<0$ and $f(U_0)<0$. Along with \eqref{eq.C.05}, we get $V(0)=V_0>0$, and $V'(0)=-(U'(0)+\gamma f(U_0))^{-1}\int_0^\infty\!U'^2(t)\,\mathrm dt>0$.
Suppose by contradiction that there exists $t_0\in(0,\infty)$ such that $V(t_0)\leq0$.
Since $\dd\lim_{t\to\infty}V(t)=0$, we may assume that $V$ attains its minimum value at $t_0\in(0,\infty)$, which implies $V''(t_0)\geq0$.
Then by \eqref{eq.C.01}, we have $U'(t_0)=V''(t_0)+f'(U(t_0))V(t_0)\geq0$, which contradicts Proposition~\ref{proposition:B.1}(a).
This shows that $V$ is positive on $[0,\infty)$.
Since $V'(0)>0$ and $\dd\lim_{t\to\infty}V(t)=0$, there exists $t^*\in(0,\infty)$ such that $V$ attains its maximum value at $t^*$ and $V'(t^*)=0$.
Suppose by contradiction that there exists another maximum point $t_1$ such that $V'(t_1)=0$.
Without loss of generality, we may assume that $t_1<t^*$, and there exists $t_2\in(t_1,t^*)$ such that $V$ attains its local minimum value at $t_2$ with $V'(t_2)=0$ and $V'(t)\geq0$ on $(t_2,t^*)$.
Integrating \eqref{eq.C.01} over $[t,\infty)$ and using (a), we obtain
\begin{align}
\label{eq.C.09}
V'(t)=\frac{f(U(t))V(t)+\int_t^\infty\!U'^2(s)\,\mathrm ds}{-U'(t)}\quad\text{for}~t\geq0.
\end{align}
Clearly, $V'(t_2)=V'(t^*)=0$.
Let $G(t)=f(U(t))V(t)+\int_t^\infty\!U'^2(s)\,\mathrm ds$ for $t\in\R$. Then $G(t_2)=G(t^*)=0$.
Differentiating $G$ gives
\[G'(t)=f'(U(t))U'(t)V(t)+f(U(t))V'(t)-U'^2(t)\quad\text{for}~t\in\R.\]
By Proposition~\ref{proposition:B.1}(a) and (A1)--(A2), we find that $U'(t)<0$, $f(U(t))<0$ and $f'(U(t))<0$.
Since $V(t)>0$ and $V'(t)>0$ on $(t_2,t^*)$, we get $G'(t)<0$ on $(t_2,t^*)$, which contradcits with $g(t_2)=g(t^*)=0$.
Therefore, there exists a unique $t^*\in(0,\infty)$ such that $V$ attains its maximum value.
The proof of (a) is complete.

Next we assume $\Phi_{bd}<\phi^*$ and prove (b).
By \Cref{proposition:B.1}(b), we have $U'(0)>0$ and $f(U_0)>0$. Along with \eqref{eq.C.05}, we get $V(0)=V_0<0$, and $V'(0)=-(U'(0)+\gamma f(U_0))^{-1}\int_0^\infty\!U'^2(t)\,\mathrm dt<0$.
Suppose by contradiction that there exists $t_0\in(0,\infty)$ such that $V(t_0)\geq0$.
Since $\dd\lim_{t\to\infty}V(t)=0$, we may assume that $V$ attains its maximum value at $t_0\in(0,\infty)$, which implies $V''(t_0)\leq0$.
Then by \eqref{eq.C.01}, we have $U'(t_0)=V''(t_0)+f'(U(t_0))V(t_0)\leq0$, which contradicts to \Cref{proposition:B.1}(b).
This shows that $V$ is negative on $[0,\infty)$.
Since $V'(0)<0$ and $\dd\lim_{t\to\infty}V(t)=0$, there exists $t^*\in(0,\infty)$ such that $V$ attains its minimum value at $t^*$ and $V'(t^*)=0$.
Suppose by contradiction that there exists another minimum point $t_1$ such that $V'(t_1)=0$.
Then by a similar fashion of (b), we can use \eqref{eq.C.09} to get a contradiction because $U'(t)>0$, $f(U(t))>0$ and $f'(U(t))<0$ (cf. \Cref{proposition:B.1}(b) and (A1)--(A2)).
Hence there exists a unique $t^*\in(0,\infty)$ such that $V$ attains its minimum value, which implis (b).
Therefore, we complete the proof of \Cref{proposition:C.1}.
\end{proof}

\begin{figure}[!htb]\centering\includegraphics[scale=0.7]{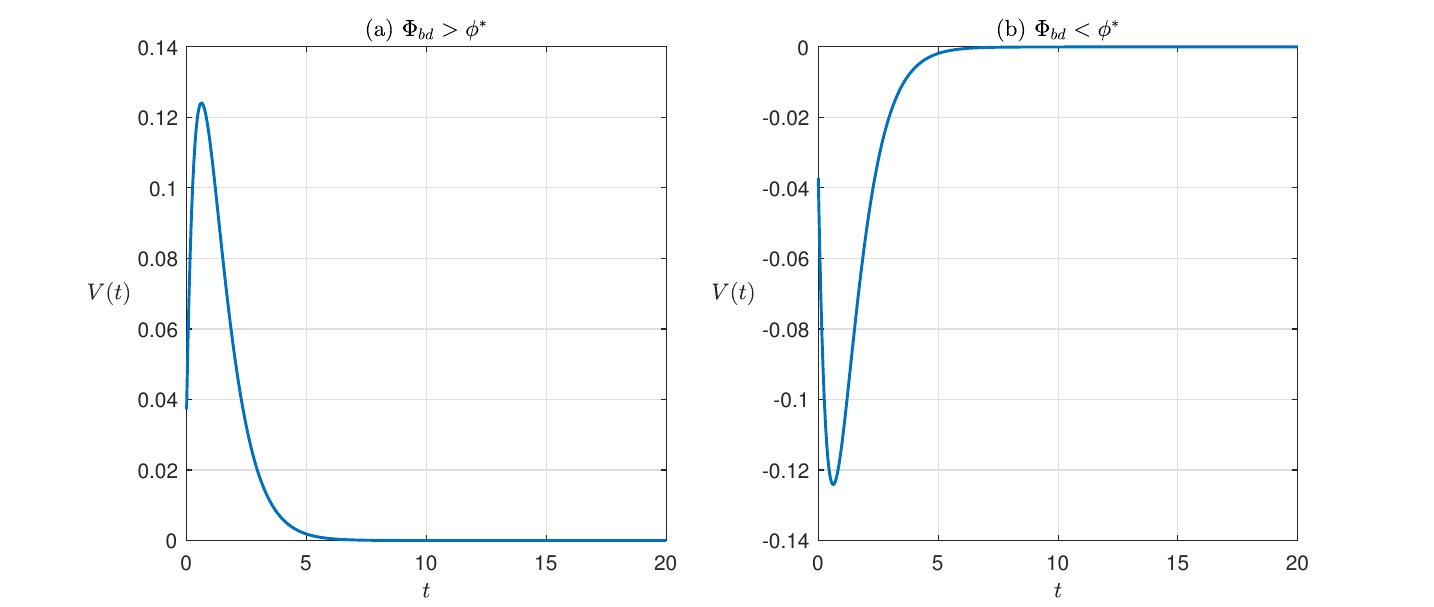}
\caption{We sketch the numerical profile for the solution $V$ to \eqref{eq.C.01}--\eqref{eq.C.03} for the case (a) $\Phi_{bd}=1$ and the case (b) $\Phi_{bd}=-1$, which is consistent with \Cref{proposition:C.1}.
Here $f(\phi)=\exp(\phi)-\exp(-\phi)$ and $\gamma=0.1$.}
\label{fig:V}\end{figure}

\begin{proposition}\label{prop:C.2}
Let $U$ and $V$ be the unique solutions to \eqref{eq.B.01}--\eqref{eq.B.03} and \eqref{eq.C.01}--\eqref{eq.C.03}, respectively.\\
If $\Phi_{bd}\neq\phi^{*}$, then function $g:=f(U)V+U'V'$ is negative on $[0,\infty)$.
\end{proposition}
\begin{proof}
By \eqref{eq.B.01}, \eqref{eq.C.01} and \Cref{proposition:B.1}, we have
\[g'(t)=f'(U(t))U'(t)V(t)+f(U(t))V'(t)+U''(t)V'(t)+U'(t)V''(t)=U'^2(t)>0\]for $t\geq0$ and $\Phi_{bd}\neq\phi^{*}$, which implies $g$ is strictly increasing on $[0,\infty)$. By (A2), \eqref{eq.B.03}, \eqref{eq.C.03} and \Cref{claim:6}, it is clear that $\dd\lim_{t\to\infty}g(t)=0$. Hence $g(t)<0$ for $t\geq0$, and the proof of \Cref{prop:C.2} is complete.
\end{proof}

\section{Properties of the solution to \texorpdfstring{\eqref{eq.1.25}}{(1.25)}--\texorpdfstring{\eqref{eq.1.27}}{(1.27)}}
\label{appendix:D}

In this appendix, we establish the existence, uniqueness, and qualitative properties for solutions to \eqref{eq.1.25}--\eqref{eq.1.27}.
\begin{align}
\label{eq.D.01}
&w_{k}''+f_{0}'(u_{k})w_{k}=-f_1(u_{k})\quad\text{in}~(0,\infty),\\
\label{eq.D.02}
&w_{k}(0)-\gamma_kw_{k}'(0)=0,\\
\label{eq.D.03}
&\lim_{t\to\infty}w_{k}(t)=Q,
\end{align}
where $u_{k}$ are the solutions to equations \eqref{eq.1.19}--\eqref{eq.1.21}, and $\gamma_k\geq0$ are given fixed constants for $k=0,1,\dots,K$.
Here $f_{0}$ and $f_1$ are smooth functions defined by
\begin{align}
\label{eq.D.04}
f_{0}(\phi)=\frac{1}{|\Omega|}\sum_{i=1}^{I}m_{i}z_{i}\exp(-z_{i}(\phi-\phi_{0}^{*}))\quad\text{and}\quad f_1(\phi)=-Qf_{0}'(\phi)+\hat f_1(\phi)\quad\text{for}~\phi\in\R,
\end{align}
where $Q\in\R$ (defined in \eqref{eq.1.30}), $m_{i}>0$ and
\begin{align}
\label{eq.D.05}
&\hat f_1(\phi)=\frac{1}{|\Omega|}\sum_{i=1}^{I}\hat{m}_iz_{i}\exp(-z_{i}(\phi-\phi_{0}^{*}))\quad\text{for}~\phi\in\R,\\
\label{eq.D.06}
&\hat m_{i}=\frac{m_{i}}{|\Omega|}\sum_{k=0}^K|\partial\Omega_k|\int_0^\infty\![1-\exp(-z_{i}(u_{k}(s)-\phi_{0}^{*}))]\,\mathrm{d}s\quad\text{for}~i=1,\dots,I,
\end{align}
In addition, $\phi_{0}^{*}$ and $u_{k}(0)$ satisfy (cf. \eqref{eq.3.026}--\eqref{eq.3.027})
\begin{align}
\label{eq.D.07}
&\phi_{bd,k}-u_{k}(0)=\sgn(\phi_{bd,k}-\phi_{0}^{*})\gamma_k\sqrt{\frac{2}{|\Omega|}\sum_{i=1}^{I}m_{i}[\exp(-z_{i}(u_{k}(0)-\phi_{0}^{*}))-1]}\quad\text{for}~k=0,1,\dots,K,\\
\label{eq.D.08}
&\sum_{k=0}^K|\partial\Omega_k|u'(0)=\sum_{k=0}^K|\partial\Omega_k|\frac{u_{k}(0)-\phi_{bd,k}}{\gamma_k}=0.
\end{align}
By standard ODE theory, the solution $w_{k}$ to equations \eqref{eq.D.01}--\eqref{eq.D.03} can be represented as
\begin{align}
\label{eq.D.09}
w_{k}(t)=\frac{w_{k}(0)}{u_{k}'(0)}u_{k}'(t)+u_{k}'(t)\int_0^t\!\frac{Qf_{0}(u_{k}(s))-\hat F_1(u_{k}(s))}{u_{k}'^2(s)}\,\mathrm{d}s\quad\text{for}~t\geq0,
\end{align}
where $w_{k}(0)$ and $\hat F_1$ are given by
\begin{align}
\label{eq.D.10}
&w_{k}(0)=\frac{\gamma_k(Qf_{0}(u_{k}(0))-\hat F_1(u_{k}(0)))}{u_{k}'(0)+\gamma_kf_{0}(u_{k}(0))},\\
\label{eq.D.11}
&\hat F_1(\phi)=\int_{\phi_{0}^{*}}^\phi\!\hat f_1(s)\,\mathrm{d}s=\frac{1}{|\Omega|}\sum_{i=1}^{I}\hat m_{i}[1-\exp(-z_{i}(\phi-\phi_{0}^{*}))]\quad\text{for}~\phi\in\R.
\end{align}
Note that \eqref{eq.D.10} is well-defined because $u_{k}'(0)+\gamma_kf_{0}(u_{k}(0))\neq0$.
Clearly, \eqref{eq.D.09} with \eqref{eq.D.10}--\eqref{eq.D.11} satisfies \eqref{eq.D.01}--\eqref{eq.D.02}.
To verify that \eqref{eq.D.09} satisfies \eqref{eq.D.03}, we need the following claim.
\begin{claim}
\label{claim:7}
Let $u_{k}$ be the solution to \eqref{eq.1.19}--\eqref{eq.1.21}, $\hat{f}_{1}$ and $\hat{F}_{1}$ defined in \eqref{eq.D.05} and \eqref{eq.D.11}, respectively, and $\phi_{0}^{*}$ be the unique zero of $f_{0}$.
Then we have
\begin{enumerate}
\item[(a)] $\hat{f}_{1}(\phi_{0}^{*})=\dfrac1{|\Omega|}\sum\limits_{i=1}^{I}\hat{m}_{i}z_{i}=0$;
\item[(b)] $\displaystyle\left|\int_{0}^{\infty}\!\frac{Qf_{0}(u_{k}(s))-\hat{F}_{1}(u_{k}(s))}{u_{k}'^2(s)}\,\mathrm{d}s\right|=\infty$ if $Q\neq0$.
\end{enumerate}
\end{claim}
\begin{proof}[Proof of \Cref{claim:7}]
By \eqref{eq.D.04} and $\sum_{i=1}^{I}m_{i}z_{i}=0$ (cf. \eqref{eq.1.03}), we have
\begin{align*}
\hat f_1(\phi_{0}^{*})&=\frac{1}{|\Omega|}\sum_{i=1}^{I}\hat m_{i}z_{i}=\frac{1}{|\Omega|}\sum_{i=1}^{I}m_{i}z_{i}\sum_{k=0}^K|\partial\Omega_k|\int_0^\infty\![1-\exp(-z_{i}(u_{k}(s)-\phi_{0}^{*}))]\,\mathrm{d}s\\
&=\frac{1}{|\Omega|}\sum_{k=0}^K|\partial\Omega_k|\int_0^\infty\!\left(\sum_{i=1}^{I}m_{i}z_{i}-\sum_{i=1}^{I}m_{i}z_{i}\exp(-z_{i}(u_{k}(s)-\phi_{0}^{*}))\right)\,\mathrm{d}s\\&=-\frac{1}{|\Omega|}\sum_{k=0}^K|\partial\Omega_k|\int_0^\infty\!f_{0}(u_{k}(s))\,\mathrm{d}s.
\end{align*}
Along with \eqref{eq.1.19}--\eqref{eq.1.20}, \eqref{eq.D.08} and \Cref{proposition:B.2}(a), we get
\[\hat f_1(\phi_{0}^{*})=\frac{1}{|\Omega|}\sum_{k=0}^K|\partial\Omega_k|\int_0^\infty\!u_{k}''(s)\,\mathrm{d}s=\frac{1}{|\Omega|}\sum_{k=0}^K|\partial\Omega_k|(-u_{k}'(0))=\frac{1}{|\Omega|}\sum_{k=0}^K|\partial\Omega_k|\frac{\phi_{bd,k}-u_{k}(0)}{\gamma_k}=0.\]
which shows (a).
To prove (b), we use \eqref{eq.1.19} and L'H\^opital's rule to get
\begin{align*}
\lim_{t\to\infty}\frac{Qf_{0}(u_{k}(t))-\hat F_1(u_{k}(t))}{u_{k}'^2(t)}&=\lim_{t\to\infty}\frac{Qf_{0}'(u_{k}(t))u_{k}'(t)-\hat f_1(u_{k}(t))u_{k}'(t)}{2u_{k}'(t)u_{k}''(t)}\\&=\lim_{t\to\infty}\frac{-Qf_{0}'(u_{k}(t))+\hat f_1(u_{k}(t))}{2f_{0}(u_{k}(t))}.
\end{align*}
By \eqref{eq.1.21}--\eqref{eq.1.28}, \eqref{eq.3.017}, \Cref{proposition:B.1}, we have $f_{0}(u_{k}(t))\neq0$ for $t\geq0$ and $\dd\lim_{t\to\infty}f_{0}(u_{k}(t))=0$.
By \eqref{eq.1.21}, \eqref{eq.3.017}, (a) and the condition $Q\neq0$, we get \[\lim_{t\to\infty}(-Qf_{0}'(u_{k}(t))+\hat f_1(u_{k}(t)))=-Qf_{0}'(\phi_{0}^{*})\neq0.\]
Hence $\left|\int_0^\infty\!(u_{k}'(s))^{-2}[Qf_{0}(u_{k}(s))-\hat F_1(u_{k}(s))]\,\mathrm{d}s\right|=\infty$, which gives (b) and completes the proof of \Cref{claim:7}.
\end{proof}

Now we show that \eqref{eq.D.09} satisfies \eqref{eq.D.03} and obtain the existence of \eqref{eq.D.01}--\eqref{eq.D.03}. For the case that $Q\neq0$, we can use \eqref{eq.B.01}, \Cref{claim:7} and L'H\^opital's rule to get
\begin{align}
\label{eq.D.12}
\begin{aligned}
\lim_{t\to\infty}w_{k}(t)&=\lim_{t\to\infty}\frac{\dd\frac{w_{k}(0)}{u_{k}'(0)}+\int_{0}^{t}\!\frac{Qf_{0}(u_{k}(s))-\hat F_1(u_{k}(s))}{u_{k}'^2(s)}\,\mathrm{d}s}{\dd\frac{1}{u_{k}'(t)}}\\&=\lim_{t\to\infty}\frac{\dd\frac{Qf_{0}(u_{k}(t))-\hat F_1(u_{k}(t))}{u_{k}'^2(t)}}{\dd-\frac{u_{k}''(t)}{u_{k}'^2(t)}}\\
&=\lim_{t\to\infty}\frac{Qf_{0}(u_{k}(t))-\hat F_1(u_{k}(t))}{f_{0}(u_{k}(t))}=Q+\lim_{t\to\infty}\frac{-\hat f_1(u_{k}(t))u_{k}'(t)}{f_{0}'(u_{k}(t))u_{k}'(t)}\\&=Q-\lim_{t\to\infty}\frac{\hat f_1(u_{k}(t))}{f_{0}'(u_{k}(t))}=Q-\frac{\hat f_1(\phi_{0}^{*})}{f_{0}'(\phi_{0}^{*})}=Q,
\end{aligned}
\end{align}
which gives \eqref{eq.D.03} for $Q\neq0$.
This shows the existence of a solution \eqref{eq.D.01}--\eqref{eq.D.03} for $Q\neq0$.
For the case of $Q=0$, we define $w_{k}^0$ and $w_{k}^\pm$ by \eqref{eq.D.09} for $Q=0$ and $Q=\pm1$, respectively.
Then due to the linearity of \eqref{eq.D.09} (with respect to $Q$), $w_{k}^0=(w_{k}^++w_{k}^-)/2$.
By \eqref{eq.D.12} with $Q=\pm1$, we find that $\lim\limits_{t\to\infty}w_{k}^0(t)=\lim\limits_{t\to\infty}(w_{k}^+(t)+w_{k}^-(t))/2=(1+(-1))/2=0$, which implies that $w_{k}^0$ satisfies \eqref{eq.D.03} for $Q=0$.
Therefore, we get the existence of the solution to \eqref{eq.D.01}--\eqref{eq.D.03} for $Q\in\R$.

It remains to prove the uniqueness of the solution to \eqref{eq.D.01}--\eqref{eq.D.03}.
Suppose that there exist two solutions $w_{k}$ and $\tilde w_{k}$ of \eqref{eq.D.01}--\eqref{eq.D.03}.
Let $\overline w_{k}=w_{k}-\tilde w_{k}$.
Then $\overline w_{k}$ satisfies
\begin{align*}
&\overline w_{k}''+f_{0}'(u_{k})\overline w_{k}=0\quad\text{in}~(0,\infty),\\
&\overline w_{k}(0)-\gamma_k\overline w_{k}'(0)=0,\\
&\lim_{t\to\infty}\overline w_{k}(t)=0.
\end{align*}
By the standard maximum principle, it is easy to obtain $\overline w_{k}\equiv0$ on $[0,\infty)$, which gives the uniqueness of solution to \eqref{eq.D.01}--\eqref{eq.D.03}.

\newpage

\bibliographystyle{elsarticle-num.bst}

\begin{thebibliography}{999}

\bibitem{1981abbena} {\sc E. Abbena, A. Gray and L. Vanhecke}, {\em Steiner's formula for the volume of a parallel hypersurface in a riemannian manifold}, Annali Sc. Norm. Sup. Pisa, 8 (1981), pp.~473--493, \url{https://www.numdam.org/item/ASNSP_1981_4_8_3_473_0}

\bibitem{1995andelman} {\sc D. Andelman}, {\em Electrostatic Properties of Membranes: The Poisson--Boltzmann Theory}. in Structure and Dynamics of Membranes, R. Lipowsky and E. Sackmann, eds., vol. 1, North-Holland, 1995, pp.~603--642, \url{https://doi.org/10.1016/S1383-8121(06)80005-9}

\bibitem{2009altman} {\sc M. D. Altman, J. P. Bardhan, J. K. White and B. Tidor}, {\em Accurate solution to multi-region continuum biomolecule electrostatic problems using the linearized Poisson–Boltzmann equation with curved boundary elements}, J. Comput. Chem., 30 (2009), pp.~132--153, \url{https://doi.org/10.1002/jcc.21027}

\bibitem{2004baker} {\sc N. A. Baker}, {\em Poisson--Boltzmann Methods for Biomolecular Electrostatics}. in Numerical Computer Methods, Part D, L. Brand and M. L. Johnson, eds., vol. 383, Academic Press, 2004, pp.~94--118, \url{https://doi.org/10.1016/S0076-6879(04)83005-2}

\bibitem{1997barcilon} {\sc V. Barcilon, D.-P. Chen, R. S. Eisenberg, and J. W. Jerome}, {\em Qualitative properties of steady-state Poisson--Nernst--Planck systems: perturbation and simulation study}, SIAM J. Appl. Math., 57 (1997), pp.~631--648, \url{https://doi.org/10.1137/S0036139995312149}

\bibitem{2005bazant} {\sc M. Z. Bazant, K. T. Chu, and B. J. Bayly}, {\em Current-Voltage relations for electrochemical thin films}, SIAM J. Appl. Math., 65 (2005), pp.~1463--1484, \url{https://doi.org/10.1137/040609938}

\bibitem{2011bazant} {\sc M. Z. Bazant, B. D. Storey, and A. A. Kornshev}, {\em Double Layer in Ionic Liquids: Overscreening versus Crowding}, Phys. Rev. Lett., 106 (2011), 046102, \url{https://doi.org/10.1103/PhysRevLett.106.046102}

\bibitem{2023blossey} {\sc R. Blossey}, {\em The Poisson--Boltzmann Equation: An Introduction}, Springer--Verlag, 2023, \url{https://doi.org/10.1007/978-3-031-24782-8}

\bibitem{1997borukhov} {\sc I. Borukhov and D. Andelman}, {\em Steric Effects in Electrolytes: A Modified Poisson--Boltzmann Equation}, Phys. Rev. Lett., 79 (1997), pp.~435--438, \url{https://doi.org/10.1103/PhysRevLett.79.435}

\bibitem{1994cecil} {\sc T. E. Cecil}, {\em Focal Points and Support Functions in Affine Differential Geometry}, Geometriae Dedicata, 50 (1994), pp.~291--300, \url{https://doi.org/10.1007/BF01267871}

\bibitem{2015clarke} {\sc B. M. N. Clarke and P. J. Stiles}, {\em Finite electric boundary-layer solutions of a generalized Poisson--Boltzmann equation}, Proc. R. Soc. A, 471 (2015), p.~24, \url{https://doi.org/10.1098/rspa.2015.0024}

\bibitem{2010das} {\sc S. Das and S. Chakraborty}, {\em Effect of Conductivity Variations within the Electric Double Layer on the Streaming Potential Estimation in Narrow Fluidic Confinements}, Langmuir, 26 (2010), pp.~11589--11596, \url{https://doi.org/10.1021/la1009237}

\bibitem{2008dong} {\sc F. Dong, B. Olsen and N. A. Baker}, {\em Computational Methods for Biomolecular Electrostatics}. in Biophysical Tools for Biologists, Volume One: In Vitro Techniques, J. J. Correia, and H. W. Detrich, III, eds., vol. 84, Academic Press, 2008, ch.~26, pp.~843--870, \url{https://doi.org/10.1016/S0091-679X(07)84026-X}

\bibitem{2022dourado} {\sc A. H. B. Dourado}, {\em Electric Double Layer: The Good, the Bad, and the Beauty}, Electrochem, 3 (2022), pp.~789--808, \url{https://doi.org/10.3390/electrochem3040052}

\bibitem{2021eakins} {\sc B. B. Eakins, S. D. Patel, A. P. Kalra, V. Rezania, K. Shankar, and J. A. Tuszynski}, {\em Modeling Microtubule Counterion Distributions and Conductivity Using the Poisson-Boltzmann Equation}, Front. Mol. Biosci., 8 (2021), 650757, \url{https://doi.org/10.3389/fmolb.2021.650757}

\bibitem{2023elisea-espinoza} {\sc J. J. Elisea-Espinoza, E. Gonza\'alez-Tovar, G. I. Guerrero-Garc\'ia}, {\em Theoretical description of the electrical double layer for a mixture of $n$ ionic species with arbitrary size and charge asymmetries. I. Spherical geometry}, J. Chem. Phys., 158 (2023), 224111, \url{https://doi.org/10.1063/5.0151140}

\bibitem{1999evans} {\sc D. F. Evans and H. Wennerstr\"om}, {\em The Colloidal Domain: Where Physics, Chemistry, Biology, and Technology Meet} (2nd ed.), VCH Publishers, 1999.

\bibitem{2015evans} {\sc L. C. Evans and R. F. Gariepy}, {\em Measure Theory and Fine Properties of Functions: Revised Edition}, CRC Press, 2015, \url{https://doi.org/10.1201/b18333}

\bibitem{2015fellner} {\sc K. Fellner and V. A. Kovtunenko}, {\em A singularly perturbed nonlinear Poisson--Boltzmann equation: uniform and super-asymptotic expansions}, Math. Methods Appl. Sci., 38 (2015), pp.~3575--3586, \url{https://doi.org/10.1002/mma.3593}

\bibitem{2002fogolari} {\sc F. Fogolari, A. Brigo and H. Molinari}, {\em The Poisson--Boltzmann equation for biomolecular electrostatics: a tool for structural biology}, J. Mol. Recognit., 15 (2002), pp.~377--392, \url{https://doi.org/10.1002/jmr.577}

\bibitem{2012fontelos} {\sc M. A. Fontelos and K. B. Gamboa}, {\em On the structure of double layers in Poisson--Boltzmann equation}, Discrete Cont. Dyn. Syst. -B., 17 (2012), pp.~1939--1967, \url{https://doi.org/10.3934/dcdsb.2012.17.1939}

\bibitem{1987friedman} {\sc A. Friedman and K. Tintarev}, {\em Boundary asymptotics for solutions of the Poisson--Boltzmann equation}, J. Differ. Equations, 69 (1987), pp.~15--38, \url{https://doi.org/10.1016/0022-0396(87)90100-8}

\bibitem{2015gebbie} {\sc M. A. Gebbie, H. A. Dobbs, M. Valtiner, and J. N. Israelachvili}, {\em Long-range electrostatic screening in ionic liquids}, PNSA, 112 (2015), pp.~7432--7437, \url{https://doi.org/10.1073/pnas.1508366112}

\bibitem{1977gilbarg} {\sc D. Gilbarg and N. Trudinger}, {\em Elliptic Partial Differential Equations of Second Order}, Springer--Verlag, 1977, \url{https://doi.org/10.1007/978-3-642-61798-0}

\bibitem{2004gray} {\sc A. Gray}, {\em Tubes} (2nd ed.), Springer Basel AG, 2004, \url{https://doi.org/10.1007/978-3-0348-7966-8}

\bibitem{2018gray} {\sc C. G. Gray and P. J. Stiles}, {\em Nonlinear electrostatics: the Poisson--Boltzmann equation}, Eur. J. Phys. {\bf 39} (2018), 053002, \url{https://doi.org/10.1088/1361-6404/aaca5a}

\bibitem{2017griffiths} {\sc D. J. Griffiths}, {\em Introduction to Electrodynamics} (4th ed.), Cambridge University Press, 2017, \url{https://doi.org/10.1017/9781108333511}

\bibitem{2007grochowski} {\sc P. Grochowski and J. Trylska}, {\em Continuum molecular electrostatics, salt effects, and counterion binding---A review of the Poisson--Boltzmann theory and its modifications}, Biopolymers, 89 (2007), pp.~93--113, \url{https://doi.org/10.1002/bip.20877}

\bibitem{2019kamysbayev} {\sc V. Kamysbayev, V. Srivastava, N. B. Ludwig, O. J. Borkiewicz, H. Zhang, J. Ilavsky, B. Lee, K. W. Chapman, S. Vaikuntanathan, and D. V. Talapin}, {\em Nanocrystals in Molten Salts and Ionic Liquids: Experimental Observation of Ionic Correlations Extending beyond the Debye Length}, ACS Nano, 13 (2019), pp.~5760--5770, \url{https://doi.org/10.1021/acsnano.9b01292}

\bibitem{2023khlyipin} {\sc A. Khlyupin, I. Nesterova, and K. Gerke}, {\em Molecular scale roughness effects on electric double layer structure in asymmetric ionic liquids}, Electrochimica Acta, 450 (2023), p.~142261, \url{https://doi.org/10.1016/j.electacta.2023.142261}

\bibitem{2003lamm} {\sc G. Lamm}, {\em The Poisson--Boltzmann equation}. in {\em Reviews in Computational Chemistry}, K. B. Lipkowitz, R. Larter, and T. R. Cundari, eds., vol. 19, John Wiley \& Sons, Ltd., 2003, ch.~4, pp.~147--365, \url{https://doi.org/10.1002/0471466638.ch4}

\bibitem{2014lee} {\sc C.-C. Lee}, {\em The charge conserving Poisson-Boltzmann equations: Existence, uniqueness, and maximum principle}, J. Math. Phys., 55 (2014), p.~051503, \url{https://doi.org/10.1063/1.4878492}

\bibitem{2011lee} {\sc C.-C. Lee, H. Lee, Y. Hyon, T.-C. Lin, and C. Liu}, {\em New Poisson–Boltzmann type equations: one-dimensional solutions}, Nonlinearity, 24 (2011), pp.~431--458, \url{https://doi.org/10.1088/0951-7715/24/2/004}

\bibitem{2016lee} {\sc C.-C. Lee, H. Lee, Y. Hyon, T.-C. Lin, and C. Liu}, {\em Boundary Layer Solution of Charge Conserving Poisson--Boltzmann Equations: One-Dimensional Case}, Comm. Math. Sci., 14 (2016), pp.~911--940, \url{https://doi.org/10.4310/CMS.2016.v14.n4.a2}

\bibitem{2009li} {\sc B. Li}, {\em Continuum electrostatics for ionic solutions with non-uniform ionic sizes}, Nonlinearity, 22 (2009), pp.~811--833, \url{https://doi.org/10.1088/0951-7715/22/4/007}

\bibitem{2013li} {\sc B. Li, P. Liu, Z. Xu, and S. Zhou}, {\em Ionic size effects: generalized Boltzmann distributions, counterion stratification and modified Debye length}, Nonlinearity, 26 (2013), pp.~2899--2922, \url{https://doi.org/10.1088/0951-7715/26/10/2899}

\bibitem{2002lin} {\sc F.-H. Lin and X. P. Yang}, {\em Geometric Measure Theory: An Introduction}, International Press of Boston, 2002.

\bibitem{2011lu} {\sc B. Lu and Y. C. Zhou}, {\em Poisson--Nernst--Planck Equations for Simulating Biomolecular Diffusion-Reaction Processes II: Size Effects on Ionic Distributions and Diffusion-Reaction Rates}, Biophys. J., 100 (2011), pp.~2475--2485, \url{https://doi.org/10.1016/j.bpj.2011.03.059}

\bibitem{2023lyu} {\sc J.-H. Lyu and T.-C. Lin}, {\em PB-steric equations: A general model of PB equations}, SIAM J. Applied Math., 83 (2023), pp.~1603--1622, \url{https://doi.org/10.1137/22M1516270}

\bibitem{2007mori} {\sc Y. Mori, J. W. Jerome, and C. S. Peskin}, {\em A three-dimensional model of cellular electrical activity}, Bull. Inst. Math. Acad. Sin., 2 (2007), pp.~367--390, \url{https://www.math.sinica.edu.tw/bulletins/20072/2007212.pdf}

\bibitem{2006olesen} {\sc L. H. Olesen, H. Bruus, and A. Ajdari}, {\em ac electrokinetic micropumps: The effect of geometrical confinement, Faradaic current injection, and nonlinear surface capacitance}, Phys. Rev. E, 73 (2006) 056313, \url{https://doi.org/10.1103/PhysRevE.73.056313}

\bibitem{2021petsev} {\sc D. N. Petsev, F. van Swol, and L. J. D. Frink}, {\em Molecular Theory of Electric Double Layers}, IOP Publishing, 2021, \url{https://doi.org/10.1088/978-0-7503-2276-8}

\bibitem{2010plouraboue} {\sc F. Plourabou\'e, and H.-C. Chang}, {\em Attraction between two similar particles in an electrolyte: effects of Stern layer absorption}, An Acad Bras Cienc, 82 (2010), pp.~95--108, \url{https://doi.org/10.1590/S0001-37652010000100009}

\bibitem{2006ryham} {\sc R. Ryham, C. Liu and Z.-Q. Wang}, {\em On electro-kinetic fluids: One dimensional configurations}, Discrete Cont. Dyn. Syst. -B., 6 (2006), pp.~357--371, \url{https://doi.org/10.3934/dcdsb.2006.6.357}

\bibitem{1989russel} {\sc W. B. Russel, D. A. Saville, and W. R. Schowalter}, {\em Colloidal Dispersions}, Cambridge Univ. Press, 1989, \url{https://doi.org/10.1017/CBO9780511608810}

\bibitem{2008struwe} {\sc M. Struwe}, {\em Variational Methods: Applications to Nonlinear Partial Differential Equations and Hamiltonian Systems}, Springer, 2008, \url{https://doi.org/10.1007/978-3-540-74013-1}

\bibitem{2014wan} {\sc L. Wan, S. Xu, M. Liao, C. Liu, and P. Sheng}, {\em Self-Consistent Approach to Global Charge Neutrality in Electrokinetics: A Surface Potential Trap Model}, Phys. Rev. X, 4 (2014), 011042, \url{https://doi.org/10.1103/PhysRevX.4.011042}

\bibitem{1995yang} {\sc Y. Yang, J. Walz, and P. Pintauro}, {\em Curvature effects on electric double-layer forces, Part 1. --- Comparisons with parallel plate geometry}, J. Chem. Soc. Farday Trans., 91 (1995), pp.~2827--2836, \url{https://doi.org/10.1039/FT9959102827}

\bibitem{1997yang} {\sc Y. Yang, J. Walz, and P. Pintauro}, {\em Curvature effects on electric double-layer forces, Part 2. --- Dependence of forces on cavity radius and relative permittivity}, J. Chem. Soc. Farday Trans., 93 (1997), pp.~603--611, \url{https://doi.org/10.1039/A607031K}
\end{thebibliography}

\end{document}